\documentclass[a4paper,10pt, oneside]{article}
\usepackage{bm,amsmath,amssymb,amsthm,mathrsfs,graphicx}
\usepackage{amsmath,amssymb,amsthm}
\usepackage{amssymb}
\usepackage{mathrsfs}
\usepackage[affil-it]{authblk}
\usepackage[inline]{enumitem}
\usepackage[font=small,labelfont=md,textfont=it]{caption}
\usepackage{floatrow,float}
\usepackage[titletoc, title]{appendix}
\usepackage[colorlinks,linkcolor=blue,citecolor=blue, hyperindex,pagebackref,bookmarks]{hyperref}
\usepackage{etoolbox}
\usepackage{longtable}
\usepackage{diagbox}
\usepackage{booktabs,makecell,multirow}
\usepackage[capitalise,nosort]{cleveref}
\usepackage{cases,color}

\numberwithin{equation}{section}

\newtheorem{myTheo}{Theorem}[section]
\newtheorem{myLem}{Lemma}[section]

\usepackage{indentfirst}
\usepackage{stmaryrd}
\usepackage{subfigure}
\usepackage[utf8]{inputenc}
\usepackage{pgf,tikz}
\usepackage{threeparttable}

\usepackage[colorlinks,linkcolor=blue,citecolor=blue]{hyperref}

\usetikzlibrary{arrows}
\newcommand{\tabcaption}{\def\@captype{table}\caption}

\newtheorem{rem}{Remark}[section]

\newtheorem{exam}{Example}[section]

\newcommand{\snmii}[1]
{
	\left\vert\kern-0.25ex
	\left\vert\kern-0.25ex
	\left\vert
	#1
	\right\vert\kern-0.25ex
	\right\vert\kern-0.25ex
	\right\vert
}

\numberwithin{equation}{section}

\date{}
\title{{\Large \bf }}

\begin{document}
	\title
	{
		\Large\bf Robust globally divergence-free HDG finite element method for steady thermally coupled incompressible MHD flow
		\thanks
		{
			This work was supported in part by National Natural Science Foundation
			of China  (12171340).
		}
	}
	\author
	{
		Min Zhang$^{1,2}$ \thanks{ 1. Institute of Mathematics, Henan Academy of Sciences, Zhengzhou 450046, China.
			2. College of Mathematies and Information Science, Henan Normal University, Xinxiang, 453007, China Email: 894858786@qq.com }, \
		Zimo Zhu \thanks{Chengdu aircraft design and Research Institute, Chengdu 610041, China.
			Email:
			309463165@qq.com }, \
		Qijia Zhai$^{1,2}$ \thanks{
			1. School of Mathematics, Sichuan University, Chengdu 610064, China.   2. Computer, Electrical and Mathematical Science and Engineering Division, King Abdullah University of Science and Technology, Thuwal 23955, Saudi Arabia.
			Email:zhaiqijia@stu.scu.edu.cn}, \
		Xiaoping Xie \thanks{Corresponding author. School of Mathematics, Sichuan University, Chengdu 610064, China.
			Email: xpxie@scu.edu.cn}
		
	}
	\date{}
	\maketitle
	
	\begin{abstract}
		This paper develops an hybridizable discontinuous Galerkin (HDG) finite element method of arbitrary order for the steady thermally coupled  incompressible Magnetohydrodynamics (MHD) flow. The HDG scheme uses piecewise polynomials of degrees $k(k\geq 1),k,k-1,k-1$ and $k$ respectively for the approximations of  the velocity, the magnetic field, the  pressure, the  magnetic pseudo-pressure, and the temperature in the interior of elements,
		 and uses piecewise
		polynomials of degree $k$ for their numerical traces on the interfaces of elements. The method is shown to yield globally divergence-free approximations of the velocity and magnetic fields. Existence and uniqueness results  for  the discrete scheme  are given and   optimal a priori error estimates are derived.  Numerical experiments are provided  to verify the  obtained theoretical results.

	\end{abstract}
	
	\section{Introduction}
	
	Magnetohydrodynamics (MHD) equations describe the  basic physics laws  of electrically conducting fluid flow interacting with magnetic fields,
	and are widely used in engineering areas; see, eg. several monographs
	\cite{GCL2006,D2001, Muller2001, Shercliff1965, Walker1980} and the  references therein.
	In this  paper we  consider the steady thermally coupled incompressible MHD  model,  which is  a coupled  system of    incompressible
	Navier-Stokes equations,    Maxwell equations and  a thermal equation.

	Let $\Omega \subset \mathbb{R}^{d}(d=2,3)$ be a polygonal/polyhedral domain. The considered steady thermally coupled incompressible MHD  model reads as follows:
	find the velocity vector $ \mathbf{u}=(u_1,u_2,...,u_d)^\mathrm{T}$,  the pressure
	$p$,  the magnetic field
	$\mathbf{B}=(B_1,B_2,...,B_d)^\mathrm{T}$, the magnetic pseudo-pressure
	$r$  
	and  the temperature $ T $, such that
	\begin{subequations}\label{mhd1}
		\begin{align}
			-\frac{1}{H^{2}_a}\Delta \mathbf{u}+\frac{1}{N}(\mathbf{u}\cdot \nabla) \mathbf{u}+\nabla p - \frac{1}{R_{m}} \nabla \times\mathbf{B}\times \mathbf{B}=\mathbf{f}_1 - \frac{G_{r}}{NR^{2}_{e}}T{\mathbf{g}},\quad
			&\mbox{in}\ \Omega,
			\label{mhd1:sub1}\\
			\nabla \cdot \mathbf{u}=0,\quad&\mbox{in}\ \Omega,
			\label{mhd1:sub2}\\
			\frac{1}{R_{m}}\nabla\times \nabla\times\mathbf{B}-\nabla\times(\mathbf{u}\times \mathbf{B})+\nabla r=\mathbf{f}_2,\quad&\mbox{in}\ \Omega,
			\label{mhd1:sub3}\\
			\nabla \cdot \mathbf{B}=0,\quad&\mbox{in}\ \Omega,
			\label{mhd1:sub4}\\
			-\frac{1}{P_{r}R_{e}}\Delta T+(\mathbf{u}\cdot\nabla)T=f_3,\quad
			&\mbox{in}\ \Omega,
			\label{mhd1:sub5}
		\end{align}
	\end{subequations}
	subject to the homogenous  boundary conditions
	\begin{equation}\label{mhd6}
		\mathbf{u}|_{\partial\Omega}=0,\quad	(\mathbf{B}\times\mathbf{n})|_{\partial\Omega}=0,\quad      T|_{\partial\Omega}=0, \quad
		r|_{\partial\Omega}=0.	 
	\end{equation}
	Here   
	$H_{a}$ is the Hartmann number, $N$    the Interaction parameter, $R_{e}$    the Reynolds number, $P_{r}$  the Prandtl number, $R_{m}$   the Magnetic Reynolds number, 
	$G_{r}$   the Grashof number,   $\mathbf{g}$  the vector of gravitational acceleration
	with $\mathbf{g}=(0,1)^{\mathrm{T}}$ when $d=2$ and $\mathbf{g}=(0,0,1)^{\mathrm{T}}$ when $d=3$,
	$\mathbf{f}_1 
	$ and $\mathbf{f}_2
	$ 
	are the forcing functions, and $f_3
	$ denotes the heat source term. We refer to \cite{BMV2010,M1994,M1995} for the study of the existence and uniqueness of weak solutions to  related steady thermally coupled incompressible MHD models. 
	
	There are limited works in the literature on  the finite element analysis of  the steady thermally coupled incompressible MHD equations. Meir  \cite{M1995} proposed	a Galerkin mixed finite element method     and established optimal error estimates. 
	Codina and   Hern\'{a}ndez	\cite{CH2011}  developed  a stabilized finite element method. 
Yang and Zhang  \cite{YZ2020} analyzed  three 	iteration algorithms, i.e. the Stokes, Newton and Oseen iterations,  based on  conforming mixed finite element discretization, and Dong and Su \cite{YZ2020} also analyzed  the three 	iteration algorithms based on 
	charge-conservative mixed finite element  discretization.  
	We refer to  \cite{LD2022,R2018,R2019} for several works on fully discrete mixed finite element methods for unsteady thermally coupled incompressible MHD model equations.

	It is well-known that the two divergence
	constraints  \eqref{mhd1:sub2} and \eqref{mhd1:sub4}  on  the velocity and magnetic fields in the steady thermally coupled incompressible MHD model \eqref{mhd1} are  corresponding to the conservation of mass and magnetic flux, respectively, and that  poor conservation of such    
	physical properties in the algorithm design   may lead to numerical instabilities \cite{ABLR2010,BB1980,JWP1996, JLMNR2017,  L2009, OR2004,T2000}. 
	%
	For  incompressible MHD equations, there have developed some divergence-free finite element methods, e.g.  the mixed  interior-penalty discontinuous Galerkin (DG)  method  with the exactly divergence-free velocity \cite{GLSW2010}, the central DG method  with the exactly divergence-free magnetic field  \cite{HMX2017,LX2012,LXY2011}, 
	the mixed DG method and an HDG method 
	 with  the exactly divergence-free velocity
	\cite{QS2017},
	the mixed DG method with  the exactly divergence-free velocity and  magnetic field  \cite{HLMZ2018}, the constrained transport finite element method  with  the exactly divergence-free velocity and  magnetic field \cite{LZZ2021}, and the weak Galerkin (WG) method with  the exactly divergence-free velocity and  magnetic field \cite{ZZX2023}.

	This paper is to extend the HDG method  to the steady thermally coupled incompressible MHD  model \eqref{mhd1}.
The HDG method, presented in \cite{CGL2009} for diffusion problems,
provides a unifying strategy for hybridization of finite element methods.
The resultant HDG methods preserve the advantages of standard DG methods and lead to discrete systems of significantly reduced sizes. We refer
to \cite{CX2024,CG2005,CNP2010,CS2014,FJQ2019,L2010,LS2016,MNP2011,NPC2010,NPC2011,PNC2010,QS2016,RW2018} for some HDG methods
for the Stokes equations, the Navier-Stokes equations and Stokes-like equations. In particular, in \cite{CX2024,RW2018} the same  technique, i.e.  introducing the pressure trace
on the inter-element boundaries as a Lagrangian multiplier so as to derive a
divergence-free velocity approximation, as  that used in \cite{CFX2016}, was adopted to
construct  globally divergence-free HDG methods for the   Navier-Stokes equations.

	Our HDG discretization for  \eqref{mhd1} and numerical analysis  are of  the following main features:
	\begin{itemize}
		\item
		The HDG scheme uses   piecewise polynomials of degrees $k(k\geq 1),k,k-1,k-1$, and $k$ to approximate respectively the velocity, the magnetic field, the pressure, the magnetic pseudo-pressure and the temperature
		in the interior of elements,  and applies piecewise polynomials of degree $k$ to approximate their numerical traces on
		the interfaces of elements. 
		
		\item The scheme is ``parameter-friendly" in the sense that no  ``sufficiently large"  stabilization parameters are  required. 
		
		\item The scheme yields   globally  exactly-divergence-free approximations of the velocity and magnetic fields, thus leading to pressure-robustness of the method.
		
		
	\item	Existence and uniqueness results  are established for  the discrete scheme.

		\item  Optima  error estimates are   obtained.
		
	\end{itemize}
	
	The rest of this paper is arranged as follows.
	Section 2 gives weak formulations of the model problem.
	Section 3 is devoted to the HDG finite element scheme and some preliminary results. In Section 4 we
	discuss the existence and uniqueness of the discrete solution. Section 5 derives a priori error estimates.
	Finally, We provide some numerical results.

	\section{Weak problem}
	
	\subsection{Notation}
	
	For any bounded domain $D\subset R^s(s=d,d-1)$, nonnegative  integer $m$ and real number $1\leq q<\infty$, let $W^{m,q}(D)$ and $W_0^{m,q}(D)$ be the usual Sobolev spaces defined on $D$ with norm $||\cdot||_{m,q, D}$ and semi-norm $|\cdot|_{m,q,D}$. In particular,   $H^m(D):=W^{m,2}(D)$ and $H^m_0(D):=W^{m,2}_0(D)$, with $||\cdot||_{m,D}:=||\cdot||_{m,2, D}$ and $|\cdot|_{m,D}:=|\cdot|_{m,2, D}$.   We use $(\cdot,\cdot)_{m,D}$ to denote the inner product of $H^m(D)$, with $(\cdot,\cdot)_D :=(\cdot,\cdot)_{0,D}$.
	When $D = \Omega$, we set $||\cdot||_m := ||\cdot||_{m,\Omega},|\cdot|_m := |\cdot|_{m,\Omega}$,
	and $(\cdot,\cdot):= (\cdot,\cdot)_{\Omega}$.
	Especially, when
	$D \subset R^{d-1}$ we use $\langle\cdot,\cdot\rangle_D$ to replace $(\cdot,\cdot)_{D}$.
	For any integer
	$k \geq 0$, let $P_k(D)$ denote the set of all polynomials on $D$ with degree no more than $k$. We also need the following spaces:
	\begin{eqnarray*}
		&&L_0^2(\Omega):=\left\{
		v\in L^2(\Omega): \ (v,1)=0
		\right\}, \\
		&&\mathbf{H}(\text{div},\Omega):=\left\{
		\mathbf{v}\in  [{L}^2(\mathcal{D})]^d:\ \nabla\cdot \mathbf{v}\in L^2(\mathcal{D})
		\right\},\\
		&&\mathbf{H}(\text{curl};\Omega):=\{\mathbf{v}\in  [{L}^2(\mathcal{D})]^d:\ \nabla\times\mathbf{v}\in [L^2(\Omega)]^{2d-3}
		\},\\
		&&\mathbf{H}_0(\text{curl};\Omega):=\{\mathbf{v}\in \mathbf{H}(\text{curl};\Omega):\ \mathbf{n}\times\mathbf{v}=0 \ \ \text{on} \ \ \partial\Omega
		\},
	\end{eqnarray*}
	where the cross product $\times$ of two vectors is defined as following: for $\mathbf{v}=(v_1,\cdots,v_d)^T, \mathbf{w}=(w_1,\cdots,w_d)^T$,
	$$\mathbf{v}\times \mathbf{w}=\left\{
	\begin{array}{ll}
		v_1w_2-v_2w_1, & \text{if }d=2,\\
		(v_2w_3-v_3w_2, v_3w_1-v_1w_3,v_1w_2-v_2w_1)^T, & \text{if }d=3.
	\end{array} \right.
	$$
	
	Let $\mathcal{T}_h$ be a shape regular  partition of $\Omega$ into closed simplexes, and let
	$\varepsilon_h$ be the set of all edges(faces) of all the elements in $\Omega$.
	For any $K\in \mathcal{T}_h$, $e\in \varepsilon_h$, we denote by $h_K$ the diameter of $K$
	and by $h_e$ the diameter of $e$.
	Let $\mathbf{n}_K$  and $\mathbf{n}_e$ denote the outward unit normal vectors
	along the boundary $\partial K$ and $e$, respectively. Sometimes we may abbreviate $\mathbf{n}_K$ as $\mathbf{n}$.
	
	We use  $\nabla_h$, $\nabla_h\cdot$ and $\nabla_h\times$ to denote  respectively  the operators of piecewise-defined gradient, divergence and
	curl with respect to the decomposition $\mathcal{T}_h$.
	We also introduce the following mesh-dependent inner products and norms:
	\begin{eqnarray*}
		\langle u,v\rangle_{\partial\mathcal{T}_h}:=\sum_{K\in \mathcal{T}_h}\langle u,v\rangle_{\partial K},\ \ \ \ \|v\|_{0,\partial\mathcal{T}_h}:=\left(\sum_{K\in \mathcal{T}_h}\|v\|^2_{0,\partial K}\right)^{1/2}.
	\end{eqnarray*}

	Throughout this paper, we use $\alpha\lesssim\beta $ to denote    $\alpha\leq C\beta $, where $C$ is a positive constant independent of the  mesh size $h$. And $\alpha\sim \beta$ simplifies $\alpha\lesssim\beta\lesssim\alpha$.

	\subsection{Weak form}
	
	For simplicity, we set
	\begin{align*}
		\mathbf{V}:=[H_0^1(\Omega)]^d,\ \ \
		\mathbf{W}:=\mathbf{H}_0(\text{curl};\Omega). 
	\end{align*}
	
	For all $ \mathbf{u},\mathbf{v},\Phi \in \mathbf{V}, \ \mathbf{B},\mathbf{w} \in
	\mathbf{W},\ T,z\in H_0^1(\Omega),\ q\in L_0^2(\Omega),\theta\in H_0^1(\Omega)$,
	we define the following  bilinear and trilinear forms:
	\begin{eqnarray*}
		&&a_{1}(\mathbf{u},\mathbf{v}):=\frac{1}{H_a^2}(\nabla\mathbf{u},\nabla\mathbf{v}), \ \ \ \ \ \ \ \ \ \ \ \ \ \
		b_1(\mathbf{v},q):=( q,\nabla\cdot\mathbf{v}),\\
		&&a_{2}(\mathbf{B},\mathbf{w})
		:=\frac{1}{R_m^2}(\nabla\times\mathbf{B},\nabla\times\mathbf{w}), \ \ \ \
		b_2(\mathbf{w},\theta):=\frac{1}{R_m}(\nabla \theta,\mathbf{w}),\\
		&&a_{3}(T,z):=\frac{1}{P_rR_e}(\nabla T,\nabla z),\ \ \ \ \ \ \ \ \ \ \ \
		G_{3}(T,\mathbf{v}):=(\frac{G_r}{NR_e^2}\frac{\mathbf{g}}{g}T,\mathbf{v}),\\
		&&c_{1}(\Phi;\mathbf{u},\mathbf{v}):=
		\frac{1}{N}((\Phi\cdot\nabla)\mathbf{u},\mathbf{v})
		=\frac{1}{N}\big\{
		\frac{1}{2}(\nabla\cdot(\Phi\otimes\mathbf{u}),\mathbf{v})-
		\frac{1}{2}(\nabla\cdot(\Phi\otimes\mathbf{v}),\mathbf{u})
		\big\},\\
		&&c_{2}(\mathbf{v};\mathbf{B},\mathbf{w}):=
		\frac{1}{R_m}(\nabla\times\mathbf{w},\mathbf{v}\times\mathbf{B}),\\
		&&c_{3}(\mathbf{u};T,z):=((\mathbf{u}\cdot\nabla)T,z)
		=\frac{1}{2}(\nabla\cdot(\mathbf{u}T),z)-
		\frac{1}{2}(\nabla\cdot(\mathbf{u}z),T),
	\end{eqnarray*}
	It is easy to see that $c_1(\Phi;\mathbf{v},\mathbf{v})=0$ and $c_3(\mathbf{u};z,z)=0$.
	
	The weak form of the problem (\ref{mhd1}) reads: find
	$\mathbf{u}\in\mathbf{V}, \mathbf{B}\in\mathbf{W}, T\in H_0^1(\Omega), p\in L_0^2(\Omega), r\in H_0^1(\Omega)$, such that
	\begin{subequations}\label{weak1}
		\begin{align}
			&a_{1}(\mathbf{u},\mathbf{v})
			+a_{2}(\mathbf{B},\mathbf{w})
			+b_1(\mathbf{u},q)-b_1(\mathbf{v},p)
			+b_2(\mathbf{w},r)-b_2(\mathbf{B}\mathbf{,\theta})\nonumber\\
			&\ \ \ \ +c_{1}(\mathbf{u};\mathbf{u},\mathbf{v})+
			c_{2}(\mathbf{v};\mathbf{B},\mathbf{B})-
			c_{2}(\mathbf{u};\mathbf{B},\mathbf{w})\nonumber\\
			&=(\mathbf{f}_1,\mathbf{v})+\frac{1}{R_m}(\mathbf{f}_2,\mathbf{w})-G_{3}(T,\mathbf{v}),
			\quad\forall(\mathbf{v},\mathbf{w},q,\theta)\in \mathbf{V}\times\mathbf{W}\times L_0^2(\Omega)\times H^1_0(\Omega),
			\label{weak1:sub1}\\
			&a_{3}(T,z)+
			c_{3}(\mathbf{u};T,z)=(f_3,z),
			\quad\forall z\in H_0^1(\Omega).
			\label{weak1:sub2}
		\end{align}
	\end{subequations}
	
	\begin{rem}
		As shown  in  \cite[Theorem 5.5]{M1995},  the weak problem (\ref{weak1})  with  $d=3$ admits at least one solution  for  $\mathbf{f}_1 \in [H^{-1}(\Omega)]^3$, $\mathbf{f}_2 \in [L^2(\Omega)]^3$, and $f_3 \in H^{-1}(\Omega)$.  In addition,  under a certain  smallness condition the solution is unique.   
	\end{rem}

\section{Hybridizable discontinuous Galerkin	 finite element method}	
\subsection{HDG scheme}

For any integer $k\geq 1$,
we introduce the following finite dimensional spaces:
\begin{equation*}\begin{aligned}
&\mathbf{D}_h:=\{\bm{d}:\bm{d}|_K\in[\mathcal{P}_{k-1}(K)]^{d\times d},\forall K\in \mathcal{T}_h\},\\			
&\mathcal{V}_h:=\{\mathbf{v}_h:
\mathbf{v}_{h}|_K\in[\mathcal{P}_k(K)]^d,
\forall K\in \mathcal{T}_h \},\\
&\hat{\mathbf{V}}^0_h:=\{\hat{\mathbf{v}}_h:\hat{\mathbf{v}}_{h}|_e\in[\mathcal{P}_k(e)]^d, \forall e\in \varepsilon_h,\ \
 \mbox{and} \ \ \hat{\mathbf{v}}_{h}|_{\partial\Omega}=\mathbf{0}\},\\	
&\mathbf{C}_h:=\{\bm{c}:\bm{c}|_K\in[\mathcal{P}_{k-1}(K)]^{2d-3},\forall K\in \mathcal{T}_h\},\\	
&\hat{\mathbf{W}}^0_h:=\{\hat{\mathbf{w}}_h:\hat{\mathbf{w}}_{h}|_e\in[\mathcal{P}_k(e)]^d, \forall e\in \varepsilon_h,\ \
\mbox{and} \ \ \hat{\mathbf{w}}_{h}\times\mathbf{n} |_{\partial\Omega}=0\},\\ 
&\mathbf{S}_h:=\{\bm{s}:\bm{s}|_K\in[\mathcal{P}_{k-1}(K)]^{d},\forall K\in \mathcal{T}_h\},\\			
&Z_h:=\{z_h:
z_{h}|_K\in\mathcal{P}_k(K),
\forall K\in \mathcal{T}_h, \},\\
&\hat{Z}^0_h:=\{\hat{z}_h:\hat{z}_{h}|_e\in\mathcal{P}_k(e), \forall e\in \varepsilon_h,\ \
\mbox{and} \ \ \hat{z}_{h}|_{\partial\Omega}=0\},\\	
&Q_h:=(q_h\in L_0^2(\Omega):q _h|_K\in\mathcal{P}_{k-1}(K),\forall K\in \mathcal{T}_h),\\
&\hat{Q}_h^0:=\{\hat{q}_h:\hat{q}_{h}|_e\in\mathcal{P}_k(e), \forall e\in \varepsilon_h\},\\
&\hat{R}_h^0:=\{\hat{r}_h:\hat{r}_{h}|_e\in\mathcal{P}_k(e), \forall e\in \varepsilon_h, \ \
\mbox{and} \ \ \hat{r}_{h}|_{\partial\Omega}=0\}.
	\end{aligned}\end{equation*}

Introducing 
$$
\bm{L}=\frac{1}{H_a^2}\nabla\mathbf{u},\ \ 
\bm{N}=\frac{1}{R_m^2}\nabla\times\mathbf{B},\ \
\bm{A}=\frac{1}{P_rR_e}\nabla T
$$
in (\ref{mhd1}), we can rewrite it as
\begin{subequations}\label{mhd1-HDG}
	\begin{align}
		H_a^{2}\bm{L}-\nabla\mathbf{u}=0,\quad
		&\mbox{in}\ \Omega,
			\label{mhd1-HDG:sub6}\\
		\frac{1}{H_a^{2}}\nabla\cdot \bm{L}  +\frac{1}{N}(\mathbf{u}\cdot \nabla) \mathbf{u}+\nabla p - \frac{1}{R_{m}} \nabla \times\mathbf{B}\times \mathbf{B}=\mathbf{f}_1 - \frac{G_{r}}{NR^{2}_{e}}T{\mathbf{g}},\quad
		&\mbox{in}\ \Omega,
		\label{mhd1-HDG:sub1}\\
		\nabla \cdot \mathbf{u}=0,\quad&\mbox{in}\ \Omega,
		\label{mhd1-HDG:sub2}\\
		R_m^{2}\bm{N}-\nabla\times\mathbf{B}=0,\quad
		&\mbox{in}\ \Omega,
		\label{mhd1-HDG:sub7}\\
		\frac{1}{R_{m}^2}\nabla\times \bm{N}-\frac{1}{R_{m}}\nabla\times(\mathbf{u}\times \mathbf{B})+\frac{1}{R_{m}}\nabla r=\frac{1}{R_{m}}\mathbf{f}_2,\quad&\mbox{in}\ \Omega,
		\label{mhd1-HDG:sub3}\\
		\nabla \cdot \mathbf{B}=0,\quad&\mbox{in}\ \Omega,
		\label{mhd1-HDG:sub4}\\
		P_rR_e\bm{A}-\nabla T=0,\quad
		&\mbox{in}\ \Omega,
		\label{mhd1-HDG:sub8}\\
		\frac{1}{P_{r}R_{e}}\nabla\cdot \bm{A}+(\mathbf{u}\cdot\nabla)T=f_3,\quad
		&\mbox{in}\ \Omega,
		\label{mhd1-HDG:sub5}
	\end{align}
\end{subequations}
Then
we define the following bilinear forms and trilinear terms:
	\begin{equation*}\begin{aligned}
&a^L(\bm{L}_h,\bm{J}_h):=\frac{1}{H_a^2}(\bm{L}_h,\bm{J}_h),\\
&
a_{1h}(\mathcal{U}_h,\bm{J}_h):=(\mathbf{u}_h,\nabla_h\cdot \bm{J}_h)-\langle\hat{\mathbf{u}}_h,\bm{J}_h\cdot\mathbf{n} \rangle_{\partial\mathcal{T}_h},\\
&\quad s_{1h}(\mathcal{U}_h,\mathcal{V}_h):=\frac{1}{H_a^2}\langle \tau(\mathbf{u}_{h}-\hat{\mathbf{u}}_h),
\mathbf{v}_{h}-\hat{\mathbf{v}}_h\rangle_{\partial\mathcal{T}_h},\\
&a^N(\bm{N}_h,\bm{I}_h):=\frac{1}{R_m^2}(\bm{N}_h,\bm{I}_h),\\
&
a_{2h}(\mathcal{B}_h,\bm{I}_h)
			:=(\mathbf{B}_h,\nabla_{h}\times \bm{I}_h)
			+\langle\hat{\mathbf{B}}_h,\bm{I}_h\times\mathbf{n} \rangle_{\partial\mathcal{T}_h},\\
			&\quad s_{2h}(\mathcal{B}_h,\mathcal{W}_h):=\langle \tau (\mathbf{B}_{h}-\hat{\mathbf{B}}_h)\times \mathbf{n},
			(\mathbf{w}_{h}-\hat{\mathbf{w}}_{h})\times \mathbf{n}\rangle_{\partial\mathcal{T}_h},\\
&a^A(\bm{A}_h,\bm{E}_h):=\frac{1}{P_rR_e}(\bm{A}_h,\bm{E}_h),\\
&
a_{3h}(\mathcal{T}_h,\bm{E}_h):=(T_h,\nabla_h\cdot \bm{E}_h)-\langle\hat{T}_h,\bm{E}_h\cdot\mathbf{n} \rangle_{\partial\mathcal{T}_h},\\
&\quad s_{3h}(\mathcal{T}_h,\mathcal{Z}_h):=\frac{1}{P_rR_e}\langle \tau(T_{h}-\hat{T}_h),
z_{h}-\hat{z}_h\rangle_{\partial\mathcal{T}_h},\\			
&
			b_{1h}(\mathcal{V}_h,\mathcal{Q}_h):=-(\nabla_{h}\cdot\mathbf{v}_{h}, q_h)+\langle \mathbf{v}_{h}\cdot \mathbf{n}, \hat{q}_h
			\rangle_{\partial\mathcal{T}_h},\\		&b_{2h}(\mathcal{W}_h,\mathcal{\theta}_h):=-\frac{1}{R_m}(\nabla_{h}\cdot\mathbf{w}_{h}, \theta_h)+\frac{1}{R_m}\langle \mathbf{w}_{h}\cdot \mathbf{n}, \hat{\theta}_h
			\rangle_{\partial\mathcal{T}_h},\\	
			&
			G_{3h}(\mathcal{T}_{h},\mathcal{V}_{h}):=\frac{G_r}{NR_e^2}(\frac{\mathbf{g}}{g}T_{h},\mathbf{v}_{h}),\\
&c_{1h}(\Phi_h;\mathcal{U}_h,\mathcal{V}_h):=
-\frac{1}{2N}(\mathbf{u}_{h}\otimes\phi_{h},\nabla_h\mathbf{v}_{h})+			
\frac{1}{2N}\langle\hat{\mathbf{u}}_{h}\otimes\hat{\phi}_{h}\mathbf{n} ,\mathbf{v}_{h}\rangle_{\partial\mathcal{T}_h}
-\frac{1}{2N}(\mathbf{v}_{h}\otimes\phi_{h},\nabla_h\mathbf{u}_{h})+			
\frac{1}{2N}\langle\hat{\mathbf{v}}_{h}\otimes\hat{\phi}_{h}\mathbf{n} ,\mathbf{u}_{h}\rangle_{\partial\mathcal{T}_h},\\
&c_{2h}(\mathcal{V}_h;\mathcal{B}_h,\mathcal{W}_h):=
\frac{1}{R_m}(\mathbf{w}_{h},\nabla_{h}\times(\mathbf{v}_{h}\times\mathbf{B}_{h}))+\frac{1}{R_m}\langle \hat{\mathbf{w}}_h\times\mathbf{n},\mathbf{v}_{h}\times\mathbf{B}_{h}
\rangle_{\partial\mathcal{T}_h},\\
			&c_{3h}(\mathcal{U}_h;\mathcal{T}_h,\mathcal{Z}_h):=
			-\frac12(\mathbf{u}_{h} T_{h},\nabla_hz_{h})+		
\frac12\langle\hat{\mathbf{u}}_{h}\hat{T}_{h}\mathbf{n} ,z_{h}\rangle_{\partial\mathcal{T}_h}
			-\frac12(\mathbf{u}_{h} z_{h},\nabla_hT_{h})+		
			\frac12\langle\hat{\mathbf{u}}_{h}\hat{T}_{h}\mathbf{n} ,T_{h}\rangle_{\partial\mathcal{T}_h},
	\end{aligned}\end{equation*}
	for
	\begin{align*}
		&\bm{L}_h,\bm{J}_h\in\mathbf{D}_h,
		\bm{N}_h,\bm{I}_h\in\mathbf{C}_h,
	\bm{A}_h,\bm{E}_h\in\mathbf{S}_h,\\
		& \mathcal{U}_h:=\{\mathbf{u}_{h},\hat{\mathbf{u}}_h\}, \mathcal{V}_h:=\{\mathbf{v}_{h},\hat{\mathbf{v}}_h\},
		\Phi_h:=\{\phi_{h},\hat{\phi}_h\}\in
		[\mathbf{V}_h\times\hat{\mathbf{V}}_h^0],\\
		&
		\mathcal{B}_h:=\{\mathbf{B}_{h},\hat{\mathbf{B}}_h\},
		\mathcal{W}_h:=\{\mathbf{w}_{h},\hat{\mathbf{w}}_h\}\in[\mathbf{V}_h\times\hat{\mathbf{W}}_h^0],\\
		&\mathcal{T}_h:=\{T_{h},\hat{T}_h\},
		\mathcal{Z}_h:=\{z,\hat{z}_h\}\in [Z_h\times\hat{Z}_h^0],\\
		&\mathcal{Q}_h:=\{q_h,\hat{q}_h\}\in [Q_h\times\hat{Q}_h^0], \quad
		\mathcal{\theta}_h:=\{\theta_h,\hat{\theta}_h\}\in [Q_h\times\hat{R}_h^0],
	\end{align*}
	where the stabilization parameter $\tau$ in $s_{1h}(\cdot;\cdot,\cdot)$, $s_{2h}(\cdot;\cdot,\cdot)$ and  $s_{3h}(\cdot;\cdot,\cdot)$ is given by  $$\tau|_{\partial K}=h_K^{-1}, \quad \forall K\in \mathcal{T}_h.$$
	We easily see that
	$$
	c_{1h}(\Phi_h;\mathcal{V}_h,\mathcal{V}_h)=0,\quad \forall \ \Phi_h,\mathcal{V}_h. \quad\quad
	c_{3h}(\mathcal{U}_h;\mathcal{Z}_h,\mathcal{Z}_h)=0,
	\quad \forall \ \mathcal{U}_h,\mathcal{Z}_h. 
	$$	
	The HDG finite element scheme for the model (\ref{mhd1}) reads as follows: find
	$(\bm{L}_h,\mathcal{U}_h,\bm{N}_h,\mathcal{B}_h,\bm{A} _h,\mathcal{T}_h,\mathcal{P}_h,\mathcal{R}_h)\in\mathbf{D}_h\times[\mathbf{V}_h\times\hat{\mathbf{V}}_h^0]\times\mathbf{C}_h\times[\mathbf{V}_h\times\hat{\mathbf{W}}_h^0]\times\mathbf{S}_h\times[Z_h\times\hat{Z}_h^0]\times[Q_h\times\hat{Q}_h^0]\times[Q_h\times\hat{R}_h^0]$,
such that, for all 
$(\bm{J}_h,\mathcal{V}_h,I_h,\mathcal{W}_h,E_h,\mathcal{Z}_h,\mathcal{Q}_h,\mathcal{\theta}_h)\in\mathbf{D}_h\times[\mathbf{V}_h\times\hat{\mathbf{V}}_h^0]\times\mathbf{C}_h\times[\mathbf{V}_h\times\hat{\mathbf{W}}_h^0]\times\mathbf{S}_h\times[Z_h\times\hat{Z}_h^0]\times[Q_h\times\hat{Q}_h^0]\times[Q_h\times\hat{R}_h^0]$,
\begin{small}
	\begin{subequations}\label{Tscheme0101*}
		\begin{align}
			a^L(\bm{L}_h,\bm{J}_h)-	a_{1h}(\mathcal{U}_h,\bm{J}_h)&	=0,
			\label{Tscheme0101*-a}\\
			a_{1h}(\mathcal{V}_h,\bm{L}_h)
			+s_{1h}(\mathcal{U}_h,\mathcal{V}_h)
			+b_{1h}(\mathcal{V}_h,\mathcal{P}_h)
			+c_{1h}(\mathcal{U}_h;\mathcal{U}_h,\mathcal{V}_h)
			+c_{2h}(\mathcal{V}_h;\mathcal{B}_h,\mathcal{B}_h)
			&=(\mathbf{f}_1,\mathbf{v}_h)-	G_{3h}(T_{h},\mathbf{v}_{h}), \label{Tscheme0101*-b}\\
			b_{1h}(\mathcal{U}_h,\mathcal{Q}_h)&=0
			, \label{Tscheme0101*-c}\\
			a^N(\bm{N}_h,\bm{I}_h)-	a_{2h}(\mathcal{B}_h,\bm{I}_h)&	=0,
			\label{Tscheme0101*-d}\\
			a_{2h}(\mathcal{W}_h,\bm{N}_h)
		+
			s_{2h}(\mathcal{B}_h,\mathcal{W}_h)
			+b_{2h}(\mathcal{W}_h,\mathcal{R}_h)-c_{2h}(\mathcal{U}_h;\mathcal{B}_h,\mathcal{W}_h)
			&= \frac{1}{R_m}(\mathbf{f}_2,\mathbf{w}_{h}), \label{Tscheme0101*-e}\\
			b_{2h}(\mathcal{B}_h,\mathcal{\theta}_h)&=0, \quad \label{Tscheme0101*-f}\\
			a^A(\bm{A}_h,\bm{E}_h)-	a_{3h}(\mathcal{T}_h,\bm{E}_h)&	=0,
			\label{Tscheme0101*-g}\\
			a_{3h}(\mathcal{Z}_h,\bm{A}_h)
			+s_{3h}(\mathcal{T}_h,\mathcal{Z}_h)+c_{3h}(\mathcal{U}_h;\mathcal{T}_h,\mathcal{Z}_h)
			&=(f_3,\bar{z}_h).\label{Tscheme0101*-h}
		\end{align}
	\end{subequations}
\end{small}	
	
	As shown in \cite[Theorem 4.1]{CX2024}, the equations \eqref{Tscheme0101*-c} and \eqref{Tscheme0101*-f} lead to the globally divergence-free discrete  solutions of velocity  and   magnetic  field, respectively, i.e. there hold
	\begin{align}\label{T-divfree-velocity}
		&\mathbf{u}_{h}\in \mathbf{H}(\text{div},\Omega),\quad \nabla\cdot\mathbf{u}_{h}=0,
		\\
		\label{T-divfree-magnet}
		& \mathbf{B}_{h}\in \mathbf{H}(\text{div},\Omega), \quad \nabla\cdot\mathbf{B}_{h}=0.
	\end{align}


Then we introduce three operators:

 $\mathcal{G}_h:\mathbf{V}_{h}\times\hat{\mathbf{V}}_{h}^0\rightarrow \mathbf{D}_h$ defined by 
\begin{align}\label{def-khu}
(\mathcal{G}_h\mathbf{v}_{h},\bm{J}_h)=-a_{1h}(\mathbf{v}_{h},\bm{J}_h),\ \forall  \mathbf{v}_{h}\in \mathbf{V}_{h}\times\hat{\mathbf{V}}_{h}^0,\ \forall \bm{J}_h\in\mathbf{D}_h.
\end{align}

$\mathcal{K}_h:\mathbf{V}_{h}\times\hat{\mathbf{W}}_{h}^0\rightarrow \mathbf{C}_h$ defined by 
\begin{align}\label{def-khB}
	(\mathcal{K}_h\mathbf{w}_{h},\bm{I}_h)=-a_{2h}(\mathbf{w}_{h},\bm{I}_h),\ \forall  \mathbf{w}_{h}\in \mathbf{V}_{h}\times\hat{\mathbf{W}}_{h}^0,\ \forall \bm{I}_h\in\mathbf{C}_h.
\end{align}

$\mathcal{Y}_h:Z_{h}\times\hat{Z}_{h}^0\rightarrow \mathbf{S}_h$ defined by 
\begin{align}\label{def-khT}
	(\mathcal{Y}_hz_{h},\bm{E}_h)=-a_{3h}(z_{h},\bm{E}_h),\ \forall  z_{h}\in Z_{h}\times\hat{Z}_{h}^0,\ \forall \bm{E}_h\in\mathbf{S}_h.
\end{align}
It is easy to see that $\mathcal{G}_h$, $\mathcal{K}_h$ and $\mathcal{Y}_h$ are well defined. From \eqref{Tscheme0101*-a}, \eqref{Tscheme0101*-d} and \eqref{Tscheme0101*-g} we immediately have
\begin{align*}
\bm{L}_h=H_a^2\mathcal{G}_h\mathcal{U}_h,\ \
\bm{N}_h=R_m^2\mathcal{K}_h\mathcal{B}_h,\ \
\bm{A}_h=P_rR_e\mathcal{Y}_h\mathcal{T}_h.
\end{align*}
Hence, we can rewrite \eqref{Tscheme0101*} as following system:

Find
$(\bm{L}_h,\mathcal{U}_h,\bm{N}_h,\mathcal{B}_h,A_h,\mathcal{T}_h,\mathcal{P}_h,\mathcal{R}_h)\in\mathbf{D}_h\times[\mathbf{V}_h\times\hat{\mathbf{V}}_h^0]\times\mathbf{C}_h\times[\mathbf{V}_h\times\hat{\mathbf{W}}_h^0]\times\mathbf{S}_h\times[Z_h\times\hat{Z}_h^0]\times[Q_h\times\hat{Q}_h^0]\times[Q_h\times\hat{R}_h^0]$,
such that 
\begin{small}
\begin{subequations}\label{Tscheme0101}
	\begin{align}
		\bm{L} _h-\frac{1}{H_a^2}\mathcal{G}_h\mathcal{U}_h&	=0,
		\label{Tscheme0101-a}\\
		\frac{1}{H_a^2}(\mathcal{G}_h\mathcal{U}_h,\mathcal{G}_h\mathcal{V}_h)
		+s_{1h}(\mathcal{U}_h,\mathcal{V}_h)
		+b_{1h}(\mathcal{V}_h,\mathcal{P}_h)
		+c_{1h}(\mathcal{U}_h;\mathcal{U}_h,\mathcal{V}_h)
		+c_{2h}(\mathcal{V}_h;\mathcal{B}_h,\mathcal{B}_h)
		&=(\mathbf{f}_1,\mathbf{v}_h)
		-G_{3h}(T_{h},\mathbf{v}_{h}), \label{Tscheme0101-b}\\
		b_{1h}(\mathcal{U}_h,\mathcal{Q}_h)&=0
		, \label{Tscheme0101-c}\\
		\bm{N}_h-\frac{1}{R_m^2}\mathcal{K}_h\mathcal{B}_h&	=0,
		\label{Tscheme0101-d}\\
		\frac{1}{R_m^2}(\mathcal{K}_h\mathcal{B}_h,\mathcal{K}_h\mathcal{W}_h)
		+s_{2h}(\mathcal{B}_h,\mathcal{W}_h)
		+b_{2h}(\mathcal{W}_h,\mathcal{R}_h)-c_{2h}(\mathcal{U}_h;\mathcal{B}_h,\mathcal{W}_h)
		&= \frac{1}{R_m}(\mathbf{f}_2,\mathbf{w}_{h}), \label{Tscheme0101-e}\\
		b_{2h}(\mathcal{B}_h,\mathcal{\theta}_h)&=0, \quad \label{Tscheme0101-f}\\
		\bm{A}_h-	\frac{1}{P_rR_e}\mathcal{Y}_h\mathcal{T}_h&	=0,
		\label{Tscheme0101-g}\\
		\frac{1}{P_rR_e}(\mathcal{Y}_h\mathcal{T}_h,\mathcal{Y}_h\mathcal{Z}_h)
		+s_{3h}(\mathcal{T}_h,\mathcal{Z}_h)
		+c_{3h}(\mathcal{U}_h;\mathcal{T}_h,\mathcal{Z}_h)
		&=(f_3,z_h),\label{Tscheme0101-h}
	\end{align}
\end{subequations}
\end{small}
holds for all 
$(\mathcal{V}_h,\mathcal{W}_h,\mathcal{Z}_h,\mathcal{Q}_h,\mathcal{\theta}_h)\in[\mathbf{V}_h\times\hat{\mathbf{V}}_h^0]\times[\mathbf{V}_h\times\hat{\mathbf{W}}_h^0]\times[Z_h\times\hat{Z}_h^0]\times[Q_h\times\hat{Q}_h^0]\times[Q_h\times\hat{R}_h^0]$.

To discuss  the existence and uniqueness of the discrete solution of the scheme (\ref{Tscheme0101*}) and derive  error estimates, we will give some preliminary results in next subsection.

\subsection{ Preliminary results}

Introduce the following semi-norms:
\begin{equation*}
	\begin{aligned}
		&|||\mathcal{V}_h|||_V :=\left(\|\mathcal{G}_h(\mathbf{v}_h,\hat{\mathbf{v}}_h)\|_0^2+\|\tau^{\frac{1}{2}}
		( \mathbf{v}_h-\hat{\mathbf{v}}_h)\|_{0,\partial\mathcal{T}_h}^2\right)^{1/2},\quad \forall \mathcal{V}_h \in \mathbf{V}_h\times \hat{\mathbf{V}}_h^0,
		\\
		&|||\mathcal{W}_h|||_W:=\left(\|\mathcal{K}_h(\mathbf{w}_h,\hat{\mathbf{w}}_h)\|_0^2+\|\tau^{\frac{1}{2}}
		( \mathbf{w}_{h}-\hat{\mathbf{w}}_h)\times \mathbf{n} \|_{0,\partial\mathcal{T}_h}^2\right)^{1/2}, \quad \forall
		\mathcal{W}_h\in \mathbf{V}_h\times \hat{\mathbf{W}}_h^0,\\
		&|||\mathcal{Z}_h|||_Z :=\left(\|\|\mathcal{Y}_h(z_h,\hat{z}_h)\|_0^2+\|\tau^{\frac{1}{2}}
		(z_h-\hat{z}_h)\|_{0,\partial\mathcal{T}_h}^2\right)^{1/2},\quad \forall \mathcal{Z}_h \in Z_h\times \hat{Z}_h^0,
		\\
		&|||\mathcal{Q}_h|||_Q:=\left(\|q_h\|_0^2+h_K^2\|q_h-\hat{q}_h\|_{0,\partial\mathcal{T}_h}^2\right)^{1/2},\quad \forall
		Q_h\in \mathcal{Q}_h\times \hat{Q}_h^0,\\
		&|||\mathcal{\theta}_h|||_R:=\left(\|\theta_h-\bar{\theta}_{h}\|_0^2+
		h_K^2\|\theta_h-\hat{\theta}_h\|_{0,\partial\mathcal{T}_h}^2\right)^{1/2}, \quad \forall  \mathcal{\theta}_h\in Q_h\times \hat{R}_h^0,
	\end{aligned}
\end{equation*}
where $\bar\theta_{h}:=\frac{1}{|\Omega|}\int_\Omega\theta_{h}  d\mathbf{x}$ denotes the mean value of $\bar{\theta}_{h}$ and we recall that $\tau|_{\partial K}=h_K^{-1}$.
It is  easy to see that $|||\cdot|||_V,  |||\cdot|||_Z, |||\cdot|||_Q $ and $|||\cdot|||_R$
are  norms  on $ [\mathbf{V}_h\times\hat{\mathbf{V}}_h^0],  [Z_h\times\hat{Z}_h^0],  [Q_h\times\hat{Q}_h^0]$
and $[Q_h\times\hat{R}_h^0]$, respectively  (cf.  \cite{CFX2016,CX2024}).
$|||\cdot|||_W$ is a norm on $[\mathbf{V}_h\times\hat{\mathbf{W}}_h^0]$  (cf.  \cite{ZZX2023}).

In view of  the Green's formula, the Cauchy-Schwarz inequality, the trace inequality and the inverse inequality, we can easily derive the following inequalities on $[\mathbf{V}_h\times\hat{\mathbf{V}}_h^0],  [\mathbf{V}_h\times\hat{\mathbf{W}}_h^0]$  and $ [Z_h\times\hat{Z}_h^0]$ .
\begin{myLem}\label{lemma3}
	\cite{CX2024}
For any $(\mathcal{V}_h,\mathcal{W}_h,\mathcal{Z}_h)\in[\mathbf{V}_h\times\hat{\mathbf{V}}_h^0]\times[\mathbf{V}_h\times\hat{\mathbf{W}}_h^0]\times[Z_h\times\hat{Z}_h^0]$, there hold	
\begin{subequations}\label{lemma41}
		\begin{align}
&|||\mathcal{V}_h|||_V\sim \|\nabla \mathbf{v}_{h}\|_{0,K}+h_{K}^{-\frac{1}{2}}\|\mathbf{v}_{h}-\hat{\mathbf{v}}_h\|_{0,\partial K}.
			\label{lemma41:sub2}\\
			&|||\mathcal{W}_h|||_W\sim \|\nabla \times \mathbf{w}_{h}\|_{0,K}+h_{K}^{-\frac{1}{2}}\|(\mathbf{w}_{h}-\hat{\mathbf{w}}_h)\times\mathbf{n}\|_{0,\partial K}.
			\label{lemma41:sub4}\\
			&|||\mathcal{Z}_h|||_Z\sim \|\nabla z_{h}\|_{0,K}+h_{K}^{-\frac{1}{2}}\|z_{h}-\hat{z}_h\|_{0,\partial K}.
			\label{lemma41:sub6}
		\end{align}
	\end{subequations}
\end{myLem}

%

\begin{myLem}\label{lemma7}\cite{CX2024}
(HDG Sobolev emdedding)
	There hold
	\begin{subequations}\label{lemma42}
		\begin{align}
			&\|\nabla_{h}\mathbf{v}_h\|_{0}\lesssim |||\mathcal{V}_{h}|||_V,\ \ \forall \mathbf{v}_{h}\in
			[\mathbf{V}_h\times\hat{\mathbf{V}}_h^0],
			\label{lemma42:sub1}\\
			&\|\nabla_{h}\times \mathbf{w}_{h}\|_{0}\lesssim |||\mathcal{W}_{h}|||_W,\ \ \forall \mathbf{w}_{h}\in[\mathbf{V}_h\times\hat{\mathbf{W}}_h^0],
			\label{lemma42:sub2}\\
			&\|\nabla_{h}z_h\|_{0}\lesssim |||\mathcal{Z}_{h}|||_Z,\ \ \forall z_{h}\in [Z_h\times\hat{Z}_h^0].
			\label{lemma42:sub3}
		\end{align}
	\end{subequations}
	In addition,
	\begin{subequations}\label{T-lemma42}
		\begin{align}
			&\|\mathbf{v}_{h}\|_{0,q}\lesssim|||\mathcal{V}_{h}|||_V,\quad \forall \mathbf{v}_{h}\in
			[\mathbf{V}_h\times\hat{\mathbf{V}}_h^0],
			\label{T-lemma42:sub3}\\
			&\|z_{h}\|_{0,q}\lesssim|||\mathcal{Z}_{h}|||_Z, \quad \forall z_{h}\in [Z_h\times\hat{Z}_h^0],
			\label{T-lemma42:sub4}
		\end{align}
	\end{subequations}
	for $ 2\leq q<\infty $ when $d=2$, and for $ 2\leq q\leq6 $ when $d=3$.
\end{myLem}

\begin{proof}
	The inequality \eqref{lemma42} follows from  Lemma \ref{lemma3} directly and (\ref{T-lemma42}) comes from \cite[Lemma 3.5]{HX2019}.
\end{proof}

\begin{myLem}\label{lemma12*}
	\cite{ZZX2023} There holds
	\begin{eqnarray}\label{T-woL03}
		\|\mathbf{w}_{h}\|_{0,3,\Omega}\lesssim  |||\mathbf{w}_{h}|||_W, \quad \forall \mathcal{W}_h=\{\mathbf{w}_{h},\hat{\mathbf{w}}_h\}\in \mathbf{\bar W}_h.
	\end{eqnarray}
\end{myLem}


In light of the trace theorem, the inverse inequality and scaling arguments, we can get the following lemma  (cf. \cite{SW2013}).
\begin{myLem}\label{lemma2}
For all $ K\in\mathcal{T}_{h}, \psi\in H^{1}(K) $, and $ 1\leq q\leq \infty $, there holds
\begin{equation*}
	\|\psi\|_{0,q,\partial K}\lesssim h_{K}^{-\frac{1}{q}}\|\psi\|_{0,q,K}+h_{K}^{1-\frac{1}{q}}|\psi|_{1,q,K}.
\end{equation*}
In particular, for all $ \psi\in \mathcal{P}_{k}(K) $,
\begin{equation*}
	\|\psi\|_{0,q,\partial K}\lesssim h_{K}^{-\frac{1}{q}}\|\psi\|_{0,q,K}.
\end{equation*}
\end{myLem}

\begin{myLem}\label{T-lemma1*}\cite{SW2013}
For any $ K\in\mathcal{T}_{h} $, $ e\in \varepsilon_{h}$,
and
$ 1\leq j\leq s+1 $, there hold
\begin{equation*}
	\begin{aligned}
		&\|\mathrm{v}-Q^{o}_{s}\mathrm{v}\|_{0,K}+h_K|\mathrm{v}-Q_{s}^{o}\mathrm{v}|_{1,K}
		\lesssim h_K^{j}|\mathrm{v}|_{j,K},\ \forall \mathrm{v}\in H^{j}(K),
		\\
		&\|\mathrm{v}-Q^{o}_{s}\mathrm{v}\|_{0,\partial K}
		+\|\mathrm{v}-Q^{b}_{s}\mathrm{v}\|_{0,\partial K}
		\lesssim
		h_K^{j-1/2}|\mathrm{v}|_{j,K},\ \forall \mathrm{v}\in H^{j}(K),
		\\
		&\|Q^{o}_{s}\mathrm{v}\|_{0,K}\leq \|\mathrm{v}\|_{0,K},\ \forall \mathrm{v}\in L^{2}(K),\\
		&\|Q^{b}_{s}\mathrm{v}\|_{0,e}\leq \|\mathrm{v}\|_{0,e},\ \forall \mathrm{v}\in L^{2}(e).
	\end{aligned}
\end{equation*}
\end{myLem}

For any
$K\in\mathcal{T}_h$, we introduce  the local Raviart-Thomas($\mathcal{RT}$) element space
\begin{equation*}
\mathbb{RT}_s(K)=[\mathcal{P}_s(K)]^d+\mathbf{x}\mathcal{P}_s(K)
\end{equation*}
and the $\mathcal{RT}$ projection operator  $\mathbf{P}_s^{\mathcal{RT}}: [H^1(K)]^d \rightarrow  \mathbb{RT}_s(K)$  (cf. \cite{BBDDFF2008}) defined by
\begin{subequations}\label{lemma91}
\begin{align}
	&\langle\mathbf{P}_s^{\mathcal{RT}}\mathbf{v}\cdot\mathbf{n}_e,w\rangle_e
	=\langle\mathbf{v}\cdot\mathbf{n}_e,w\rangle_e,\ \ \forall w\in \mathcal{P}_s(e),e\in\partial K,
	\quad \text{ for }s\geq 0,
	\label{lemma91:sub1}\\
	&(\mathbf{P}_s^{\mathcal{RT}}\mathbf{v},\mathbf{w})_K=(\mathbf{v},\mathbf{w})_K,\ \ \forall \mathbf{w}\in [\mathcal{P}_{s-1}(k)]^d, \quad  \text{ for } s\geq 1.
	\label{lemma91:sub2}
\end{align}
\end{subequations}

Lemmas \ref{lemma8}-\ref{lemma10} give some properties of the $\mathcal{RT}$ element space and the $\mathcal{RT}$ projection.

\begin{myLem}\label{lemma8}\cite{BBDDFF2008}
For any $\mathbf{v}_h\in\mathbb{RT}_s(K), \nabla\cdot\mathbf{v}_h|_K=0$ gives $\mathbf{v}_h\in [\mathcal{P}_s(K)]^d$.
\end{myLem}

\begin{myLem}\label{lemma9}\cite{BBDDFF2008}
For any $K\in \mathcal{T}_h$ and $\mathbf{v}\in[H^1(K)]^d$, 
the following properties   hold:
\begin{equation*}
	(\nabla\cdot\mathbf{P}_s^{RT}\mathbf{v},\phi_h)_K=(\nabla\cdot\mathbf{v},\phi_h)_K, \ \
	\forall \mathbf{v}\in[H^1(K)]^d,\phi_h\in \mathcal{P}_s(K), 
\end{equation*}
\begin{equation*}
	\|\mathbf{v}-\mathbf{P}_s^{\mathcal{RT}}\mathbf{v}\|_{0,K}\lesssim h_K^j|\mathbf{v}|_{j,K},\ \ \
	\forall 1\leq j\leq s+1,\ \ \forall \mathbf{v}\in[H^j(K)]^d.
\end{equation*}
\end{myLem}

By using   the triangle inequality, the inverse inequality, Lemma \ref{T-lemma1*} and
Lemma \ref{lemma9} we can get more estimates for the $\mathcal{RT}$ projection (cf. \cite{CX2024}):

\begin{myLem}\label{lemma10}
For any $K\in \mathcal{T}_h$, $\mathbf{v}\in[H^j(K)]^d$ and $1\leq j\leq s+1$,
the following estimates hold:
\begin{align*}
	|\mathbf{v}-\mathbf{P}_s^{\mathcal{RT}}\mathbf{v}|_{1,K}
	&\lesssim
	h_K^{j-1}|\mathbf{v}|_{j,K}, \\
	|\mathbf{v}-\mathbf{P}_s^{\mathcal{RT}}\mathbf{v}|_{0,\partial K}
	&\lesssim
	h_K^{j-\frac{1}{2}}|\mathbf{v}|_{j,K},
	\\
	|\mathbf{v}-\mathbf{P}_s^{\mathcal{RT}}\mathbf{v}|_{0,3,K}&\lesssim h_K^{j-\frac{d}{6}}|\mathbf{v}|_{j,K},\\
	|\mathbf{v}-\mathbf{P}_s^{\mathcal{RT}}\mathbf{v}|_{0,3,\partial K}&\lesssim h_K^{j-\frac{1}{3}-\frac{d}{6}}|\mathbf{v}|_{j,K}.
\end{align*}
\end{myLem}


\section{Existence and uniqueness of discrete solution}
\subsection{Stability conditions}
\begin{myLem}\label{lemma13}
For any $ \mathcal{U}_h, \mathcal{V}_h,\Phi_h \in
[\mathbf{V}_h\times\hat{\mathbf{V}}_h^0]$, $\mathcal{B}_h, \mathcal{W}_h \in[\mathbf{V}_h\times\hat{\mathbf{W}}_h^0]$, and $\mathcal{T}_h, \mathcal{Z}_h \in [Z_h\times\hat{Z}_h^0]$, there hold the following stability conditions:
\begin{subequations}\label{lemma13results}
	\begin{align}
		&c_{1h}(\Phi_h;\mathcal{U}_h,\mathcal{V}_h)
		\lesssim
		|||\Phi_h|||_V|||\mathcal{U}_h|||_V|||\mathcal{V}_h|||_V,
		\label{lemma13results:sub9}\\
		&c_{2h}(\mathcal{V}_h;\mathcal{B}_h,\mathcal{W}_h)
		\lesssim
		|||\mathcal{B}_h|||_W|||\mathcal{W}_h|||_W|||\mathcal{V}_h|||_V,
		\label{lemma13results:sub10}\\
		&c_{3h}(\mathcal{U}_h;\mathcal{T}_h,\mathcal{Z}_h)
		\lesssim
		|||\mathcal{U}_h|||_V|||\mathcal{T}_h|||_Z|||\mathcal{Z}_h|||_Z.
		\label{lemma13results:sub12}
	\end{align}
\end{subequations}

\end{myLem}
\begin{proof}
	
	For all $\mathcal{U}_h,\mathcal{V}_h,\Phi_h\in[\mathbf{V}_h\times\hat{\mathbf{V}}_h^0]$, using
	the definitions of $c_{1h}(\cdot;\cdot,\cdot)$  we have
	\begin{align*}
		2c_{1h}(\Phi_h;\mathcal{U}_h,\mathcal{V}_h)&=
		\frac{1}{N}\{
		(\mathbf{v}_{h}\otimes\phi_{h},\nabla_h\mathbf{u}_{h})-
		(\mathbf{u}_{h}\otimes\phi_{h},\nabla_h\mathbf{v}_{h})\\
		&\ \ \ \
		-\langle\hat{\mathbf{v}}_{h}\otimes\hat{\phi}_{h}\mathbf{n},\mathbf{u}_{h}\rangle_{\partial\mathcal{T}_h}
		+\langle\hat{\mathbf{u}}_{h}\otimes\hat{\phi}_{h}\mathbf{n},\mathbf{v}_{h}\rangle_{\partial\mathcal{T}_h}\}
		\\
		&=\frac{1}{N}\{(\mathbf{v}_{h}\otimes\phi_{h},\nabla_h\mathbf{u}_{h})-
		(\mathbf{u}_{h}\otimes\phi_{h},\nabla_h\mathbf{v}_{h})\\
		&\ \ \ \ +\langle(\mathbf{u}_{h}-\hat{\mathbf{u}}_{h})
		\otimes(\phi_{h}-\hat{\phi}_{h})
		\mathbf{n},\mathbf{v}_{h}\rangle_{\partial\mathcal{T}_h}
		-\langle(\mathbf{u}_{h}-\hat{\mathbf{u}}_{h})
		\otimes\phi_{h}\mathbf{n},\mathbf{v}_{h}\rangle_{\partial\mathcal{T}_h}\\
		&\ \ \ \ -\langle(\mathbf{v}_{h}-\hat{\mathbf{v}}_{h})
		\otimes(\phi_{h}-\hat{\phi}_{h})
		\mathbf{n},\mathbf{u}_{h}\rangle_{\partial\mathcal{T}_h}
		+\langle(\mathbf{v}_{h}-\hat{\mathbf{v}}_{h})
		\otimes\phi_{h}\mathbf{n},\mathbf{u}_{h}\rangle_{\partial\mathcal{T}_h}\}\\
		&
		=:\sum_{i=1}^5\mathcal{M}_i.
	\end{align*}
	Combining the H\"{o}lder's inequality and Lemma \ref{lemma7}, we have
	\begin{align*}
		|\mathcal{M}_1|&=\frac{1}{N}|(\mathbf{v}_{h}\otimes\phi_{h},\nabla_h\mathbf{u}_{h})-
		(\mathbf{u}_{h}\otimes\phi_{h},\nabla_h\mathbf{v}_{h})|\\
		&\leq\frac{1}{N}(
		\|\mathbf{v}_{h}\|_{0,4}\|\phi_{h}\|_{0,4}\|\nabla_h\mathbf{u}_{h}\|_{0,2}+
		\|\mathbf{u}_{h}\|_{0,4}\|\phi_{h}\|_{0,4}\|\nabla_h\mathbf{v}_{h}\|_{0,2})\\
		&
		\lesssim\frac{1}{N}|||\Phi_h|||_V |||\mathcal{U}_h|||_V |||\mathcal{V}_h|||_V.
	\end{align*}
	
	Using the H\"{o}lder's inequality,  the inverse inequality, Lemma \ref{lemma2} and Lemma \ref{lemma7}, we have
	\begin{align*}
		|\mathcal{M}_2|&=\frac{1}{N}|\langle(\mathbf{u}_{h}-\hat{\mathbf{u}}_{h})
		\otimes(\phi_{h}-\hat{\phi}_{h})
		\mathbf{n},\mathbf{v}_{h}\rangle_{\partial\mathcal{T}_h} |\\
		&\leq\frac{1}{N}\sum_{K\in\mathcal{T}_h}\|\phi_{h}-\hat{\phi}_{h}\|_{0,3,\partial K}
		\|\mathbf{u}_{h}-\hat{\mathbf{u}}_{h}\|_{0,2,\partial K}
		\|\mathbf{v}_{h}\|_{0,6,\partial K}\\
		&\leq\frac{1}{N}\sum_{K\in\mathcal{T}_h}h_K^{-\frac{d-1}{6}}\|\phi_{h}-\hat{\phi}_{h}\|_{0,2,\partial K}
		\|\mathbf{u}_{h}-\hat{\mathbf{u}}_{h}\|_{0,2,\partial K}
		h_K^{-\frac{1}{6}}\|\mathbf{v}_{h}\|_{0,6,K}\\
		&\leq\frac{1}{N}\sum_{K\in\mathcal{T}_h}h_K^{-\frac{1}{2}}\|\phi_{h}-\hat{\phi}_{h}\|_{0,2,\partial K}
		h_K^{-\frac{1}{2}}\|\mathbf{u}_{h}-\hat{\mathbf{u}}_{h}\|_{0,2,\partial K}
		h_K^{1-\frac{d}{6}}\|\mathbf{v}_{h}\|_{0,6,K}\\
		&\lesssim|||\Phi_h|||_V |||\mathcal{U}_h|||_V |||\mathcal{V}_h|||_V,
		\\
		|\mathcal{M}_3|&=\frac{1}{N}|-\langle(\mathbf{u}_{h}-\hat{\mathbf{u}}_{h})
		\otimes\phi_{h}\mathbf{n},\mathbf{v}_{h}\rangle_{\partial\mathcal{T}_h}|\\
		&\leq\frac{1}{N}\sum_{K\in\mathcal{T}_h}\|\phi_{h}\|_{0,4,\partial K}
		\|\mathbf{u}_{h}-\hat{\mathbf{u}}_{h}\|_{0,2,\partial K}
		\|\mathbf{v}_{h}\|_{0,4,\partial K}\\
		&\leq\frac{1}{N}\sum_{K\in\mathcal{T}_h}h_K^{-\frac{1}{4}}\|\phi_{h}\|_{0,4,K}
		\|\mathbf{u}_{h}-\hat{\mathbf{u}}_{h}\|_{0,2,\partial K}
		h_K^{-\frac{1}{4}}\|\mathbf{v}_{h}\|_{0,4,K}\\
		&\leq\frac{1}{N}\sum_{K\in\mathcal{T}_h}\|\phi_{h}\|_{0,4,K}
		h_K^{-\frac{1}{2}}\|\mathbf{u}_{h}-\hat{\mathbf{u}}_{h}\|_{0,2,\partial K}
		\|\mathbf{v}_{h}\|_{0,4,K}\\
		&\lesssim|||\Phi_h|||_V |||\mathcal{U}_h|||_V |||\mathcal{V}_h|||_V.
	\end{align*}
	Similarly, we can get
	\begin{align*}
		|\mathcal{M}_4|+|\mathcal{M}_5|\lesssim|||\Phi_h|||_V |||\mathcal{U}_h|||_V |||\mathcal{V}_h|||_V.
	\end{align*}
	Combining the above estimates yields the inequality (\ref{lemma13results:sub9}),
	and the proof of (\ref{lemma13results:sub12}) is similar.
	
	Then, for any $\mathcal{B}_h, \mathcal{W}_h \in [\mathbf{V}_h\times\hat{\mathbf{W}}_h^0]$, by the definitions of $c_{2h}(\cdot;\cdot,\cdot)$ we have
	\begin{align*}
		c_{2h}(\mathcal{V}_h;\mathcal{B}_h,\mathcal{W}_h)
		=\frac{1}{R_m}(\nabla_h\times\mathbf{w}_{h},\mathbf{v}_{h}\times\mathbf{B}_{h})
		-\frac{1}{R_m}\langle(\mathbf{w}_{h}-\hat{\mathbf{w}}_{h})
		\times\mathbf{n},\mathbf{v}_{h}\times\mathbf{B}_{h}\rangle_{\partial\mathcal{T}_h}.
	\end{align*}
	Using the H\"{o}lder's inequality, the inverse inequality, Lemmas \ref{lemma7}, \ref{lemma12*},and \ref{lemma2} again gives
	\begin{align*}
		| (\nabla_h\times\mathbf{w}_{h},\mathbf{v}_{h}\times\mathbf{B}_{h})|
		&\leq|\nabla_h\times\mathbf{w}_{h}|_{0,2}|\mathbf{B}_{h}|_{0,3}|\mathbf{v}_{h}|_{0,6}\\
		&\lesssim|||\mathbf{w}_{h}|||_W|||\mathbf{B}_{h}|||_W|||\mathbf{v}_{h}|||_V,
		\\
		|\langle(\mathbf{w}_{h}-\hat{\mathbf{w}}_{h})
		\times\mathbf{n},\mathbf{v}_{h}\times\mathbf{B}_{h}\rangle_{\partial\mathcal{T}_h}|
		&\leq\sum_{K\in\mathcal{T}_h}|(\mathbf{w}_{h}-\hat{\mathbf{w}}_{h})\times\mathbf{n}|_{0,2,\partial K}
		|\mathbf{B}_{h}|_{0,3,\partial K}
		|\mathbf{v}_{h}|_{0,6,\partial K}
		\\
		&\lesssim\sum_{K\in\mathcal{T}_h}
		|(\mathbf{w}_{h}-\hat{\mathbf{w}}_{h})\times\mathbf{n}|_{0,2,\partial K}
		h_K^{-\frac{1}{3}}|\mathbf{B}_{h}|_{0,3,K}
		h_K^{-\frac{1}{6}}|\mathbf{v}_{h}|_{0,6,K}
		\\
		&\lesssim\sum_{K\in\mathcal{T}_h}
		h_K^{-\frac{1}{2}}|(\mathbf{w}_{h}-\hat{\mathbf{w}}_{h})\times\mathbf{n}|_{0,2,\partial K}
		|\mathbf{B}_{h}|_{0,3,K}
		|\mathbf{v}_{h}|_{0,6,K}\\
		&\lesssim
		|||\mathbf{w}_{h}|||_W|||\mathbf{B}_{h}|||_W|||\mathbf{v}_{h}|||_V.
	\end{align*}
	
	As a result, the desired inequality (\ref{lemma13results:sub10}) follows.
\end{proof}

\begin{myLem}\cite{CX2024,ZZX2023}\label{lemma15}
There hold the following inf-sup inequalities:
\begin{align}\label{T-inf-sup-bh}
	&\sup_{\mathcal{V}_h\in [\mathbf{V}_h\times\hat{\mathbf{V}}_h^0]}\frac{b_{1h}(\mathcal{V}_h,\mathcal{Q}_h)}{|||\mathbf{V}_h|||_V}\gtrsim|||\mathcal{Q}_h|||_Q, \quad \forall  \mathcal{Q}_h\in   [Q_h\times\hat{Q}_h^0],\\
	&\sup_{\mathcal{W}_h\in [\mathbf{V}_h\times\hat{\mathbf{W}}_h^0]}\frac{b_{2h}(\mathcal{W}_h,\mathcal{\theta}_h)}{|||\mathcal{W}_h|||_W}\gtrsim|||\mathcal{\theta}_h|||_R, \quad \forall  \mathcal{\theta}_h \in   [Q_h\times\hat{R}_h^0].\label{T-inf-sup-tildebh}
\end{align}
\end{myLem}

\subsection{Existence and uniqueness results}

We introduce the following two spaces:
\begin{align*}
\mathbf{\bar{V}}_h&:=\{\mathcal{V}_h\in[\mathbf{V}_h\times\hat{\mathbf{V}}_h^0]: \ b_{1h}(\mathcal{V}_h,\mathcal{Q}_h)=0,\forall \mathcal{Q}_h\in [Q_h\times\hat{Q}_h^0]    \}\\
&= \{\mathcal{V}_h\in[\mathbf{V}_h\times\hat{\mathbf{V}}_h^0]: \  \mathbf{v}_h\in \mathbf{H}({\rm div},\Omega), \nabla\cdot\mathbf{v}_h=0    \},
\\
\mathbf{\bar{W}}_h&:=\{\mathcal{W}_h\in[\mathbf{V}_h\times\hat{\mathbf{W}}_h^0]: \ b_{2h}(\mathcal{W}_h,\mathcal{\theta}_h)=0,\forall \mathcal{\theta}_h\in [Q_h\times\hat{R}_h^0]    \}\\
&=\{\mathcal{W}_h\in[\mathbf{V}_h\times\hat{\mathbf{W}}_h^0]: \ \mathbf{w}_{h}\in \mathbf{H}({\rm div},\Omega), \nabla\cdot\mathbf{w}_{h}=0  \}.
\end{align*}
Thus, the solution $(\bm{L}_h,\mathcal{U}_h,\bm{N}_h,\mathcal{B}_h,\bm{A}_h,\mathcal{T}_h)\in\mathbf{D}_h\times[\mathbf{V}_h\times\hat{\mathbf{V}}_h^0]\times\mathbf{C}_h\times[\mathbf{V}_h\times\hat{\mathbf{W}}_h^0]\times\mathbf{S}_h\times[Z_h\times\hat{Z}_h^0]$ of the scheme  (\ref{Tscheme0101*}) also solves the following discretization problem: find
$(\bm{L}_h,\mathcal{U}_h,\bm{N}_h,\mathcal{B}_h,\bm{A}_h,\mathcal{T}_h)\in\mathbf{D}_h\times\mathbf{\bar{V}}_h\times\mathbf{C}_h\times\mathbf{\bar{W}}_h\times\mathbf{S}_h\times[Z_h\times\hat{Z}_h^0]$
such that
\begin{small}
\begin{subequations}\label{3c1}
\begin{align}
\bm{L}_h-\frac{1}{H_a^2}\mathcal{G}_h\mathcal{U}_h&	=0,
\label{3c1-a}\\
\frac{1}{H_a^2}(\mathcal{G}_h\mathcal{U}_h,\mathcal{G}_h\mathcal{V}_h)
+s_{1h}(\mathcal{U}_h,\mathcal{V}_h)
+c_{1h}(\mathcal{U}_h;\mathcal{U}_h,\mathcal{V}_h)
+c_{2h}(\mathcal{V}_h;\mathcal{B}_h,\mathcal{B}_h)
&=(\mathbf{f},\mathbf{v}_h)-G_{3h}(T_{h},\mathbf{v}_{h}), \label{3c1-b}\\
\bm{N}_h-\frac{1}{R_m^2}\mathcal{K}_h\mathcal{B}_h&	=0,
\label{3c1-d}\\
\frac{1}{R_m^2}(\mathcal{K}_h\mathcal{B}_h,\mathcal{K}_h\mathcal{W}_h)
-c_{2h}(\mathcal{U}_h;\mathcal{B}_h,\mathcal{W}_h)
&= \frac{1}{R_m}(\mathbf{f}_2,\mathbf{w}_{h}), \label{3c1-e}\\
\bm{A}_h-	\frac{1}{P_rR_e}\mathcal{Y}_h\mathcal{T}_h&	=0,
\label{3c1-g}\\
\frac{1}{P_rR_e}(\mathcal{Y}_h\mathcal{T}_h,\mathcal{Y}_h\mathcal{Z}_h)+c_{3h}(\mathcal{U}_h;\mathcal{T}_h,\mathcal{Z}_h)
&=(f_3,z_h),\label{3c1-h}
\end{align}
\end{subequations}
\end{small}
holds for all 
$  \mathcal{V}_h\in
[\mathbf{V}_h\times\hat{\mathbf{V}}_h^0],$
$
\mathcal{W}_h\in[\mathbf{V}_h\times\hat{\mathbf{W}}_h^0],$
$\mathcal{Z}_h\in [Z_h\times\hat{Z}_h^0].$

For the discretization problems  (\ref{Tscheme0101*}) and (\ref{3c1}), we easily  have   the following equivalence result.


\begin{myLem}\label{lemma15a}

The problems (\ref{Tscheme0101*}) and (\ref{3c1}) are equivalent in the sense that (I) and
(II) hold:
\begin{itemize}
\item[(I)]  If $(\bm{L}_h,\mathcal{U}_h,\bm{N}_h,\mathcal{B}_h,\bm{A}_h,\mathcal{T}_h,\mathcal{P}_h,\mathcal{R}_h)\in\mathbf{D}_h\times[\mathbf{V}_h\times\hat{\mathbf{V}}_h^0]\times\mathbf{C}_h\times[\mathbf{V}_h\times\hat{\mathbf{W}}_h^0]\times\mathbf{S}_h\times[Z_h\times\hat{Z}_h^0]\times[Q_h\times\hat{Q}_h^0]\times[Q_h\times\hat{R}_h^0]$ is the solution to the problem (\ref{Tscheme0101*}), then $(\bm{L}_h,\mathcal{U}_h,\bm{N}_h,\mathcal{B}_h,\bm{A}_h,\mathcal{T}_h,\mathcal{P}_h,\mathcal{R}_h)$ is also the solution to the problem (\ref{3c1});

\item[(II)] If $(\bm{L}_h,\mathcal{U}_h,\bm{N}_h,\mathcal{B}_h,\bm{A}_h,\mathcal{T}_h)\in\mathbf{D}_h\times[\mathbf{V}_h\times\hat{\mathbf{V}}_h^0]\times\mathbf{C}_h\times[\mathbf{V}_h\times\hat{\mathbf{W}}_h^0]\times\mathbf{S}_h\times[Z_h\times\hat{Z}_h^0]$ is the solution to the problem (\ref{3c1}), then
$(\bm{L}_h,\mathcal{U}_h,\bm{N}_h,\mathcal{B}_h,\bm{A}_h,\mathcal{T}_h,\mathcal{P}_h,\mathcal{R}_h)$ is also the solution to the problem (\ref{Tscheme0101*}), where $(\mathcal{P}_h,\mathcal{R}_h)\in [Q_h\times\hat{Q}_h^0]\times [Q_h\times\hat{R}_h^0]$ is given by
\begin{subequations}\label{lemma15a-1}
\begin{align}
	&b_{1h}(\mathcal{V}_h,\mathcal{P}_h)\nonumber\\
	=&(\mathbf{f}_1,\mathbf{v}_h)-
	\frac{1}{H_a^2}(\mathcal{G}_h\mathcal{U}_h,\mathcal{G}_h\mathcal{V}_h)
	-s_{1h}(\mathcal{U}_h,\mathcal{V}_h)
	-c_{1h}(\mathcal{U}_h;\mathcal{U}_h,\mathcal{V}_h)-c_{2h}(\mathcal{V}_h;\mathcal{B}_h,\mathcal{B}_h)-G_{3h}(\mathcal{T}_h,\mathcal{V}_h), \nonumber\\
	&\quad\quad\quad\quad\quad\quad
	\quad\quad\quad\quad\quad\quad
	\quad\quad\quad\quad\quad\quad
	\quad\quad\quad\quad\quad\quad
	\forall\mathcal{V}_h\in[\mathbf{V}_h\times\hat{\mathbf{V}}_h^0],
	\label{lemma15a-1:sub1}\\
	&b_{2h}(\mathcal{W}_h,\mathcal{R}_h)\nonumber\\
	=&\frac{1}{R_m}(\mathbf{f}_2,\mathbf{w}_{h})
	-\frac{1}{R_m^2}(\mathcal{K}_h\mathcal{B}_h,\mathcal{K}_h\mathcal{W}_h)
	-s_{2h}(\mathcal{B}_h,\mathcal{W}_h)
	+c_{2h}(\mathcal{U}_h;\mathcal{B}_h,\mathcal{W}_h), \nonumber\\
	&\quad\quad\quad\quad\quad\quad
	\quad\quad\quad\quad\quad\quad
	\quad\quad\quad\quad\quad\quad
	\quad\quad\quad\quad\quad\quad
	\forall\mathcal{W}_h\in[\mathbf{V}_h\times\hat{\mathbf{W}}_h^0].
	\label{lemma15a-1:sub2}
\end{align}
\end{subequations}
\end{itemize}
\end{myLem}
Denote
\begin{align}
&M_{1h}:=\sup_{\mathbf{0}\neq\Phi_h,\mathcal{U}_h,\mathcal{V}_h\in \mathbf{\bar{V}}_h}
\frac{c_{2h}(\Phi_h;\mathcal{U}_h,\mathcal{V}_h)}
{|||\Phi_h|||_V|||\mathcal{U}_h|||_V|||\mathcal{V}_h|||_V},
\label{M1h}\\
&M_{2h}:=\sup_{
\substack{
\mathbf{0}\neq\mathcal{B}_h,\mathcal{W}_h\in\mathbf{\bar{W}}_h,\ \
\mathbf{0}\neq\mathcal{V}_h\in \mathbf{\bar{V}}_h}
}
\frac{c_{2h}(\mathcal{V}_h;\mathcal{B}_h,\mathcal{W}_h)}
{|||\mathcal{B}_h|||_W|||\mathcal{V}_h|||_V|||\mathcal{W}_h|||_W},
\label{M2h}\\
&M_{3h}:=\sup_{
\substack{
0\neq \mathcal{T}_h,\mathcal{Z}_h\in [Z_h\times\hat{Z}_h^0],\ \
\mathbf{0}\neq\mathcal{U}_h\in \mathbf{\bar{V}}_h}
}
\frac{c_{3h}(\mathcal{U}_h;\mathcal{T}_h,\mathcal{Z}_h)}
{|||\mathcal{U}_h|||_V|||\mathcal{T}_h|||_Z|||\mathcal{Z}_h|||_Z}.
\label{M3h}
\end{align}	
From Lemma \ref{lemma13}, it is  easy to know that   $M_{1h}$, $M_{2h}$, and $M_{3h}$ are bounded from above by a positive constant independent of the mesh size $h$.

\begin{myLem}\label{results11}

The WG scheme (\ref{3c1}) admits at least one solution $(\bm{L}_h,\mathcal{U}_h,\bm{N}_h,\mathcal{B}_h,\bm{A}_h,\mathcal{T}_h)\in\mathbf{D}_h\times[\mathbf{V}_h\times\hat{\mathbf{V}}_h^0]\times\mathbf{C}_h\times[\mathbf{V}_h\times\hat{\mathbf{W}}_h^0]\times\mathbf{S}_h\times[Z_h\times\hat{Z}_h^0]$.
In addition, there hold
\begin{align}\label{T-3c41*}
&|||\mathcal{U}_h|||_V+|||\mathcal{B}_h|||_W\leq
2\zeta^2\left(\|\mathbf{f_1}\|_{1h}+\|\mathbf{f_2}\|_{2h}+\|f_{3}\|_{3h} \right),\\
\label{T-3c42*}
&|||\mathcal{T}_h|||_Z\leq P_rR_e\|f_3\|_{3h},
\end{align}
where
\begin{align*}
&\zeta:=\max\{H_a,R_m,\frac{H_aP_rR_eG_r\mathbf{g}}{NR_e^2g}\}, \quad\quad 
\|\mathbf{f}_1\|_{1h}:=\sup_{\mathbf{0}\neq\mathcal{V}_h\in \mathbf{\bar{V}}_h}\frac{(\mathbf{f}_1,\mathbf{v}_h)}{|||\mathcal{V}_h|||_V},\\
&\|\mathbf{f}_2\|_{2h}:=\sup_{\mathbf{0}\neq\mathcal{W}_h\in \mathbf{\bar{W}}_h}\frac{1/R_m(\mathbf{f}_2,\mathbf{w}_{h})}{|||\mathcal{W}_h|||_W},\quad\quad
\|f_3\|_{3h}:=\sup_{\mathbf{0}\neq \mathcal{Z}_h\in [Z_h\times\hat{Z}_h^0]}\frac{(\mathbf{f}_3,z_h)}{|||\mathcal{Z}_h|||_Z}.
\end{align*}
\end{myLem}

\begin{proof}
First, by Lemma \ref{lemma13} it is easy to see that, for a given $\mathcal{U}_h 
\in\mathbf{\bar{V}}_h
$, the bilinear form
$\frac{1}{P_rR_e}(\mathcal{Y}_h,\cdot)+s_{3h}(\mathcal{T}_h,\cdot)+c_{3h}(\mathcal{U}_h;\mathcal{T}_h,\cdot)$
is continuous and coercive on $[Z_h\times\hat{Z}_h^0]\times [Z_h\times\hat{Z}_h^0]$. Then the Lax-Milgram theorem implies that the problem (\ref{3c1-h}) has a unique $\mathcal{T}_h=\mathcal{T}_h(\mathcal{U}_h)
\in [Z_h\times\hat{Z}_h^0]$ with the boundedness estimate  \eqref{T-3c42*}.

We easily see that the scheme (\ref{3c1}) is equivalent to  the following problem: 
find $(\mathcal{U}_h,\mathcal{B}_h)\in\mathbf{\bar{V}}_h\times\mathbf{\bar{W}}_h$ such that
\begin{align}\label{3c21}
&\frac{1}{H_a^2}(\mathcal{G}_h\mathcal{U}_h,\mathcal{G}_h\mathcal{V}_h)
+s_{1h}(\mathcal{U}_h,\mathcal{V}_h)
+\frac{1}{R_m^2}(\mathcal{K}_h\mathcal{B}_h,\mathcal{K}_h\mathcal{W}_h)
+s_{2h}(\mathcal{B}_h,\mathcal{W}_h)\nonumber\\
&+c_{1h}(\mathcal{U}_h;\mathcal{U}_h,\mathcal{V}_h)
+c_{2h}(\mathcal{V}_h;\mathcal{B}_h,\mathcal{B}_h)-
c_{2h}(\mathcal{U}_h;\mathcal{B}_h,\mathcal{W}_h)\nonumber\\
=&(\mathbf{f}_1,\mathcal{V}_h)
+\frac{1}{R_m}(\mathbf{f}_2,\mathcal{W}_h)
-G_{3h}(\mathcal{T}_h(\mathcal{U}_h),\mathcal{V}_h),
\ \ \ \ \ \ \ \forall(\mathcal{V}_h,\mathcal{W}_h)
\in\mathbf{\bar{V}}_h\times\mathbf{\bar{W}}_h.
\end{align}
Taking $\mathcal{V}_h=\mathcal{U}_h$ and $\mathcal{W}_h=\mathcal{B}_h$ in (\ref{3c21}),
by Lemma \ref{lemma13} we obtain
\begin{align*}
\frac{1}{H_a^2}|||\mathcal{U}_h|||_V^2+\frac{1}{R_m^2}|||\mathcal{B}_h|||_W^2
\leq&
-G_{3h}(\mathcal{T}_h(\mathcal{U}_h),\mathcal{U}_h)
+(\mathbf{f}_1,\mathcal{U}_h)+(\mathbf{f}_2,\mathcal{B}_h)\\
\leq&\left(\frac{P_rR_eG_r\mathbf{g}}{NR_e^2g}\|f_3\|_{3h}+\|\mathbf{f}_1\|_{1h}
\right)|||\mathcal{U}_h|||_V
+\|\mathbf{f}_2\|_{2h}|||\mathcal{B}_h|||_W,
\end{align*}
which further yields
\begin{align*}
\frac12\min\{\frac{1}{H_a^2},\frac{1}{R_m^2}\}\left(|||\mathcal{U}_h|||_V+|||\mathcal{B}_h|||_W\right)^2
\leq &\min\{\frac{1}{H_a^2},\frac{1}{R_m^2}\}\left(|||\mathcal{U}_h|||_V^2+|||\mathcal{B}_h|||_W^2\right)\nonumber\\
\leq &\frac{1}{H_a^2}|||\mathcal{U}_h|||_V^2+\frac{1}{R_m^2}|||\mathcal{B}_h|||_W^2\nonumber\\
\leq &
H_a^2(\frac{P_rR_eG_r\mathbf{g}}{NR_e^2g}\|f_3\|_{3h}+\|\mathbf{f}_1\|_{1h})^2+ R_m^2\|\mathbf{f}_2\|_{2h}^2\nonumber\\
\leq & \left(\frac{H_aP_rR_eG_r\mathbf{g}}{NR_e^2g}\|f_3\|_{3h}+H_a\|\mathbf{f}_1\|_{1h} + R_m\|\mathbf{f}_2\|_{2h}\right)^2.
\end{align*}
This indicates  the boundedness result \eqref{T-3c41*}.

To show the existence of solution  to the problem \eqref{3c21},
we define a  mapping $\mathbb{A}:\mathbf{\bar{V}}_h\times\mathbf{\bar{W}}_h\rightarrow
\mathbf{\bar{V}}_h\times\mathbf{\bar{W}}_h$ with $\mathbb{A}(\mathcal{U}_h,\mathcal{B}_h)=
(\mathbf{x}_u,\mathbf{x}_B)$, where $(\mathbf{x}_u,\mathbf{x}_B)\in\mathbf{\bar{V}}_h\times\mathbf{\bar{W}}_h$ is given by
\begin{align}\label{3c51}
&\frac{1}{H_a^2}(\mathcal{G}_h\mathbf{x}_u,\mathcal{G}_h\mathcal{V}_h)
+s_{1h}(\mathbf{x}_u,\mathcal{V}_h)
+\frac{1}{R_m^2}(\mathcal{K}_h\mathbf{x}_B,\mathcal{K}_h\mathcal{W}_h)
+s_{2h}(\mathbf{x}_B,\mathcal{W}_h)
\nonumber\\
=&(\mathbf{f}_1,\mathcal{V}_h)+\frac{1}{R_m}(\mathbf{f}_2,\mathcal{W}_h)
-G_{3h}(\mathcal{T}_h(\mathcal{U}_h),\mathcal{V}_h)
-c_{1h}(\mathcal{U}_h;\mathcal{U}_h,\mathcal{V}_h)\nonumber\\
&-c_{2h}(\mathcal{V}_h;\mathcal{B}_h,\mathcal{B}_h)
+c_{2h}(\mathcal{U}_h;\mathcal{B}_h,\mathcal{W}_h), \quad \forall (\mathcal{V}_h,\mathcal{W}_h)
\in\mathbf{\bar{V}}_h\times\mathbf{\bar{W}}_h.
\end{align}
Clearly, $(\mathcal{U}_h,\mathcal{B}_h)$ is a solution to (\ref{3c21})
if it is a fixed point of $\mathbb{A}$, i.e.  
\begin{align}\label{uniqu}
\mathbb{A}(\mathcal{U}_h,\mathcal{B}_h)=
(\mathcal{U}_h,\mathcal{B}_h).
\end{align}
In order to show the system \eqref{uniqu} has a solution,
from the    Schaefer fixed point theorem \cite[Theorem 2.11]{LeDretHerve2018NEPD} 
it suffices to prove the following two assertions:
\begin{itemize}
\item[(i)] $\mathbb{A}$ is a continuous and compact mapping;
\item[(ii)]  The set $$
\mathbf{\Theta}_h:=\{(\mathcal{V}_h,\mathcal{W}_h)\in\mathbf{\bar{V}}_h\times\mathbf{\bar{W}}_h    ;(\mathcal{V}_h,\mathcal{W}_h)=\lambda\mathbb{A}(\mathcal{V}_h,\mathcal{W}_h)
\text{ for some } 0< \lambda\leq 1 \}$$ is bounded.
\end{itemize}

To show (i), let $\mathcal{U}_{h1},\mathcal{U}_{h2}\in\mathbf{\bar{V}}_h$ and $\mathcal{B}_{h1},\mathcal{B}_{h2}\in\mathbf{\bar{W}}_h$ be such that
$\mathbb{A}(\mathcal{U}_{h1},\mathcal{B}_{h1})=(\mathbf{x}_{1u},\mathbf{x}_{1B})$ and
$\mathbb{A}(\mathcal{U}_{h2},\mathcal{B}_{h2})=(\mathbf{x}_{2u},\mathbf{x}_{2B})$, then we have
\begin{align}
&\frac{1}{H_a^2}(\mathcal{G}_h\mathbf{x}_{1u},\mathcal{G}_h\mathcal{V}_h)
+s_{1h}(\mathbf{x}_{1u},\mathcal{V}_h)
+\frac{1}{R_m^2}(\mathcal{K}_h\mathbf{x}_{1B},\mathcal{K}_h\mathcal{W}_h)
+s_{2h}(\mathbf{x}_{1B},\mathcal{W}_h)
\nonumber\\
&+c_{1h}(\mathcal{U}_{h1};\mathcal{U}_{h1},\mathcal{V}_h)
+c_{2h}(\mathcal{V}_h;\mathcal{B}_{h1},\mathcal{B}_{h1})-
c_{2h}(\mathcal{U}_{h1};\mathcal{B}_{h1},\mathcal{W}_h)\nonumber\\
=&(\mathbf{f}_1,\mathcal{V}_h)
+\frac{1}{R_m}(\mathbf{f}_2,\mathcal{W}_h)
-G_{3h}(\mathcal{T}_h(\mathcal{U}_{h1}),\mathcal{V}_h),\label{3c61}\\
\label{3c71}
&\frac{1}{H_a^2}(\mathcal{G}_h\mathbf{x}_{2u},\mathcal{G}_h\mathcal{V}_h)
+s_{1h}(\mathbf{x}_{2u},\mathcal{V}_h)
+\frac{1}{R_m^2}(\mathcal{K}_h\mathbf{x}_{2B},\mathcal{K}_h\mathcal{W}_h)
+s_{2h}(\mathbf{x}_{2B},\mathcal{W}_h)
\nonumber\\
&+c_{1h}(\mathcal{U}_{h2};\mathcal{U}_{h2},\mathcal{V}_h)
+c_{2h}(\mathcal{V}_h;\mathcal{B}_{h2},\mathcal{B}_{h2})-
c_{2h}(\mathcal{U}_{h2};\mathcal{B}_{h2},\mathcal{W}_h)\nonumber\\
=&(\mathbf{f}_1,\mathcal{V}_h)
+\frac{1}{R_m}(\mathbf{f}_2,\mathcal{W}_h)
-G_{3h}(\mathcal{T}_h(\mathcal{U}_{h2}),\mathcal{V}_h),
\end{align}
for all $(\mathcal{V}_h,\mathcal{W}_h)
\in\mathbf{\bar{V}}_h\times\mathbf{\bar{W}}_h.$
Subtracting (\ref{3c71}) from (\ref{3c61}), and taking $\mathcal{V}_h=\mathbf{x}_{1u}-\mathbf{x}_{2u}$,
$\mathcal{W}_h=\mathbf{x}_{1B}-\mathbf{x}_{2B}$,  we get
\begin{align}\label{X-UB}
&\frac{1}{H_a^2}(\mathcal{G}_h(\mathbf{x}_{1u}-\mathbf{x}_{2u}),\mathcal{G}_h(\mathbf{x}_{1u}-\mathbf{x}_{2u}))
+s_{1h}(\mathbf{x}_{1u}-\mathbf{x}_{2u},\mathbf{x}_{1u}-\mathbf{x}_{2u})\nonumber\\
&+\frac{1}{R_m^2}(\mathcal{K}_h(\mathbf{x}_{1B}-\mathbf{x}_{2B}),\mathcal{K}_h(\mathbf{x}_{1B}-\mathbf{x}_{2B}))
+s_{2h}(\mathbf{x}_{1B}-\mathbf{x}_{2B},\mathbf{x}_{1B}-\mathbf{x}_{2B})\nonumber\\
=&-c_{1h}(\mathcal{U}_{h1}-\mathcal{U}_{h2};\mathcal{U}_{h1},\mathbf{x}_{1u}-\mathbf{x}_{2u})
-c_{1h}(\mathcal{U}_{h2};\mathcal{U}_{h1}-\mathcal{U}_{h2},\mathbf{x}_{1u}-\mathbf{x}_{2u})\nonumber\\
&-c_{2h}(\mathbf{x}_{1u}-\mathbf{x}_{2u};\mathcal{B}_{h1}-\mathcal{B}_{h2},\mathcal{B}_{h1})
-c_{2h}(\mathbf{x}_{1u}-\mathbf{x}_{2u};\mathcal{B}_{h2},\mathcal{B}_{h1}-\mathcal{B}_{h2})\nonumber\\
&+c_{2h}(\mathcal{U}_{h1};\mathcal{B}_{h1}-\mathcal{B}_{h2},\mathbf{x}_{1B}-\mathbf{x}_{2B})
+c_{2h}(\mathcal{U}_{h1}-\mathcal{U}_{h2};\mathcal{B}_{h2},\mathbf{x}_{1B}-\mathbf{x}_{2B})\nonumber\\
&-G_{3h}(\mathcal{T}_h(\mathcal{U}_{h1})-\mathcal{T}_h(\mathcal{U}_{h2})
,\mathbf{x}_{1u}-\mathbf{x}_{2u}).
\end{align}
Substitute $T_{h1}=\mathcal{T}_h(\mathcal{U}_{h1})$ and $T_{h2}=\mathcal{T}_h(\mathcal{U}_{h2})$  into (\ref{3c1-h}), respectively,
and then subtract the first resulting equation from the second one, we can obtain
\begin{align}\label{Th-bound}
&\frac{1}{P_rR_e}(\mathcal{Y}_h(\mathcal{T}_h(\mathcal{U}_{h1})-
\mathcal{T}_h(\mathcal{U}_{h2})),\mathcal{Y}_h\mathcal{Z}_h)+
s_{3h}(\mathcal{T}_h(\mathcal{U}_{h1})-
\mathcal{T}_h(\mathcal{U}_{h2}),\mathcal{Z}_h)\nonumber\\
=&-c_{3h}(\mathcal{U}_{h1}-\mathcal{U}_{h2};\mathcal{T}_h(\mathcal{U}_{h1}),\mathcal{Z}_h)
-c_{3h}(\mathcal{U}_{h2};\mathcal{T}_h(\mathcal{U}_{h1})
-\mathcal{T}_h(\mathcal{U}_{h2}),\mathcal{Z}_h), \ \ \ \forall z\in [Z_h\times\hat{Z}_h^0].
\end{align}
Taking $\mathcal{Z}_h=\mathcal{T}_h(\mathcal{U}_{h1})-
\mathcal{T}_h(\mathcal{U}_{h2})$ in (\ref{Th-bound}) and using Lemma \ref{lemma13} and  (\ref{T-3c42*}), we   get
\begin{align}\label{T-X-UB}
\frac{1}{P_rR_e}|||\mathcal{T}_h(\mathcal{U}_{h1})-
\mathcal{T}_h(\mathcal{U}_{h2})|||_Z
&\lesssim|||\mathcal{U}_{h1}-\mathcal{U}_{h2}|||_V
|||\mathcal{T}_h(\mathcal{U}_{h1})|||_Z\nonumber\\
&\leq M_{3h}P_rR_e\|f_3\|_{3h}|||\mathcal{U}_{h1}-\mathcal{U}_{h2}|||_V.
\end{align}
which, together with Lemma \ref{lemma13} and  (\ref{X-UB}), implies
\begin{align*}
&\frac{1}{H_a^2}|||\mathbf{x}_{1u}-\mathbf{x}_{2u}|||_V^2+
\frac{1}{R_m^2}|||\mathbf{x}_{1B}-\mathbf{x}_{2B}|||_W^2\\
\leq&M_{1h}(|||\mathcal{U}_{h1}|||_V+|||\mathcal{U}_{h2}|||_V)
|||\mathcal{U}_{h1}-\mathcal{U}_{h2}|||_V|||\mathbf{x}_{1u}-\mathbf{x}_{2u}|||_V\\
&+M_{2h}(|||\mathcal{B}_{h1}|||_W+|||\mathcal{B}_{h2}|||_W)
|||\mathbf{x}_{1u}-\mathbf{x}_{2u}|||_V|||\mathcal{B}_{h1}-\mathcal{B}_{h2}|||_W\\
&+M_{2h}|||\mathbf{x}_{1B}-\mathbf{x}_{2B}|||_W
(|||\mathcal{U}_{h1}|||_W|||\mathcal{B}_{h1}-\mathcal{B}_{h2}|||_W
+|||\mathcal{B}_{h2}|||_W|||\mathcal{U}_{h1}-\mathcal{U}_{h2}|||_V)\\
&+M_{3h}P_r^2R_e^2\frac{G_r\mathbf{g}}{NR_e^2g}\|f_3\|_{3h}|||\mathcal{U}_{h1}-\mathcal{U}_{h2}|||_V
|||\mathbf{x}_{1u}-\mathbf{x}_{2u}|||_V.
\end{align*}
This estimate plus (\ref{T-3c41*}) yields
\begin{align*}
&|||\mathbf{x}_{1u}-\mathbf{x}_{2u}|||_V+
|||\mathbf{x}_{1B}-\mathbf{x}_{2B}|||_W\nonumber\\
\leq&
2\zeta \left[
H_aM_{1h}\left(|||\mathcal{U}_{h1}|||_V+|||\mathcal{U}_{h2}|||_V\right)
|||\mathcal{U}_{h1}-\mathcal{U}_{h2}|||_V
+H_aM_{2h}(|||\mathcal{B}_{h1}|||_W+|||\mathcal{B}_{h2}|||_W)
|||\mathcal{B}_{h1}-\mathcal{B}_{h2}|||_W\right.\nonumber\\
&\left.+R_mM_{2h}(|||\mathcal{U}_{h1}|||_W|||\mathcal{U}_{h1}-\mathcal{U}_{h2}|||_W
+|||\mathcal{B}_{h2}|||_W|||\mathcal{U}_{h1}-\mathcal{U}_{h2}|||_V) \right]
+2\zeta M_{3h}H_aP_r^2R_e^2\frac{G_r\mathbf{g}}{NR_e^2g}\|f_3\|_{3h}|||\mathcal{U}_{h1}-\mathcal{U}_{h2}|||_V\nonumber\\
\leq&
4\zeta^3\left(\|\mathbf{f_1}\|_{1h}+\|\mathbf{f_2}\|_{2h}+\|f_{3}\|_{3h}\right)
\left[
2H_aM_{1h}
|||\mathcal{U}_{h1}-\mathcal{U}_{h2}|||_V
+2H_aM_{2h}
|||\mathcal{B}_{h1}-\mathcal{B}_{h2}|||_W\right.\nonumber\\
&\left.+R_mM_{2h}|||\mathcal{B}_{h1}-\mathcal{B}_{h2}|||_W
+R_mM_{2h}|||\mathcal{U}_{h1}-\mathcal{U}_{h2}|||_V\right]\nonumber\\
&+ 2\zeta M_{3h}H_aP_r^2R_e^2\frac{G_r\mathbf{g}}{NR_e^2g}\|f_3\|_{3h}|||\mathcal{U}_{h1}-\mathcal{U}_{h2}|||_V\nonumber\\
\leq&
\left(4\zeta^3(\|\mathbf{f_1}\|_{1h}+\|\mathbf{f_2}\|_{2h}+\|f_{3}\|_{3h})
(2H_aM_{1h}+R_mM_{2h})+2\zeta M_{3h}H_aP_r^2R_e^2\frac{G_r\mathbf{g}}{NR_e^2g}\|f_3\|_{3h}\right)
|||\mathcal{U}_{h1}-\mathcal{U}_{h2}|||_V\nonumber\\
&+4\zeta^3(\|\mathbf{f_1}\|_{1h}+\|\mathbf{f_2}\|_{2h}+\|f_{3}\|_{3h})
(2H_aM_{2h}+R_mM_{2h})
|||\mathcal{B}_{h1}-\mathcal{B}_{h2}|||_W\nonumber\\
\leq&
\left(4\zeta^3(\|\mathbf{f_1}\|_{1h}+\|\mathbf{f_2}\|_{2h}+\|f_{3}\|_{3h})
(2H_aM_{1h}+R_mM_{2h})+2\zeta  M_{3h}H_aP_r^2R_e^2\frac{G_r\mathbf{g}}{NR_e^2g}\|f_3\|_{3h}\right)\nonumber\\
&\quad \times (|||\mathcal{U}_{h1}-\mathcal{U}_{h2}|||_V+|||\mathcal{B}_{h1}-\mathcal{B}_{h2}|||_W)\nonumber\\
\leq&\max\{M_{1h},M_{2h},M_{3h}\}
\left(12\zeta^4(\|\mathbf{f_1}\|_{1h}+\|\mathbf{f_2}\|_{2h}+\|f_{3}\|_{3h})
+2\zeta^2 P_rR_e\|f_3\|_{3h}\right)\nonumber\\
&
\quad \times (|||\mathcal{U}_{h1}-\mathcal{U}_{h2}|||_V+|||\mathcal{B}_{h1}-\mathcal{B}_{h2}|||_W).
\end{align*}
This   implies that $\mathbb{A}$ is equicontinuous and uniformly bounded,
since
$$
\mathbb{A}(\mathcal{U}_{h1},\mathcal{B}_{h1})-
\mathbb{A}(\mathcal{U}_{h2},\mathcal{B}_{h2}) =
(\mathbf{x}_{1u}-\mathbf{x}_{2u},\mathbf{x}_{1B}-\mathbf{x}_{2B}).$$
Hence, $\mathbb{A}$ is compact by the
Arzel\'{a}-Ascoli theorem \cite{B2010}, and (i) holds.

The thing left is to prove (ii). 
For any $(\mathbf{v}_u,\mathbf{v}_B)
\in\mathbf{\Theta}_h$, using (\ref{3c51}) we have
\begin{align*}
&\lambda^{-1}\left(
\frac{1}{H_a^2}(\mathcal{G}_h\mathbf{\tilde{x}}_{u},\mathcal{G}_h\mathcal{V}_h)
+s_{1h}(\mathbf{\tilde{x}}_{u},\mathcal{V}_h)
+\frac{1}{R_m^2}(\mathcal{K}_h\mathbf{\tilde{x}}_{B},\mathcal{K}_h\mathcal{W}_h)
+s_{2h}(\mathbf{\tilde{x}}_{B},\mathcal{W}_h)\right)\\
&
+c_{1h}(\mathbf{\tilde{x}}_{u};\mathbf{\tilde{x}}_{u},\mathcal{V}_h)
+c_{2h}(\mathcal{V}_h;\mathbf{\tilde{x}}_{B},\mathbf{\tilde{x}}_{B})-
c_{2h}(\mathbf{\tilde{x}}_{u};\mathbf{\tilde{x}}_{B},\mathcal{W}_h)\nonumber\\
=&(\mathbf{f}_1,\mathcal{V}_h)
+\frac{1}{R_m}(\mathbf{f}_2,\mathcal{W}_h)
-G_{3h}(\mathcal{T}_h(\mathbf{\tilde{x}}_{u}),\mathcal{V}_h),
\ \ \ \
\forall(\mathcal{V}_h,\mathcal{W}_h)
\in\mathbf{\bar{V}}_h\times\mathbf{\bar{W}}_h.
\end{align*}
Similar to \eqref{T-3c41*}, there holds
\begin{align*}
|||\mathbf{\tilde{x}}_{u}|||_V+|||\mathbf{\tilde{x}}_{B}|||_W
\leq2\zeta\lambda\left(\|\mathbf{f_1}\|_{1h}+\|\mathbf{f_2}\|_{2h}+\|f_{3}\|_{3h} \right).
\end{align*}
i.e. (ii) holds. 

As a result,  $(\bm{L}_h=\frac{1}{H_a^2}\mathcal{G}_h\mathcal{U}_h,\mathcal{U}_h,\bm{N}_h=\frac{1}{R_m}\mathcal{K}_h\mathcal{B}_h,\mathcal{B}_h,\bm{A}_h=\frac{1}{P_rR_e}\mathcal{Y}_h\mathcal{T}_h,\mathcal{T}_h)\in
\mathbf{D}_h\times[\mathbf{V}_h\times\hat{\mathbf{V}}_h^0]\times
\mathbf{C}_h\times[\mathbf{V}_h\times\hat{\mathbf{W}}_h^0]\times
\mathbf{S}_h\times[Z_h\times\hat{Z}_h^0]$ is a solution to the scheme (\ref{3c1}).
This completes the proof.

\end{proof}

\begin{myLem}\label{results211}
Under the smallness condition  
\begin{align}\label{3c81}
\max\{M_{1h},M_{2h},M_{3h}\}
\left(12\zeta^4\|\mathbf{f_1}\|_{1h}+12\zeta^4\|\mathbf{f_2}\|_{2h}+(12\zeta^4H_a+2\zeta^2P_rR_e)
\|f_3\|_{3h}\right)<1,
\end{align}
the problem (\ref{3c1}) admits a unique solution $(\bm{L}_h,\mathcal{U}_h,\bm{N}_h,\mathcal{B}_h,\bm{A}_h,\mathcal{T}_h)\in\mathbf{D}_h\times[\mathbf{V}_h\times\hat{\mathbf{V}}_h^0]\times\mathbf{C}_h\times[\mathbf{V}_h\times\hat{\mathbf{W}}_h^0]\times\mathbf{S}_h\times[Z_h\times\hat{Z}_h^0]$.
\end{myLem}

\begin{proof}
Since 
$(\bm{L}_h=\frac{1}{H_a^2}\mathcal{G}_h\mathbf{u}_{h},\bm{N}_h=\frac{1}{R_m}\mathcal{K}_h\mathbf{B}_{h},\bm{A}_h=\frac{1}{P_rR_e}\mathcal{Y}_h\mathcal{T}_h)\in
\mathbf{D}_h\times
\mathbf{C}_h\times
\mathbf{S}_h$, we only need to show the uniqueness of $\mathcal{U}_h$, $\mathcal{B}_h$ and $\mathcal{T}_h$.

Let $(\mathcal{U}_{h1},\mathcal{B}_{h1})$,
$(\mathcal{U}_{h2},\mathcal{B}_{h2})\in\mathbf{\bar{V}}_h\times\mathbf{\bar{W}}_h$ be two solutions to problem (\ref{3c21}), then
for any $(\mathcal{V}_{h},\mathcal{W}_{h})\in\mathbf{\bar{V}}_h\times\mathbf{\bar{W}}_h$ we have 
\begin{align*}
&\frac{1}{H_a^2}(\mathcal{G}_h\mathcal{U}_{h1},\mathcal{G}_h\mathcal{V}_h)
+s_{1h}(\mathcal{U}_{h1},\mathcal{V}_h)
+\frac{1}{R_m^2}(\mathcal{K}_h\mathcal{B}_{h1},\mathcal{K}_h\mathcal{W}_h)
+s_{2h}(\mathcal{B}_{h1},\mathcal{W}_h)\\
&+c_{1h}(\mathcal{U}_{h1};\mathcal{U}_{h1},\mathcal{V}_h)
+c_{2h}(\mathcal{V}_h;\mathcal{B}_{h1},\mathcal{B}_{h1})-
c_{2h}(\mathcal{U}_{h1};\mathcal{B}_{h1},\mathcal{W}_h)\nonumber\\
=&(\mathbf{f}_1,\mathcal{V}_h)
+\frac{1}{R_m}(\mathbf{f}_2,\mathcal{W}_h)
-G_{3h}(\mathcal{T}_h(\mathcal{U}_{h1}),\mathcal{V}_h),\\
&\frac{1}{H_a^2}(\mathcal{G}_h\mathcal{U}_{h2},\mathcal{G}_h\mathcal{V}_h)
+s_{1h}(\mathcal{U}_{h2},\mathcal{V}_h)
+\frac{1}{R_m^2}(\mathcal{K}_h\mathcal{B}_{h2},\mathcal{K}_h\mathcal{W}_h)
+s_{2h}(\mathcal{B}_{h2},\mathcal{W}_h)\\
&+c_{1h}(\mathcal{U}_{h2};\mathcal{U}_{h2},\mathcal{V}_h)
+c_{2h}(\mathcal{V}_h;\mathcal{B}_{h2},\mathcal{B}_{h2})-
c_{2h}(\mathcal{U}_{h2};\mathcal{B}_{h2},\mathcal{W}_h)\nonumber\\
=&(\mathbf{f}_1,\mathcal{V}_h)
+\frac{1}{R_m}(\mathbf{f}_2,\mathcal{W}_h)
-G_{3h}(\mathcal{T}_h(\mathcal{U}_{h2}),\mathcal{V}_h).
\end{align*}
Subtracting the above first  equation from the second one   and choosing $\mathcal{V}_h=\mathcal{U}_{h1}-\mathcal{U}_{h2}$,
$\mathcal{W}_h=\mathcal{B}_{h1}-\mathcal{B}_{h2}$, we get
\begin{align*}
&\frac{1}{H_a^2}(\mathcal{G}_h(\mathcal{U}_{h1}-\mathcal{U}_{h2}),\mathcal{G}_h(\mathcal{U}_{h1}-\mathcal{U}_{h2}))
+s_{1h}(\mathcal{U}_{h1}-\mathcal{U}_{h2},\mathcal{U}_{h1}-\mathcal{U}_{h2})\\
&+\frac{1}{R_m^2}(\mathcal{K}_h(\mathcal{B}_{h1}-\mathcal{B}_{h2}),\mathcal{K}_h(\mathcal{B}_{h1}-\mathcal{B}_{h2}))
+s_{2h}(\mathcal{B}_{h1}-\mathcal{B}_{h2},\mathcal{B}_{h1}-\mathcal{B}_{h2})\\
=&-c_{1h}(\mathcal{U}_{h1};\mathcal{U}_{h1}-\mathcal{U}_{h2},\mathcal{U}_{h1}-\mathcal{U}_{h2})
-c_{1h}(\mathcal{U}_{h1}-\mathcal{U}_{h2};\mathcal{U}_{h2},\mathcal{U}_{h1}-\mathcal{U}_{h2})\\
&-c_{2h}(\mathcal{U}_{h1}-\mathcal{U}_{h2};\mathcal{B}_{h1}-\mathcal{B}_{h2};\mathcal{B}_{h1})
-c_{2h}(\mathcal{U}_{h1}-\mathcal{U}_{h2};\mathcal{B}_{h2},\mathcal{B}_{h1}-\mathcal{B}_{h2})\\
&+c_{2h}(\mathcal{U}_{h1};\mathcal{B}_{h1}-\mathcal{B}_{h2},\mathcal{B}_{h1}-\mathcal{B}_{h2})
+c_{2h}(\mathcal{U}_{h1}-\mathcal{U}_{h2};\mathcal{B}_{h2},\mathcal{B}_{h1}-\mathcal{B}_{h2})\\
&-G_{3h}((\mathcal{T}_h(\mathcal{U}_{h1})-
(\mathcal{T}_h(\mathcal{U}_{h2}),\mathcal{V}_h),
\end{align*}
which, together with  Lemma \ref{lemma13} and (\ref{T-X-UB}), leads to
\begin{align*}
&\frac{1}{H_a^2}|||\mathcal{U}_{h1}-\mathcal{U}_{h2}|||_V^2
+\frac{1}{R_m^2}|||\mathcal{B}_{h1}-\mathcal{B}_{h2}|||_W^2\\
\leq&M_{1h}|||\mathcal{U}_{h1}-\mathcal{U}_{h2}|||_V^2|||\mathcal{U}_{h2}|||_V
+M_{2h}(|||\mathcal{B}_{h1}|||_W+|||\mathcal{B}_{h2}|||_W)
|||\mathcal{B}_{h1}-\mathcal{B}_{h2}|||_W
|||\mathcal{U}_{h1}-\mathcal{U}_{h2}|||_V\\
&+M_{2h}|||\mathcal{B}_{h1}-\mathcal{B}_{h2}|||_W^2
|||\mathcal{U}_{h1}|||_V
+M_{2h}|||\mathcal{B}_{h2}|||_W|||\mathcal{B}_{h1}-\mathcal{B}_{h2}|||_W
|||\mathcal{U}_{h1}-\mathcal{U}_{h2}|||_V\\
&+M_{3h}P_r^2R_e^2\frac{G_r\mathbf{g}}{NR_eg}\|f_3\|_{3h}|||\mathcal{U}_{h1}-\mathcal{U}_{h2}|||_V^2,
\end{align*}
This estimate plus \eqref{T-3c41*} yields
\begin{align*} 
&
|||\mathcal{U}_{h1}-\mathcal{U}_{h2}|||_V
+|||\mathcal{B}_{h1}-\mathcal{B}_{h2}|||_W\nonumber\\
\leq&
2\zeta \left( H_aM_{1h}|||\mathcal{U}_{h1}-\mathcal{U}_{h2}|||_1|||\mathcal{U}_{h2}|||_V
+H_aM_{2h} \left(|||\mathcal{B}_{h1}|||_W+|||\mathcal{B}_{h2}|||_W\right)
|||\mathcal{B}_{h1}-\mathcal{B}_{h2}|||_W\right.\nonumber\\
&
\left.+R_mM_{2h}|||\mathcal{B}_{h1}-\mathcal{B}_{h2}|||_W
|||\mathcal{U}_{h1}|||_V
+R_mM_{2h}|||\mathcal{B}_{h2}|||_W|||\mathcal{U}_{h1}-\mathcal{U}_{h2}|||_V\right.\nonumber\\
& \left.+H_aM_{3h}P_r^2R_e^2\frac{G_r\mathbf{g}}{NR_eg}\|f_3\|_{3h}|||\mathcal{U}_{h1}-\mathcal{U}_{h2}|||_V\right)\nonumber\\
\leq&
4\zeta^3\left(\|\mathbf{f_1}\|_{1h}+\|\mathbf{f_2}\|_{2h}+\|f_{3}\|_{3h} )
(H_aM_{1h}|||\mathcal{U}_{h1}-\mathcal{U}_{h2}|||_V
+2H_aM_{2h}|||\mathcal{B}_{h1}-\mathcal{B}_{h2}|||_W\right.\nonumber\\
&\left.+R_mM_{2h}|||\mathcal{B}_{h1}-\mathcal{B}_{h2}|||_W
\quad +R_mM_{2h}|||\mathcal{U}_{h1}-\mathcal{U}_{h2}|||_V\right)\nonumber\\
&+2\zeta H_aM_{3h}P_r^2R_e^2\frac{G_r\mathbf{g}}{NR_eg}\|f_3\|_{3h}|||\mathcal{U}_{h1}-\mathcal{U}_{h2}|||_V\nonumber\\
\leq&
\left (4\zeta^3(\|\mathbf{f_1}\|_{1h}+\|\mathbf{f_2}\|_{2h}+\|f_{3}\|_{3h})
(H_aM_{1h}+R_mM_{2h})+2\zeta M_{3h}H_aP_r^2R_e^2\frac{G_r\mathbf{g}}{NR_e^2g}\|f_3\|_{3h}\right)
|||\mathcal{U}_{h1}-\mathcal{U}_{h2}|||_V\nonumber\\
&+4\zeta^3(\|\mathbf{f_1}\|_{1h}+\|\mathbf{f_2}\|_{2h}+\|f_{3}\|_{3h})
(2H_aM_{2h}+R_mM_{2h})
|||\mathcal{B}_{h1}-\mathcal{B}_{h2}|||_W\nonumber\\
\leq&\max\{M_{1h},M_{2h},M_{3h}\}
\left(4\zeta^3(\|\mathbf{f_1}\|_{1h}+\|\mathbf{f_2}\|_{2h}+\|f_{3}\|_{3h})
(2H_a+R_m)+2\zeta H_aP_r^2R_e^2\frac{G_r\mathbf{g}}{NR_e^2g}\|f_3\|_{3h}\right)\nonumber\\
&\times (|||\mathcal{U}_{h1}-\mathcal{U}_{h2}|||_V+|||\mathcal{B}_{h1}-\mathcal{B}_{h2}|||_W)\nonumber\\
\leq&
\max\{M_{1h},M_{2h},M_{3h}\}
\left(12\zeta^4(\|\mathbf{f_1}\|_{1h}+\|\mathbf{f_2}\|_{2h}+\|f_{3}\|_{3h})
+2\zeta^2P_rR_e\|f_3\|_{3h}\right)\nonumber\\
&
\times (|||\mathcal{U}_{h1}-\mathcal{U}_{h2}|||_V+|||\mathcal{B}_{h1}-\mathcal{B}_{h2}|||_W),
\end{align*}
which, together with   the assumption (\ref{3c81}),   implies
$$\mathcal{U}_{h1}=\mathcal{U}_{h2}, \quad \mathcal{B}_{h1}=\mathcal{B}_{h2}. $$
This completes the proof.
\end{proof}

Finally, we obtain the following existence and uniqueness results for the HDG scheme (\ref{Tscheme0101*}).

\begin{myTheo}\label{T-main-results11*}
The scheme (\ref{Tscheme0101*}) admits at least one solution $(\bm{L}_h,\mathcal{U}_h,\bm{N}_h,\mathcal{B}_h,\bm{A}_h,\mathcal{T}_h,\mathcal{P}_h,\mathcal{R}_h)\in\mathbf{D}_h\times[\mathbf{V}_h\times\hat{\mathbf{V}}_h^0]\times\mathbf{C}_h\times[\mathbf{V}_h\times\hat{\mathbf{W}}_h^0]\times\mathbf{S}_h\times[Z_h\times\hat{Z}_h^0]\times[Q_h\times\hat{Q}_h^0]\times[Q_h\times\hat{R}_h^0]$ and, under the condition
\eqref{3c81},  admits a unique solution.
\end{myTheo}
\begin{proof}
The existence and uniqueness of the discrete solutions $\mathcal{U}_h$, $\mathcal{B}_h$ and $\mathcal{T}_h$ follow from Lemma \ref{lemma15a} - Lemma \ref{results211}, and the existence and uniqueness of the discrete solution, $\mathcal{P}_h$,   to  \eqref{lemma15a-1:sub1} and the discrete solution, $\mathcal{R}_h$,   to    \eqref{lemma15a-1:sub2}  follow from  the two discrete inf-sup inequalities in Lemma \ref{lemma15}.
\end{proof}

\section{Error estimates}
This section is devoted to establishing  error estimates  for the HDG scheme (\ref{Tscheme0101*}). To this end, we assume that  the weak solution, $(\mathbf{u},\mathbf{B},T,p,r)$, to the problem (\ref{mhd1}) satisfies the following regularity conditions:
\begin{eqnarray}\label{Tregularity}
&\mathbf{u}\in\mathbf{V}\cap[H^{k+1}(\Omega)]^d,
\quad  \mathbf{B}\in  \mathbf{W}\cap[H^{k+1}(\Omega)]^d,\nonumber\\
&T \in  H_0^1\cap[H^{k+1}(\Omega)]^d,
\quad p \in L_0^2(\Omega)\cap  H^{k}(\Omega) ,
\quad r\in H_0^1(\Omega)\cap H^{k}(\Omega).
\end{eqnarray}
Here we recall that $k\geq 1$.
We set
\begin{eqnarray*}
&\Pi_1\mathbf{u}|_K:=\{\mathbf{P}_k^{\mathcal{RT}}(\mathbf{u}|_K),\mathbf{Q}_k^b(\mathbf{u}|_K)\},\quad
\Pi_2\mathbf{B}|_K:=\{\mathbf{P}_k^{\mathcal{RT}}(\mathbf{B}|_K),\mathbf{Q}_k^b(\mathbf{B}|_K)\},\\
&\Pi_3T|_K:=\{Q_k^o(T|_K),Q_k^b(T|_K)\}, \quad
\Pi_4p|_K:=\{Q_{k-1}^o(p|_K),Q_k^b(p|_K)\}, \quad
\Pi_5r|_K:=\{Q_{k-1}^o(r|_K),Q_k^b(r|_K)\},
\end{eqnarray*}
for any $ K\in\mathcal{T}_h $.

\begin{myLem}\label{lemma21}
For any $(\mathcal{V}_h,\mathcal{W}_h,\mathcal{Z}_h,\mathcal{Q}_h,\mathcal{R}_h)\in[\mathbf{V}_h\times\hat{\mathbf{V}}_h^0]\times[\mathbf{V}_h\times\hat{\mathbf{W}}_h^0]\times [Z_h\times\hat{Z}_h^0]\times [Q_h\times\hat{Q}_h^0]\times [Q_h\times\hat{R}_h^0]$,
there hold
\begin{small}
\begin{subequations}\label{d1}
\begin{align}
&\Pi_m^u\bm{L}_h-\frac{1}{H_a^2}\mathcal{G}_h\Pi_1\mathbf{u}=0,
\label{d1:sub1}\\
&\frac{1}{H_a^2}(\mathcal{G}_h\Pi_1\mathbf{u},\mathcal{G}_h\mathcal{V}_h)
+s_{1h}(\Pi_1\mathbf{u},\mathcal{V}_h)
+b_{1h}(\mathcal{V}_h,\Pi_4p)-b_{1h}(\Pi_1\mathbf{u},\mathcal{Q}_h)
+c_{1h}(\Pi_1\mathbf{u};\Pi_1\mathbf{u},\mathcal{V}_h)
+c_{2h}(\mathcal{V}_h;\Pi_2\mathbf{B},\Pi_2\mathbf{B})
\nonumber\\
=&(\mathbf{f}_1,\mathcal{V}_h)
+E_u({\mathbf{u},\mathcal{V}_h})
+E_{\tilde{u}}(\mathbf{u},\mathcal{V}_h)
+E_{\tilde{B1}}(\mathbf{B},\mathcal{V}_h)
-G_{3h}(\Pi_3\mathcal{T}_h,\mathcal{V}_h),
\label{d1:sub2}\\
&\Pi_m^B\bm{N}_h-\frac{1}{R_e^2}\mathcal{K}_h\Pi_2\mathbf{B}=0
,\label{d1:sub3}\\
&
\frac{1}{R_m}(\mathcal{K}_h\Pi_2\mathbf{B},\mathcal{K}_h\mathcal{W}_h)
+s_{2h}(\Pi_2\mathbf{B},\mathcal{W}_h) 
+b_{2h}(\mathcal{W}_h,\Pi_5r)-b_{2h}(\Pi_2\mathbf{B},\mathcal{\theta}_h)
-c_{2h}(\Pi_1\mathbf{u};\Pi_2\mathbf{B},\mathcal{W}_h)
\nonumber\\
=&
\frac{1}{R_m}(\mathbf{f}_2,\mathcal{W}_h)
+E_B({\mathbf{B},\mathcal{W}_h})
+E_{\tilde{B2}}(\mathbf{u};\mathbf{B},\mathcal{V}_h),
\label{d1:sub4}\\
&\Pi_m^T\bm{A}_h-\frac{1}{P_rR_e}\mathcal{Y}_h\Pi_3T=0,
\label{d1:sub5}\\
&
\frac{1}{P_rR_e}(\mathcal{Y}_h\Pi_3T,\mathcal{Y}_h\mathcal{Z}_h)
+s_{3h}(\Pi_3T,\mathcal{Z}_h)+c_{3h}(\Pi_1\mathbf{u};\Pi_3T,\mathcal{Z}_h)
=(f_3,\mathcal{Z}_h)+E_T(T,\mathcal{Z}_h)+E_{\tilde{T}}(\mathbf{u};T,\mathcal{Z}_h),
\label{d1:sub6}
\end{align}
\end{subequations}
\end{small}
where
\begin{align*}
&
E_u({\mathbf{u},\mathcal{V}_h}):=-\frac{1}{H_a^2}\langle  
(\hat{\mathbf{v}}_h-\mathbf{v}_h),(\Pi_m^u\bm{L}-\bm{L}) \mathbf{n}
\rangle_{\partial\mathcal{T}_h}
+\frac{1}{H_a^2}\langle h_K^{-1}(\mathbf{P}_k^{RT}\mathbf{u}-\mathbf{u}),
\mathbf{v}_h-\hat{\mathbf{v}}_h\rangle_{\partial\mathcal{T}_h},\\
&
E_B({\mathbf{B},\mathcal{W}_h}):=\frac{1}{R_m^2}\langle  
(\hat{\mathbf{w}}_h-\mathbf{w}_h),(\Pi_m^u\bm{N}-\bm{N})\times \mathbf{n}
\rangle_{\partial\mathcal{T}_h}
+\frac{1}{R_m^2}\langle h_K^{-1}(\mathbf{P}_k^{RT}\mathbf{B}-\mathbf{B})\times \mathbf{n},
(\mathbf{w}_h-\hat{\mathbf{w}}_h)\times \mathbf{n}\rangle_{\partial\mathcal{T}_h},\\
&
E_T(T,\mathcal{Z}_h):=-\frac{1}{P_rR_e}\langle  
(\hat{z}_h-z_h),(\Pi_m^T\bm{A}-\bm{A}) \mathbf{n}
\rangle_{\partial\mathcal{T}_h}
+\frac{1}{P_rR_e}\langle h_K^{-1}(Q_k^oT-T),
z_h-\hat{z}_h\rangle_{\partial\mathcal{T}_h},\\
&
E_{\tilde{u}}(\mathbf{u},\mathcal{V}_h):=
\frac{1}{2N}(\mathbf{u}\otimes\mathbf{u}-
\mathbf{P}_k^{\mathcal{RT}}\mathbf{u}\otimes\mathbf{P}_k^{\mathcal{RT}}\mathbf{u},\nabla_h\mathbf{v}_h)
-\frac{1}{2N}\langle(\mathbf{u}\otimes\mathbf{u}-
\mathbf{Q}_k^b\mathbf{u}\otimes\mathbf{Q}_k^b\mathbf{u})\mathbf{n},
\mathbf{v}_h\rangle_{\partial\mathcal{T}_h}\\
&
\quad\quad\quad\quad\quad\quad
-\frac{1}{2N}(\mathbf{u}\cdot\nabla\mathbf{u}-
\mathbf{P}_k^{\mathcal{RT}}\mathbf{u}\cdot\nabla_h\mathbf{P}_k^{\mathcal{RT}}\mathbf{u},\mathbf{v}_h)
-\frac{1}{2N}\langle\hat{\mathbf{v}}_h\otimes\mathbf{Q}_k^b\mathbf{u}\mathbf{n},
\mathbf{P}_k^{\mathcal{RT}}\mathbf{u}\rangle_{\partial\mathcal{T}_h},\\
&
E_{\tilde{B1}}(\mathbf{B},\mathcal{V}_h)
:=-\frac{1}{R_m}(\nabla_h\times(\mathbf{B}-\mathbf{P}_k^{\mathcal{RT}}\mathbf{B}),\mathbf{v}_h\times\mathbf{B})
+\frac{1}{R_m}(\nabla_h\times\mathbf{P}_k^{\mathcal{RT}}\mathbf{B},\mathbf{v}_h\times(\mathbf{P}_k^{\mathcal{RT}}\mathbf{B}-\mathbf{B}))\\
&
\quad\quad\quad\quad\quad\quad
-\frac{1}{R_m}\langle(\mathbf{P}_k^{\mathcal{RT}}\mathbf{B}-\mathbf{Q}_k^b\mathbf{B})\times\mathbf{n},
\mathbf{v}_h\times\mathbf{P}_k^{\mathcal{RT}}\mathbf{B}
\rangle_{\partial\mathcal{T}_h},\\
&
E_{\tilde{B2}}(\mathbf{u};\mathbf{B},\mathcal{W}_h)
:=-\frac{1}{R_m}(\nabla_h\times\mathbf{w}_{h},
(\mathbf{u}\times\mathbf{B}-
\mathbf{P}_k^{\mathcal{RT}}\mathbf{u}\times\mathbf{P}_k^{\mathcal{RT}}\mathbf{B}))
-\frac{1}{R_m}\langle\mathbf{w}_{h}\times\mathbf{n},\mathbf{u}\times\mathbf{B}
\rangle_{\partial\mathcal{T}_h}\\
&
\quad\quad\quad\quad\quad\quad
-\frac{1}{R_m}\langle(\mathbf{w}_{h}-\hat{\mathbf{w}}_h)\times\mathbf{n},
\mathbf{P}_k^{\mathcal{RT}}\mathbf{u}\times\mathbf{Q}_k^o\mathbf{B}
\rangle_{\partial\mathcal{T}_h},\\
&
E_{\tilde{T}}(\mathbf{u};T,\mathcal{Z}_h):=
\frac{1}{2}(\mathbf{u}T-
\mathbf{P}_k^{\mathcal{RT}}\mathbf{u}Q_k^0T,\nabla_hz_h)-
\frac{1}{2}\langle(\mathbf{u}T-
\mathbf{Q}_k^b\mathbf{u}Q_k^bT)\cdot\mathbf{n},
z_h\rangle_{\partial\mathcal{T}_h}\\
&
\quad\quad\quad\quad\quad\quad
-\frac{1}{2}(\mathbf{u}\cdot\nabla T-
\mathbf{P}_k^{\mathcal{RT}}\mathbf{u}\cdot\nabla_hQ_k^oT,z_h)-
\frac{1}{2}\langle(\mathbf{Q}_k^b\mathbf{u}\hat{z}_h)\cdot\mathbf{n},
Q_k^oT\rangle_{\partial\mathcal{T}_h}.
\end{align*}
In addition, there hold
\begin{align}\label{d2}
\mathbf{P}_k^{\mathcal{RT}}\mathbf{u}|_K\in[\mathcal{P}_k(K)]^d\quad  \mbox{and} \quad
\mathbf{P}_k^{\mathcal{RT}}\mathbf{B}|_K\in[\mathcal{P}_k(K)]^d,
\quad \forall K\in\mathcal{T}_h.
\end{align}
\end{myLem}

\begin{proof}
We first show (\ref{d2}). For all $K\in\mathcal{T}_h$, by Lemma \ref{lemma10}  we have
\begin{align*}
(\nabla\cdot\mathbf{P}_j^{\mathcal{RT}}\mathbf{u},\varphi_h)_K=(\nabla\cdot\mathbf{u},\varphi_h)_K=0,
\quad\forall \varphi_h\in \mathcal{P}_k(K),\\
(\nabla\cdot\mathbf{P}_j^{\mathcal{RT}}\mathbf{B},\mathcal{\theta}_h)_K=(\nabla\cdot\mathbf{B},\mathcal{\theta}_h)_K=0,
\quad\forall \mathcal{\theta}_h\in \mathcal{P}_k(K),
\end{align*}
which mean that
\begin{align*}
\nabla\cdot\mathbf{P}_j^{\mathcal{RT}}\mathbf{u}=0,
\quad\quad
\nabla\cdot\mathbf{P}_j^{\mathcal{RT}}\mathbf{B}=0.
\end{align*}
Then the result (\ref{d2}) follows from Lemma \ref{lemma8}.

By the simple $L^2$-projection  $\Pi_m^u$, 
 integration by parts, and the fact 
$\bm{L}=\frac{1}{H_a^2}\nabla\mathbf{u}$ we get
\begin{align*}
&a^L(\Pi_m^u\bm{L},\bm{J}_h)-	a_{1h}(\Pi_1\mathbf{u},\bm{J}_h)\\
=&{H_a^2}(\Pi_m^u\bm{L},\bm{J}_h)+
(\mathbf{P}_k^{\mathcal{RT}}\mathbf{u},\nabla_h\cdot \bm{J}_h)-\langle\mathbf{Q}_k^b\mathbf{u},\bm{J}_h\mathbf{n} \rangle_{\partial\mathcal{T}_h}\\
=&{H_a^2}(\bm{L},\bm{J}_h)+
(\mathbf{u},\nabla_h\cdot \bm{J}_h)-\langle\mathbf{u},\bm{J}_h\mathbf{n} \rangle_{\partial\mathcal{T}_h}\\
=&0.
\end{align*}
From the definitions of the bilinear forms $a_{1h}(\cdot,\cdot)$ and the fact that
$\Pi_m^u$ is simple $L^2$-projection,
 the Green's formula, the relation $\langle\nabla\mathbf{u}\ \ \mathbf{n},\hat{\mathbf{v}}_h\rangle_{\partial\mathcal{T}_h}=0$  and the definition of $E_u({\mathbf{u},\mathcal{V}_h})$, we immediately get,  for any $\mathcal{V}_h \in[\mathbf{V}_h\times\hat{\mathbf{V}}_h^0]$,
\begin{align}\label{d4}
&a_{1h}(\mathcal{V}_h,\Pi_m^u\bm{L})
+s_{1h}(\Pi_1\mathbf{u},\mathcal{V}_h)\nonumber\\
=&\frac{1}{H_a^2}(\mathbf{v}_h,\nabla_{h}\cdot\Pi_m^u\bm{L})
-\frac{1}{H_a^2}\langle  
\mathbf{v}_h,\Pi_m^u\bm{L} \mathbf{n}
\rangle_{\partial\mathcal{T}_h}
+\frac{1}{H_a^2}\langle h_K^{-1}(\mathbf{P}_k^{RT}\mathbf{u}-\mathbf{u}),
\mathbf{v}_h-\hat{\mathbf{v}}_h\rangle_{\partial\mathcal{T}_h}\nonumber\\
=&\frac{1}{H_a^2}(\mathbf{v}_h,\nabla\cdot \bm{L})
-\frac{1}{H_a^2}\langle  
(\hat{\mathbf{v}}_h-\mathbf{v}_h),(\Pi_m^uL-L) \mathbf{n}
\rangle_{\partial\mathcal{T}_h}
+\frac{1}{H_a^2}\langle h_K^{-1}(\mathbf{P}_k^{RT}\mathbf{u}-\mathbf{u}),
\mathbf{v}_h-\hat{\mathbf{v}}_h\rangle_{\partial\mathcal{T}_h}\nonumber\\
=&\frac{1}{H_a^2}(\mathbf{v}_h,\nabla\cdot \bm{L})
+E_u({\mathbf{u},\mathcal{V}_h}).
\end{align}
In view of  the definitions of $b_{1h}(\cdot,\cdot)$, the projection property, and the relations (\ref{lemma91:sub1}), (\ref{d2})  and  $\langle\mathbf{u}\cdot\mathbf{n},\hat{q}_h\rangle_{\partial\mathcal{T}_h}=0$, we get
\begin{align*}
b_{1h}(\mathcal{V}_h,\Pi_3p)-b_{1h}(\Pi_1\mathbf{u},\mathcal{Q}_h)
=&
-(Q_{k-1}^op,\nabla_{h}\cdot\mathbf{v}_h)
+\langle \hat{\mathbf{v}}_h \mathbf{n},Q_k^bp
\rangle_{\partial\mathcal{T}_h}
+(\nabla\cdot \mathbf{P}_k^{\mathcal{RT}}\mathbf{u},q_h)-
\langle \mathbf{P}_k^{\mathcal{RT}}\mathbf{u}\cdot\mathbf{n},\hat{q}_h\rangle_{\partial\mathcal{T}_h}\\
=&
(\nabla p,\mathbf{v}_h)+(\nabla\cdot \mathbf{P}_k^{\mathcal{RT}}\mathbf{u},q_h)-
\langle \mathbf{P}_k^{\mathcal{RT}}\mathbf{u}\cdot\mathbf{n},\hat{q}_h\rangle_{\partial\mathcal{T}_h}\\
=&
(\nabla p,\mathbf{v}_h)-\langle\mathbf{u}\cdot\mathbf{n},\hat{q}_h\rangle_{\partial\mathcal{T}_h}\\
=&
(\nabla p,\mathbf{v}_h), \quad \forall \mathcal{V}_h \in[\mathbf{V}_h\times\hat{\mathbf{V}}_h^0].
\end{align*}
By the Green's formula and the definitions of $c_{1h}(\cdot;\cdot,\cdot)$ and $E_{\tilde{u}}
(\cdot,\cdot)$  we    get
\begin{align*}
c_{1h}(\Pi_1\mathbf{u};\Pi_1\mathbf{u},\mathcal{V}_h)
=&
-\frac{1}{2N}(
\mathbf{P}_k^{\mathcal{RT}}\mathbf{u}\otimes\mathbf{P}_k^{\mathcal{RT}}\mathbf{u},
,\nabla_{h}\mathbf{v}_h)
+
\frac{1}{2N}\langle\mathbf{Q}_k^{b}\mathbf{u}\otimes\mathbf{Q}_k^{b}\mathbf{u} \mathbf{n}
,\mathbf{v}_h\rangle_{\partial_{\mathcal{T}_h}}\\
&-\frac{1}{2N}(
\mathbf{v}_h\otimes\mathbf{P}_k^{\mathcal{RT}}\mathbf{u},
\mathbf{P}_k^{\mathcal{RT}}\mathbf{u})
+\frac{1}{2N}\langle
\hat{\mathbf{v}}_h\otimes
\mathbf{Q}_k^{b}\mathbf{u}\mathbf{n}
,\mathbf{P}_k^{\mathcal{RT}}\mathbf{u}\rangle_{\partial_{\mathcal{T}_h}}\\
=&
\frac{1}{2N}(\nabla\cdot(\mathbf{u}\otimes\mathbf{u}),\mathbf{v}_h)+
\frac{1}{2N}(\mathbf{u}\otimes\mathbf{u}-\mathbf{P}_k^{\mathcal{RT}}\mathbf{u}\otimes\mathbf{P}_k^{\mathcal{RT}}\mathbf{u},
\nabla_h\mathbf{v}_h)\nonumber\\
& -\frac{1}{2N}\langle(\mathbf{u}\otimes\mathbf{u}-\mathbf{Q}_k^{b}\mathbf{u}\otimes\mathbf{Q}_k^{b}\mathbf{u}) \mathbf{n}
,\mathbf{v}_h\rangle_{\partial_{\mathcal{T}_h}}\\
&+
\frac{1}{2N}(\nabla\cdot(\mathbf{u}\otimes \mathbf{u}),\mathbf{v}_h)
+\frac{1}{2N}( \mathbf{u}\cdot\nabla\mathbf{u}-\mathbf{P}_k^{\mathcal{RT}} \mathbf{u}\cdot\nabla_h\mathbf{P}_k^{\mathcal{RT}}\mathbf{u},
\mathbf{v}_h)\nonumber\\
&
+\frac{1}{2N}\langle(\hat{\mathbf{v}}_h\otimes\mathbf{Q}_k^{b} \mathbf{u}\ \mathbf{n}
,\mathbf{P}_k^{\mathcal{RT}}\mathbf{u}\rangle_{\partial_{\mathcal{T}_h}}\\
=&
\frac{1}{N}(\nabla\cdot(\mathbf{u}\otimes\mathbf{u}),\mathbf{v}_h)+E_{\tilde{u}}
(\mathbf{u},\mathcal{V}_h).
\end{align*}
By the definition of the projection $Q_k^o$, we have
\begin{align*}
G_{3h}(\Pi_3\mathcal{T}_h,\mathcal{V}_h)
=
\frac{G_r}{NR_e^2}(\frac{\mathbf{g}}{g}Q_k^oT_{h},\mathbf{v}_{h})
=
\frac{G_r}{NR_e^2}(\frac{\mathbf{g}}{g}T,\mathbf{v}_{h}).
\end{align*}
Combining the above relations and \eqref{mhd1}, we finally arrive at the desired conclusion (\ref{d1:sub1})-(\ref{d1:sub2}).
Similarly, we can obtain (\ref{d1:sub3})-(\ref{d1:sub4}) and (\ref{d1:sub5})-(\ref{d1:sub6}).
This completes the proof.
\end{proof}

\begin{myLem}\label{lemma18}
For $\mathcal{V}_h\in[\mathbf{V}_h\times\hat{\mathbf{V}}_h^0]$,
$\mathcal{W}_h\in[\mathbf{V}_h\times\hat{\mathbf{W}}_h^0]$, and
$\mathcal{Z}_h\in [Z_h\times\hat{Z}_h^0]$,
there  hold
\begin{subequations}\label{lemma18-res}
\begin{align}
&|E_u(\mathbf{u},\mathcal{V}_h)|\lesssim h^k|\mathbf{u}|_{k+1}|||\mathcal{V}_h|||_V,
\label{lemma18-res:sub1}\\
&|E_B(\mathbf{B},\mathcal{W}_h)|\lesssim h_k|\mathbf{B}|_{k+1}|||\mathcal{W}_h|||_W,
\label{lemma18-res:sub2}\\
&|E_T(T,\mathcal{Z}_h)|\lesssim h^k|T|_{k+1}|||\mathcal{Z}_h|||_Z,
\label{lemma18-res:sub3}\\
&|E_{\tilde{u}}(\mathbf{u},\mathbf{v})|\lesssim h^k\|\mathbf{u}\|_2\|\mathbf{u}\|_{k+1}|||\mathcal{V}_h|||_V,
\label{lemma18-res:sub4}\\
&|E_{\tilde{B}1}(\mathbf{B},\mathcal{V}_h)|
\lesssim h^k\|\mathbf{B}\|_2\|\mathbf{B}\|_{k+1}|||\mathcal{V}_h|||_W;
\label{lemma18-res:sub5}\\
&|E_{\tilde{B}2}(\mathbf{u};\mathbf{B},\mathcal{W}_h)|
\lesssim h^k(\|\mathbf{u}\|_2\|\mathbf{B}\|_{k+1}+
\|\mathbf{B}\|_2\|\mathbf{u}\|_{k+1})|||\mathcal{W}_h|||_W,
\label{lemma18-res:sub6}\\
&|E_{\tilde{T}}(\mathbf{u};T,\mathcal{Z}_h)|\lesssim h^k(\|\mathbf{u}\|_2\|T\|_{k+1}
+\|T\|_2\|\mathbf{u}\|_{k+1})|||\mathcal{Z}_h|||_Z.
\label{lemma18-res:sub7}
\end{align}
\end{subequations}
\end{myLem}

\begin{proof}
We only show \eqref{lemma18-res:sub4}, since the other results can be derived similarly.

Let us estimate the four terms of $E_{\tilde{u}}(\mathbf{u},\mathcal{V}_h)$ one by one.
Using the Cauchy-Schwarz inequality, the  H\"{o}lder's inequality, the Sobolev embedding theorem,
and Lemmas \ref{lemma9}, \ref{lemma10} and \ref{lemma7}, we have
\begin{align*}
&\ \ \ \
|(\mathbf{u}\otimes\mathbf{u}-\mathbf{P}_k^{\mathcal{RT}}\mathbf{u}\otimes\mathbf{P}_k^{\mathcal{RT}}\mathbf{u},
\nabla_h\mathbf{v}_h)|\\
&\leq|((\mathbf{u}-\mathbf{P}_k^{\mathcal{RT}}\mathbf{u})\otimes\mathbf{u},\nabla_h\mathbf{v}_h)|
+|(\mathbf{P}_k^{\mathcal{RT}}\mathbf{u}\otimes(\mathbf{u}-\mathbf{P}_k^{\mathcal{RT}}\mathbf{u}),
\nabla_h\mathbf{v}_h)|\\
&\leq |\mathbf{u}|_{0,\infty,\Omega}
\sum_{K\in\mathcal{T}_h}|\mathbf{u}-\mathbf{P}_k^{\mathcal{RT}}\mathbf{u}|_{0,K}
\|\nabla_h\mathbf{v}_h\|_{0,K}
+\sum_{K\in\mathcal{T}_h}|\mathbf{u}-\mathbf{P}_k^{\mathcal{RT}}\mathbf{u}|_{0,3,K}
|\mathbf{P}_k^{\mathcal{RT}}\mathbf{u}|_{0,6,K}
\|\nabla_h\mathbf{v}_h\|_{0, K}\\
&\leq|\mathbf{u}|_{0,\infty,\Omega}\sum_{K\in\mathcal{T}_h}|\mathbf{u}-\mathbf{P}_k^{\mathcal{RT}}\mathbf{u}|_{0,K}
\|\nabla_h\mathbf{v}_h\|_{0,K}
+\sum_{K\in\mathcal{T}_h}|\mathbf{u}-\mathbf{P}_k^{\mathcal{RT}}\mathbf{u}|_{0,3,K}
(|\mathbf{u}-\mathbf{P}_k^{\mathcal{RT}}\mathbf{u}|_{0,6,K}+|\mathbf{u}|_{0,6,K})
\|\nabla_h\mathbf{v}_h\|_{0, K}\\
&\leq |\mathbf{u}|_{0,\infty,\Omega}\sum_{K\in\mathcal{T}_h}|\mathbf{u}-\mathbf{P}_k^{\mathcal{RT}}\mathbf{u}|_{0,K}
\|\nabla_h\mathbf{v}_h\|_{0,K}
+(|\mathbf{u}-\mathbf{P}_k^{\mathcal{RT}}\mathbf{u}|_{0,6,\Omega}+|\mathbf{u}|_{0,6,\Omega})\sum_{K\in\mathcal{T}_h}|\mathbf{u}-\mathbf{P}_k^{\mathcal{RT}}\mathbf{u}|_{0,3,K}
\|\nabla_h\mathbf{v}_h\|_{0, K}\\
&\lesssim h^{k+1}|\mathbf{u}|_{0,\infty,\Omega}
|\mathbf{u}|_{k+1}|||\mathcal{V}_h|||_V
+||\mathbf{u}||_{1}
\sum_{K\in\mathcal{T}_h}|\mathbf{u}-\mathbf{P}_k^{\mathcal{RT}}\mathbf{u}|_{0,3,K}
\|\nabla_h\mathbf{v}_h\|_{0,K}\\
&\lesssim h^{k+1}|\mathbf{u}|_{0,\infty}
|\mathbf{u}|_{k+1}|||\mathcal{V}_h|||_V
+h^{k+1-d/6}||\mathbf{u}||_{1}|\mathbf{u}|_{k+1}
||\nabla_h\mathbf{v}_h||_{0}\\
&\lesssim h^{k}\|\mathbf{u}\|_2\|\mathbf{u}\|_{k+1}|||\mathcal{V}_h|||_V.
\end{align*}

Similarly, we can obtain
\begin{align*}
&
|\langle(\mathbf{u}\otimes\mathbf{u}-\mathbf{Q}_l^{b}\mathbf{u}
\otimes\mathbf{Q}_K^{b}\mathbf{u})\ \mathbf{n}
,\mathbf{v}_h\rangle_{\partial_{\mathcal{T}_h}}|\\
=&|\langle(\mathbf{u}\otimes\mathbf{u}-\mathbf{Q}_K^{b}\mathbf{u}
\otimes\mathbf{Q}_k^{b}\mathbf{u})\ \mathbf{n}
,\mathbf{v}_h-\hat{\mathbf{v}}_h\rangle_{\partial_{\mathcal{T}_h}}|
\\
\leq&
|\langle(\mathbf{u}-\mathbf{Q}_k^{b}\mathbf{u})
\otimes(\mathbf{u}-\mathbf{Q}_k^{o}\mathbf{u}) \mathbf{n}
,\mathbf{v}_h-\hat{\mathbf{v}}_h\rangle_{\partial_{\mathcal{T}_h}}|
+|\langle(\mathbf{u}-\mathbf{Q}_k^{b}\mathbf{u})\otimes\mathbf{Q}_k^{o}\mathbf{u} \mathbf{n}
,\mathbf{v}_h-\hat{\mathbf{v}}_h\rangle_{\partial_{\mathcal{T}_h}}|
\\
&\ \ \ \
+|\langle(\mathbf{Q}_k^{o}\mathbf{u}-\mathbf{Q}_k^{b}\mathbf{u})
\otimes(\mathbf{u}-\mathbf{Q}_k^{b}\mathbf{u}) \mathbf{n}
,\mathbf{v}_h-\hat{\mathbf{v}}_h\rangle_{\partial_{\mathcal{T}_h}}|
+|\langle\mathbf{Q}_k^{o}\mathbf{u}\otimes(\mathbf{u}-\mathbf{Q}_k^{b}\mathbf{u}) \mathbf{n}
,\mathbf{v}_h-\hat{\mathbf{v}}_h\rangle_{\partial_{\mathcal{T}_h}}|
\\
\leq &
\sum_{K\in\mathcal{T}_h}\left(|\mathbf{u}-\mathbf{Q}_k^{b}\mathbf{u}|_{0,\partial K}
|\mathbf{u}-\mathbf{Q}_k^{o}\mathbf{u}|_{0,\partial K}
|\mathbf{v}_h-\hat{\mathbf{v}}_h|_{0,\infty,\partial K}
+|\mathbf{u}-\mathbf{Q}_k^{b}\mathbf{u}|_{0,\partial K}
|\mathbf{Q}_k^{o}\mathbf{u}|_{0,\infty,\partial K}
|\mathbf{v}_h-\hat{\mathbf{v}}_h|_{0,\partial K}
\right)
\\
&\
+\sum_{K\in\mathcal{T}_h}\left(|\mathbf{Q}_k^{o}\mathbf{u}-\mathbf{Q}_k^{b}\mathbf{u}|_{0,\partial K}
|\mathbf{u}-\mathbf{Q}_k^{b}\mathbf{u}|_{0,\partial K}
|\mathbf{v}_h-\hat{\mathbf{v}}_h|_{0,\infty,\partial K}
+|\mathbf{u}-\mathbf{Q}_k^{b}\mathbf{u}|_{0,\partial K}
|\mathbf{Q}_k^{o}\mathbf{u}|_{0,\infty,\partial K}
|\mathbf{v}_h-\hat{\mathbf{v}}_h|_{0,\partial K}
\right)
\\
\lesssim & h^{k}\|\mathbf{u}\|_2\|\mathbf{u}\|_{k+1}|||\mathcal{V}_h|||_V,
\\\\
&
|(\mathbf{u}\cdot\nabla\mathbf{u}-\mathbf{P}_k^{\mathcal{RT}}\mathbf{u}\cdot\nabla_h\mathbf{P}_k^{\mathcal{RT}}\mathbf{u},
\mathbf{v}_h)|
\\
\leq &
|((\mathbf{u}-\mathbf{P}_k^{\mathcal{RT}}\mathbf{u})\cdot\nabla\mathbf{u},\mathbf{v}_h)|
+|(\mathbf{P}_k^{\mathcal{RT}}\mathbf{u}
\cdot(\nabla\mathbf{u}-\nabla_h\mathbf{P}_k^{\mathcal{RT}}\mathbf{u}),\mathbf{v}_h)|
\\
\leq&
\sum_{K\in\mathcal{T}_h}|\mathbf{u}-\mathbf{P}_k^{\mathcal{RT}}\mathbf{u}|_{0,3,K}
|\nabla\mathbf{u}|_{0,K}
\|\mathbf{v}_h\|_{0,6,K}
+
\sum_{K\in\mathcal{T}_h}|\nabla\mathbf{u}-\nabla_h\mathbf{P}_k^{\mathcal{RT}}\mathbf{u}|_{0,K}
|\mathbf{P}_k^{\mathcal{RT}}\mathbf{u}|_{0,6,K}
\|\mathbf{v}_h\|_{0,3,K}
\\
\lesssim &h^{k}\|\mathbf{u}\|_2\|\mathbf{u}\|_{k+1}|||\mathcal{V}_h|||_V,
\end{align*}
and
\begin{align*}
&
|\langle\hat{\mathbf{v}}_h\otimes\mathbf{Q}_k^{b}\mathbf{u}\ \mathbf{n}
,\mathbf{P}_k^{\mathcal{RT}}\mathbf{u}\rangle_{\partial_{\mathcal{T}_h}}|
=|\langle\hat{\mathbf{v}}_h\otimes\mathbf{Q}_k^{b}\mathbf{u}\ \mathbf{n}
,\mathbf{P}_k^{\mathcal{RT}}\mathbf{u}-\mathbf{Q}_k^{b}\mathbf{u}\rangle_{\partial_{\mathcal{T}_h}}|\\
\leq &
|\langle(\mathbf{v}_h-\hat{\mathbf{v}}_h)\otimes(\mathbf{Q}_k^{b}\mathbf{u}-
\mathbf{Q}_k^{o}\mathbf{u}) \ \mathbf{n}
,\mathbf{P}_k^{\mathcal{RT}}\mathbf{u}-\mathbf{Q}_k^{b}\mathbf{u}\rangle_{\partial_{\mathcal{T}_h}}|
+|\langle\mathbf{v}_h\otimes
(\mathbf{Q}_k^{b}\mathbf{u}-
\mathbf{Q}_k^{o}\mathbf{u})\ \mathbf{n}
,\mathbf{P}_k^{\mathcal{RT}}\mathbf{u}-\mathbf{Q}_k^{b}\mathbf{u}\rangle_{\partial_{\mathcal{T}_h}}|
\\
&\ \ \ \
+|\langle(\mathbf{v}_h-\hat{\mathbf{v}}_h)\otimes\mathbf{Q}_k^{o}\mathbf{u}\ \mathbf{n}
,\mathbf{P}_k^{\mathcal{RT}}\mathbf{u}-\mathbf{Q}_k^{b}\mathbf{u}\rangle_{\partial_{\mathcal{T}_h}}|
+|\langle\mathbf{v}_h\otimes\mathbf{Q}_k^{o}\mathbf{u}\ \mathbf{n}
,\mathbf{P}_k^{\mathcal{RT}}\mathbf{u}-\mathbf{Q}_k^{b}\mathbf{u}\rangle_{\partial_{\mathcal{T}_h}}|
\\
\lesssim & h^{k}\|\mathbf{u}\|_2\|\mathbf{u}\|_{k+1}|||\mathcal{V}_h|||_V.
\end{align*}
As a result,  the desired estimate \eqref{lemma18-res:sub4} follows.
\end{proof}

\begin{myTheo}\label{estimates1}
Let $(\mathcal{U}_h,\mathcal{B}_h,\mathcal{T}_h,\mathcal{P}_h,\mathcal{R}_h)\in
[\mathbf{V}_h\times\hat{\mathbf{V}}_h^0]\times[\mathbf{V}_h\times\hat{\mathbf{W}}_h^0]\times [Z_h\times\hat{Z}_h^0]\times[Q_h\times\hat{Q}_h^0]\times [Q_h\times\hat{R}_h^0]$ be the solutions to the HDG scheme (\ref{Tscheme0101*}).
Under the regularity assumption \eqref{Tregularity} and the smallness condition \eqref{results211},
there hold the following estimates:
\begin{subequations}\label{estimates1res}
\begin{align}
&|||\Pi_1\mathbf{u}-\mathcal{U}_h|||_V
+|||\Pi_2\mathbf{B}-\mathcal{B}_h|||_W
+|||\Pi_3T-\mathcal{T}_h|||_Z
\lesssim h^k C_1,
\label{estimates1res:sub1}\\
&|||\Pi_4p-\mathcal{P}_h|||_Q+|||\Pi_5r-\mathcal{R}_h|||_R\lesssim h^k C_1(\mathbf{u},\mathbf{B},T)
+h^{2k} C_2(\mathbf{u},\mathbf{B}),
\label{estimates1res:sub2}
\end{align}
\end{subequations}
where
\begin{align*}
&C_1(\mathbf{u},\mathbf{B},T):=(||\mathbf{u}||_{k+1}+||\mathbf{B}||_{k+1}+||T||_{k+1})
(1+||\mathbf{u}||_2+||\mathbf{B}||_2+||T||_2),\\
&C_2(\mathbf{u},\mathbf{B}):=(||\mathbf{u}||_{k+1}+||\mathbf{B}||_{k+1})^2
(1+||\mathbf{u}||_2+||\mathbf{B}||_2)^2.
\end{align*}
\end{myTheo}

\begin{proof}
From   (\ref{Tscheme0101*}) and   Lemma \ref{lemma21} we   get the error equations as follows:
\begin{subequations}\label{estimate1}
\begin{align}
&\frac{1}{H_a^2}(\mathcal{G}_h(\Pi_1\mathbf{u}-\mathcal{U}_h),\mathcal{G}_h\mathcal{V}_h)
+
\frac{1}{R_m^2}(\mathcal{K}_h(\Pi_2\mathbf{B}-\mathcal{B}_h),\mathcal{K}_h\mathcal{W}_h)\nonumber\\
&+s_{1h}(\Pi_1\mathbf{u}-\mathcal{U}_h,\mathcal{V}_h)
+s_{2h}(\Pi_2\mathbf{B}-\mathcal{B}_h,\mathcal{W}_h)
+b_{1h}(\mathcal{V}_h,\Pi_4p-\mathcal{P}_h)-b_{1h}(\Pi_1\mathbf{u}-\mathcal{U}_h,\mathcal{Q}_h)\nonumber\\
&
+b_{2h}(\mathcal{W}_h,\Pi_5r-\mathcal{R}_h)-b_{2h}(\Pi_2\mathbf{B}-\mathcal{B}_h,\mathcal{\theta}_h)
+c_{1h}(\Pi_1\mathbf{u};\Pi_1\mathbf{u},\mathcal{V}_h)
-c_{1h}(\mathcal{U}_h;\mathcal{U}_h,\mathcal{V}_h)\nonumber\\
&
+c_{2h}(\mathcal{V}_h;\Pi_2\mathbf{B},\Pi_2\mathbf{B})
-c_{2h}(\mathcal{V}_h;\mathcal{B}_h,\mathcal{B}_h)
-c_{2h}(\Pi_1\mathbf{u};\Pi_2\mathbf{B},\mathcal{W}_h)
+c_{2h}(\mathcal{U}_h;\mathcal{B}_h,\mathcal{W}_h)
+G_{3h}(\Pi_3T-\mathcal{T}_h,\mathcal{V}_h)\nonumber\\
=&
E_u({\mathbf{u},\mathcal{V}_h})+E_B({\mathbf{B},\mathcal{W}_h})
+E_{\tilde{u}}(\mathbf{u},\mathcal{V}_h)
+E_{\tilde{B}1}(\mathcal{B}_h,\Pi_1\mathbf{u}-\mathcal{U}_h)
+E_{\tilde{B}2}(\mathbf{u};\mathbf{B},\Pi_2\mathbf{B}-\mathcal{B}_h),\nonumber\\
&
\quad\quad
\forall(\mathcal{V}_h,\mathcal{W}_h,\mathcal{Q}_h,\mathcal{R}_h)\in[\mathbf{V}_h\times\hat{\mathbf{V}}_h^0]\times[\mathbf{V}_h\times\hat{\mathbf{W}}_h^0]\times [Q_h\times\hat{Q}_h^0]\times [Q_h\times\hat{R}_h^0],
\label{estimate1:sub1}\\
\nonumber\\
&\frac{1}{P_rR_e}(\mathcal{Y}_h(\Pi_3T-\mathcal{T}_h),\mathcal{Y}_h\mathcal{Z}_h)
+s_{3h}(\Pi_3T-\mathcal{T}_h,\mathcal{Z}_h)+c_{3h}(\Pi_1\mathbf{u};\Pi_3T,\mathcal{Z}_h)
-c_{3h}(\mathcal{U}_h;\mathcal{T}_h,\mathcal{Z}_h)\nonumber\\
=&
E_T(T,\mathcal{Z}_h)+E_{\tilde{T}}(\mathbf{u};T,\mathcal{Z}_h), \quad \forall \mathcal{Z}_h\in [Z_h\times\hat{Z}_h^0].
\label{estimate1:sub2}
\end{align}
\end{subequations}
Taking $(\mathcal{V}_h,\mathcal{W}_h,\mathcal{Z}_h,\mathcal{Q}_h,\mathcal{\theta}_h)
=(\Pi_1\mathbf{u}-\mathcal{U}_h,\Pi_2\mathbf{B}-\mathcal{B}_h,\Pi_3T-\mathcal{T}_h,
\Pi_4p-\mathcal{P}_h,\Pi_5r-\mathcal{R}_h)$ in (\ref{estimate1}) and
using the relation $c_{1h}(\Pi_1\mathbf{u};\Pi_1\mathbf{u}-\mathcal{U}_h,\Pi_1\mathbf{u}-\mathcal{U}_h)=0$ and
$c_{3h}(\Pi_3T;\Pi_3T-\mathcal{T}_h,\Pi_3T-\mathcal{T}_h)=0$,
we obtain
\begin{subequations}\label{estimate2}
\begin{align}
&\frac{1}{H_a^2}(\mathcal{G}_h(\Pi_1\mathbf{u}-\mathcal{U}_h),\mathcal{G}_h(\Pi_1\mathbf{u}-\mathcal{U}_h))
+
\frac{1}{R_m^2}(\mathcal{K}_h(\Pi_2\mathbf{B}-\mathcal{B}_h),\mathcal{K}_h(\Pi_2\mathbf{B}-\mathcal{B}_h))\nonumber\\
&
+s_{1h}(\Pi_1\mathbf{u}-\mathcal{U}_h,\Pi_1\mathbf{u}-\mathcal{U}_h)
+s_{2h}(\Pi_2\mathbf{B}-\mathcal{B}_h,\Pi_2\mathbf{B}-\mathcal{B}_h)\nonumber\\
=&
E_u(\mathbf{u},\Pi_1\mathbf{u}-\mathcal{U}_h)+E_B({\mathbf{B},\Pi_2\mathbf{B}-\mathcal{B}_h})
+E_{\tilde{u}}({\mathbf{u},\Pi_1\mathbf{u}-\mathcal{U}_h})
+E_{\tilde{B}1}(\mathcal{B}_h,\Pi_1\mathbf{u}-\mathcal{U}_h)
+E_{\tilde{B}2}(\mathbf{u};\mathbf{B},\Pi_2\mathbf{B}-\mathcal{B}_h)\nonumber\\
&
-c_{1h}(\Pi_1\mathbf{u};\Pi_1\mathbf{u},\Pi_1\mathbf{u}-\mathcal{U}_h)
+c_{1h}(\mathcal{U}_h;\mathcal{U}_h,\Pi_1\mathbf{u}-\mathcal{U}_h)
-c_{2h}(\Pi_1\mathbf{u}-\mathcal{U}_h;\Pi_2\mathbf{B},\Pi_2\mathbf{B})
+c_{2h}(\Pi_1\mathbf{u}-\mathcal{U}_h;\mathcal{B}_h,\mathcal{B}_h)\nonumber\\
&
+c_{2h}(\Pi_1\mathbf{u};\Pi_2\mathbf{B},\Pi_2\mathbf{B}-\mathcal{B}_h)
-c_{2h}(\mathcal{U}_h;\mathcal{B}_h,\Pi_2\mathbf{B}-\mathcal{B}_h)
-G_{3h}(\Pi_3T-\mathcal{T}_h,\Pi_1\mathbf{u}-\mathcal{U}_h)\nonumber\\
=&
E_u(\mathbf{u},\Pi_1\mathbf{u}-\mathcal{U}_h)+E_B({\mathbf{B},\Pi_2\mathbf{B}-\mathcal{B}_h})
+E_{\tilde{u}}({\mathbf{u};\mathbf{u},\Pi_1\mathbf{u}-\mathcal{U}_h})
+E_{\tilde{B}1}(\mathcal{B}_h,\Pi_1\mathbf{u}-\mathcal{U}_h)
+E_{\tilde{B}2}(\mathbf{u};\mathbf{B},\Pi_2\mathbf{B}-\mathcal{B}_h)\nonumber\\
&
-c_{1h}(\Pi_1\mathbf{u}-\mathcal{U}_h;\mathcal{U}_h,\Pi_1\mathbf{u}-\mathcal{U}_h)
-c_{2h}(\mathcal{U}_h;\Pi_2\mathbf{B}-\mathcal{B}_h,\Pi_2\mathbf{B}-\mathcal{B}_h)
-c_{2h}(\Pi_1\mathbf{u}-\mathcal{U}_h;\Pi_2\mathbf{B}-\mathcal{B}_h,\mathcal{B}_h)\nonumber\\
&-G_{3h}(\Pi_3T-\mathcal{T}_h,\Pi_1\mathbf{u}-\mathcal{U}_h),
\label{estimate2:sub1}\\
\nonumber\\
&\frac{1}{P_rR_e}(\mathcal{Y}_h(\Pi_3T-\mathcal{T}_h),\mathcal{Y}_h(\Pi_3T-\mathcal{T}_h))
+s_{3h}(\Pi_3T-\mathcal{T}_h,\Pi_3T-\mathcal{T}_h)\nonumber\\
=&
E_T(T,\Pi_3T-\mathcal{T}_h)+E_{\tilde{T}}(\mathbf{u};T,\Pi_3T-\mathcal{T}_h)
-c_{3h}(\Pi_1\mathbf{u};\Pi_3T,\Pi_3T-\mathcal{T}_h)
+c_{3h}(\mathcal{U}_h;\mathcal{T}_h,\Pi_3T-\mathcal{T}_h)\nonumber\\
=&
E_T(T,\Pi_3T-\mathcal{T}_h)+E_{\tilde{T}}(\mathbf{u};T,\Pi_3T-\mathcal{T}_h)
-c_{3h}(\Pi_1\mathbf{u}-\mathcal{U}_h;\mathcal{T}_h,\Pi_3T-\mathcal{T}_h).
\label{estimate2:sub2}
\end{align}
\end{subequations}
In view of Lemmas \ref{lemma13}, \ref{lemma18}, and the definitions of $M_{ih}(i=1,2,3)$ in \eqref{M1h}-\eqref{M3h}, we further have
\begin{subequations}\label{estimate22}
\begin{align}
&
|||\Pi_1\mathbf{u}-\mathcal{U}_h|||_V
+|||\Pi_2\mathbf{B}-\mathcal{B}_h|||_W\nonumber\\
\leq &
2\zeta C\left( h^k||\mathbf{u}||_{k+1}
+ h^k||\mathbf{B}||_{k+1}
+ h^k\|\mathbf{B}\|_2||\mathbf{B}||_{k+1} + h^k\|\mathbf{u}\|_2||\mathbf{u}||_{k+1}
+h^k\|\mathbf{u}\|_2||\mathbf{B}||_{k+1}+h^k\|\mathbf{B}\|_2||\mathbf{u}||_{k+1}\right)
\nonumber\\
&+2\zeta \left( H_aM_{1h}|||\mathcal{U}_h|||_V|||\Pi_1\mathbf{u}-\mathcal{U}_h|||_V
+R_mM_{2h}|||\mathcal{U}_h|||_V|||\Pi_2\mathbf{B}-\mathcal{B}_h|||_W\right.\nonumber\\
&
\quad +\left.\frac{1}{2}H_aM_{2h}|||\mathcal{B}_h|||_W|||\Pi_1\mathbf{u}-\mathcal{U}_h|||_V
+\frac{1}{2}R_mM_{2h}|||\mathcal{B}_h|||_W|||\Pi_2\mathbf{B}-\mathcal{B}_h|||_W\right)\nonumber\\
&+2\zeta H_aM_{3h}\frac{G_r\mathbf{g}}{NR_eg}|||\Pi_3T-\mathcal{T}_h|||_{Z},
\label{estimate22:sub1}\\
\nonumber\\
&|||\Pi_3T-\mathcal{T}_h|||_Z\nonumber\\
\leq&
C(h^k|T|_{k+1}+h^k\|\mathbf{u}\|_2\|T\|_{k+1}
+h^k\|T\|_2\|\mathbf{u}\|_{k+1})
+M_{3h}|||\Pi_1\mathbf{u}-\mathcal{U}_h|||_1|||\mathcal{T}_h|||_Z\nonumber\\
\leq &C(h^k|T|_{k+1}+h^k\|\mathbf{u}\|_2\|T\|_{k+1}
+h^k\|T\|_2\|\mathbf{u}\|_{k+1})
+M_{3h}P_r^2R_e^2\|f_3\|_{3h}|||\Pi_1\mathbf{u}-\mathcal{U}_h|||_V.
\label{estimate22:sub2}
\end{align}
\end{subequations}
Taking \eqref{estimate22:sub2} in \eqref{estimate22:sub1} we get
\begin{align*}
&
|||\Pi_1\mathbf{u}-\mathcal{U}_h|||_V
+|||\Pi_2\mathbf{B}-\mathcal{B}_h|||_W\nonumber\\
\leq &
2\zeta C\left( h^k||\mathbf{u}||_{k+1}
+ h^k||\mathbf{B}||_{k+1}
+ h^k\|\mathbf{B}\|_2||\mathbf{B}||_{k+1} + h^k\|\mathbf{u}\|_2||\mathbf{u}||_{k+1}
+h^k\|\mathbf{u}\|_2||\mathbf{B}||_{k+1}+h^k\|\mathbf{B}\|_2||\mathbf{u}||_{k+1}\right)
\nonumber\\
&+2\zeta \left( H_aM_{1h}|||\mathcal{U}_h|||_V|||\Pi_1\mathbf{u}-\mathcal{U}_h|||_V
+R_mM_{2h}|||\mathcal{U}_h|||_V|||\Pi_2\mathbf{B}-\mathcal{B}_h|||_W\right.\nonumber\\
&
\quad +\left.\frac{1}{2}H_aM_{2h}|||\mathcal{B}_h|||_W|||\Pi_1\mathbf{u}-\mathcal{U}_h|||_V
+\frac{1}{2}R_mM_{2h}|||\mathcal{B}_h|||_W|||\Pi_2\mathbf{B}-\mathcal{B}_h|||_W\right)\nonumber\\
&+C(h^k|T|_{k+1}+h^k\|\mathbf{u}\|_2\|T\|_{k+1}
+h^k\|T\|_2\|\mathbf{u}\|_{k+1})
+2\zeta H_aM_{3h}P_r^2R_e^2\frac{G_r\mathbf{g}}{NR_eg}\|f_3\|_{3h}|||\Pi_1\mathbf{u}-\mathcal{U}_h|||_V\nonumber\\
\leq &
2\zeta C\left( h^k||\mathbf{u}||_{k+1}
+ h^k||\mathbf{B}||_{k+1}
+ h^k\|\mathbf{B}\|_2||\mathbf{B}||_{k+1} + h^k\|\mathbf{u}\|_2||\mathbf{u}||_{k+1}
+h^k\|\mathbf{u}\|_2||\mathbf{B}||_{k+1}+h^k\|\mathbf{B}\|_2||\mathbf{u}||_{k+1}\right)
\nonumber\\
&+4\zeta^3\left(\|\mathbf{f_1}\|_{1h}+\|\mathbf{f_2}\|_{2h}+\|f_{3}\|_{3h} )
(H_aM_{1h}|||\Pi_1\mathbf{u}-\mathcal{U}_h|||_V
+R_mM_{2h}|||\Pi_2\mathbf{B}-\mathcal{B}_h|||_W\right.\nonumber\\
&\left.+\frac{1}{2}H_aM_{2h}|||\Pi_1\mathbf{u}-\mathcal{U}_h|||_V
+\frac{1}{2}R_mM_{2h}|||\Pi_2\mathbf{B}-\mathcal{B}_h|||_W
\right)\nonumber\\
&+C(h^k|T|_{k+1}+h^k\|\mathbf{u}\|_2\|T\|_{k+1}
+h^k\|T\|_2\|\mathbf{u}\|_{k+1})
+2\zeta H_aM_{3h}P_r^2R_e^2\frac{G_r\mathbf{g}}{NR_eg}\|f_3\|_{3h}|||\Pi_1\mathbf{u}-\mathcal{U}_h|||_V\nonumber\\
\leq&
\max\{M_{1h},M_{2h},M_{3h}\}
\left(12\zeta^4(\|\mathbf{f_1}\|_{1h}+\|\mathbf{f_2}\|_{2h}+\|f_{3}\|_{3h})
+2\zeta^2P_rR_e\|f_3\|_{3h}\right)\nonumber\\
&\times (|||\Pi_1\mathbf{u}-\mathcal{U}_h|||_V+|||\Pi_2\mathbf{B}-\mathcal{B}_h|||_W)\nonumber\\
&+2\zeta C\left( h^k||\mathbf{u}||_{k+1}
+ h^k||\mathbf{B}||_{k+1}
+ h^k\|\mathbf{B}\|_2||\mathbf{B}||_{k+1} + h^k\|\mathbf{u}\|_2||\mathbf{u}||_{k+1}
+h^k\|\mathbf{u}\|_2||\mathbf{B}||_{k+1}+h^k\|\mathbf{B}\|_2||\mathbf{u}||_{k+1}\right)
\nonumber\\
&+C(h^k|T|_{k+1}+h^k\|\mathbf{u}\|_2\|T\|_{k+1}
+h^k\|T\|_2\|\mathbf{u}\|_{k+1}),
\end{align*}
which plus the smallness condition \eqref{3c81}
yields
\begin{align*}
|||\Pi_1\mathbf{u}-\mathcal{U}_h|||_V
+|||\Pi_2\mathbf{B}-\mathcal{B}_h|||_W
\lesssim h^k C_1(\mathbf{u},\mathbf{B},T).
\end{align*}
Hence, combining this   estimate with \eqref{estimate22:sub2} leads to the desired estimate \eqref{estimates1res:sub1}.

Next let us estimate the pressure error.
Taking $(\mathcal{W}_h,\mathcal{Q}_h,\mathcal{R}_h)=(0,0,0)$ in  the equation \eqref{d1:sub1},   we have
\begin{align*}
&\frac{1}{H_a^2}(\mathcal{G}_h(\Pi_1\mathbf{u}-\mathcal{U}_h),\mathcal{G}_h\mathcal{V}_h)
+s_{1h}(\Pi_1\mathbf{u},\mathcal{V}_h)
+b_{1h}(\mathcal{V}_h,\Pi_3p)
+c_{1h}(\Pi_1\mathbf{u};\Pi_1\mathbf{u},\mathcal{V}_h)
+c_{2h}(\mathcal{V}_h;\Pi_2\mathbf{B},\Pi_2\mathbf{B})\\
=&(\mathbf{f_1},\mathbf{v}_h)
+E_u({\mathbf{u},\mathcal{V}_h})
+E_{\tilde{u}}( \mathbf{u},\mathcal{V}_h)-
G_{3h}(\Pi_3T,\mathcal{V}_h) ,
\end{align*}
which, together with   \eqref{Tscheme0101*-a},   gives
\begin{align*}
&b_{1h}(\mathcal{V}_h,\Pi_3p-\mathcal{P}_h)\\ =&E_u({\mathbf{u},\mathcal{V}_h})
+E_{\tilde{u}}(\mathbf{u},\mathcal{V}_h)-G_{3h}(\Pi_3T-\mathcal{T}_h,\mathcal{V}_h)
-
\frac{1}{H_a^2}(\mathcal{G}_h(\Pi_1\mathbf{u}-\mathcal{U}_h),\mathcal{G}_h\mathcal{V}_h)
-s_{1h}(\Pi_1\mathbf{u}-\mathcal{U}_h,\mathcal{V}_h)\\
&-c_{1h}(\Pi_1\mathbf{u};\Pi_1\mathbf{u},\mathcal{V}_h)
+c_{1h}(\mathcal{U}_h;\mathcal{U}_h,\mathcal{V}_h)
-c_{2h}(\mathcal{V}_h;\Pi_2\mathbf{B},\Pi_2\mathbf{B})
+c_{2h}(\mathcal{V}_h;\mathcal{B}_h,\mathcal{B}_h)\\
=&E_u({\mathbf{u},\mathcal{V}_h})
+E_{\tilde{u}}(\mathbf{u},\mathcal{V}_h)-G_{3h}(\Pi_3T-\mathcal{T}_h,\mathcal{V}_h)
-
\frac{1}{H_a^2}(\mathcal{G}_h(\Pi_1\mathbf{u}-\mathcal{U}_h),\mathcal{G}_h\mathcal{V}_h)
-s_{1h}(\Pi_1\mathbf{u}-\mathcal{U}_h,\mathcal{V}_h)\\
& -c_{1h}(\Pi_1\mathbf{u}-\mathcal{U}_h;\Pi_1\mathbf{u}-\mathcal{U}_h,\mathcal{V}_h)-c_{1h}(\mathcal{U}_h;\Pi_1\mathbf{u}-\mathcal{U}_h,\mathcal{V}_h)
-c_{1h}(\Pi_1\mathbf{u}-\mathcal{U}_h;\mathcal{U}_h,\mathcal{V}_h)\\
&
-c_{2h}(\mathcal{V}_h;\Pi_2\mathbf{B}-\mathcal{B}_h,\Pi_2\mathbf{B}-\mathcal{B}_h)
-c_{2h}(\mathcal{V}_h;\mathcal{B}_h,\Pi_2\mathbf{B}-\mathcal{B}_h)
-c_{2h}(\mathcal{V}_h;\Pi_2\mathbf{B}-\mathcal{B}_h,\mathcal{B}_h).
\end{align*}
Thus, using the inf-sup condition (\ref{T-inf-sup-bh}), Lemmas \ref{lemma13} and  \ref{lemma18}, and the estimate \eqref{estimates1res:sub1}, we get
\begin{align*}
|||\Pi_3p-\mathcal{P}_h|||_Q\lesssim&\sup_{0\neq\mathcal{V}_h\in\mathcal{V}_h^0}
\frac{b_{1h}(\mathcal{V}_h,\Pi_3p-\mathcal{P}_h)}{|||\mathcal{V}_h|||_V}\\
\lesssim&   h^k \left(\|\mathbf{u}\|_{k+1} +\|\mathbf{B}\|_{k+1}+\|T\|_{k+1}\right)\left(1+||\mathbf{u}||_2+||\mathbf{B}||_2+||T||_2\right)\\
&+h^{2k} \left(\|\mathbf{u}\|_{k+1} +\|\mathbf{B}\|_{k+1}\right)^2 \left(1+||\mathbf{u}||_2+||\mathbf{B}||_2\right)^2\\
\lesssim& h^k C_1(\mathbf{u},\mathbf{B},T)
+h^{2k}C_2(\mathbf{u},\mathbf{B}) .
\end{align*}
Similarly,
by using the inf-sup condition (\ref{T-inf-sup-tildebh}), Lemmas \ref{lemma13} and  \ref{lemma18}, and  \eqref{estimates1res:sub1}, we can obtain
$$ |||\Pi_4r-\mathcal{R}_h|||_R\lesssim h^k C_1(\mathbf{u},\mathbf{B},T)
+h^{2k}C_2(\mathbf{u},\mathbf{B}) .$$
Combining the above two inequalities leads to the desired result \eqref{estimates1res:sub2}. This finishes the proof.
\end{proof}

In light of  Theorem \ref{estimates1},  Lemmas \ref{lemma3}, \ref{T-lemma1*}, \ref{lemma9} and \ref{lemma10}, and  the triangle inequality,   we can finally obtain the following main error estimates.
\begin{myTheo}\label{estimates2}
Under the same conditions of Theorem \ref{estimates1}, there hold
\begin{subequations}\label{estimates2res}
\begin{align}
&
\|\nabla\mathbf{u}-\nabla_h\mathbf{u}_{h}\|_0+H_a^2\|\bm{L}-\bm{L}_h\|_0
\lesssim h^kC_1(\mathbf{u},\mathbf{B},T),
\label{estimates2res:sub1}\\
&
\|\nabla\times\mathbf{B}-\nabla_h\times\mathbf{B}_{h}\|_0
+R_m^2\|\bm{N}-\bm{N}_h\|_0
\lesssim h^kC_1(\mathbf{u},\mathbf{B},T),
\label{estimates2res:sub2}\\
&\|\nabla
T-\nabla_hT_{h}\|_0+
P_rR_e\|\bm{A}-\bm{A}_h\|_0
\lesssim h^kC_1(\mathbf{u},\mathbf{B},T),
\label{estimates2res:sub3}\\
&\|p-p_{h}\|_0
\lesssim h^kC_1(\mathbf{u},\mathbf{B},T)+h^k\|p\|_k+h^{2k}C_2(\mathbf{u},\mathbf{B}),
\label{estimates2res:sub4}\\
&\|r-\bar{r}_h-(\bar r -\bar{r}_h)\|_0
\lesssim h^kC_1(\mathbf{u},\mathbf{B},T)+h^k\|r\|_k+h^{2k}C_2(\mathbf{u},\mathbf{B}),
\label{estimates2res:sub5}
\end{align}
\end{subequations}
where $\bar r$ and $\bar{r}_h$ denote the mean values of $r$ and $\bar{r}_h$ on $\Omega$, respectively.
\end{myTheo}
\begin{rem}
From the estimates \eqref{estimates2res:sub1} and \eqref{estimates2res:sub2} we see that the upper bounds of the errors of the velocity  and the magnetic field are  independent of the
approximations of the	pressure and  the magnetic pseudo-pressure. This means that our HDG scheme is pressure-robust.
\end{rem}

\section{Numerical examples}

In this section, we give a  2D  numerical example and a 3D example   to verify the performance
of the HDG scheme  (\ref{Tscheme0101*}) for the steady incompressible MHD flow \eqref{mhd1}.

Since  (\ref{Tscheme0101*}) is a nonlinear scheme, we apply the  following Oseen iterative algorithm to solve it:   
given the initial guess
$(\mathbf{u}_{h}^0,\mathbf{B}_{h}^0)=(0,0)$,
find
$(\bm{L}_h^{n},\mathcal{U}_h^{n},\bm{N}_h^{n},\mathcal{B}_h^{n},\bm{A} _h^{n},\mathcal{T}_h^{n},\mathcal{P}_h^{n},\mathcal{R}_h^{n})\in\mathbf{D}_h\times[\mathbf{V}_h\times\hat{\mathbf{V}}_h^0]\times\mathbf{C}_h\times[\mathbf{V}_h\times\hat{\mathbf{W}}_h^0]\times\mathbf{S}_h\times[Z_h\times\hat{Z}_h^0]\times[Q_h\times\hat{Q}_h^0]\times[Q_h\times\hat{R}_h^0]$  for $n=1,2, \cdots$,
such that
\begin{small}
 \begin{subequations}\label{scheme0101*}
  \begin{align}
   a^L(\bm{L}_h^{n},\bm{J}_h)- a_{1h}(\mathcal{U}_h^{n},\bm{J}_h)& =0,
   \label{scheme0101*-a}\\
   a_{1h}(\mathcal{V}_h,\bm{L}_h^{n})
   +s_{1h}(\mathcal{U}_h^{n},\mathcal{V}_h)
   +b_{1h}(\mathcal{V}_h,\mathcal{P}_h^{n})
   +c_{1h}(\mathcal{U}_h^{n-1};\mathcal{U}_h^{n},\mathcal{V}_h)
   +c_{2h}(\mathcal{V}_h;\mathcal{B}_h^{n-1},\mathcal{B}_h^{n})
   &=(\mathbf{f}_1,\mathbf{v}_h)- G_{3h}(T_{h}^{n},\mathbf{v}_{h}), \label{scheme0101*-b}\\
   b_{1h}(\mathcal{U}_h^{n},\mathcal{Q}_h)&=0
   , \label{scheme0101*-c}\\
   a^N(\bm{N}_h^{n},\bm{I}_h)- a_{2h}(\mathcal{B}_h^{n},\bm{I}_h)& =0,
   \label{scheme0101*-d}\\
   a_{2h}(\mathcal{W}_h,\bm{N}_h^{n})
   +
   s_{2h}(\mathcal{B}_h^{n},\mathcal{W}_h)
   +b_{2h}(\mathcal{W}_h,\mathcal{R}_h^{n})
   -c_{2h}(\mathcal{U}_h^{n-1};\mathcal{B}_h^{n},\mathcal{W}_h)
   &= \frac{1}{R_m}(\mathbf{f}_2,\mathbf{w}_{h}), \label{scheme0101*-e}\\
   b_{2h}(\mathcal{B}_h^{n},\mathcal{\theta}_h)&=0, \quad \label{scheme0101*-f}\\
   a^A(\bm{A}_h^{n},\bm{E}_h)- a_{3h}(\mathcal{T}_h^{n},\bm{E}_h)& =0,
   \label{scheme0101*-g}\\
   a_{3h}(\mathcal{Z}_h,\bm{A}_h^{n})
   +s_{3h}(\mathcal{T}_h^{n},\mathcal{Z}_h)+c_{3h}(\mathcal{U}_h^{n-1};\mathcal{T}_h^{n},\mathcal{Z}_h)
   &=(f_3,z_h),\label{scheme0101*-h}
  \end{align}
 \end{subequations}
\end{small} 
 for all 
$(\bm{J}_h,\mathcal{V}_h,I_h,\mathcal{W}_h,E_h,\mathcal{Z}_h,\mathcal{Q}_h,\mathcal{\theta}_h)\in\mathbf{D}_h\times[\mathbf{V}_h\times\hat{\mathbf{V}}_h^0]\times\mathbf{C}_h\times[\mathbf{V}_h\times\hat{\mathbf{W}}_h^0]\times\mathbf{S}_h\times[Z_h\times\hat{Z}_h^0]\times[Q_h\times\hat{Q}_h^0]\times[Q_h\times\hat{R}_h^0]$.
In all numerical experiments we adopt the stop criterion
$$\|\mathcal{U}_h^n-\mathcal{U}_h^{n-1}\|_0< 1e-8 $$.

\begin{figure}[htbp]
\centering
\subfigure[]
{\includegraphics[height=3.5cm,width=5cm]{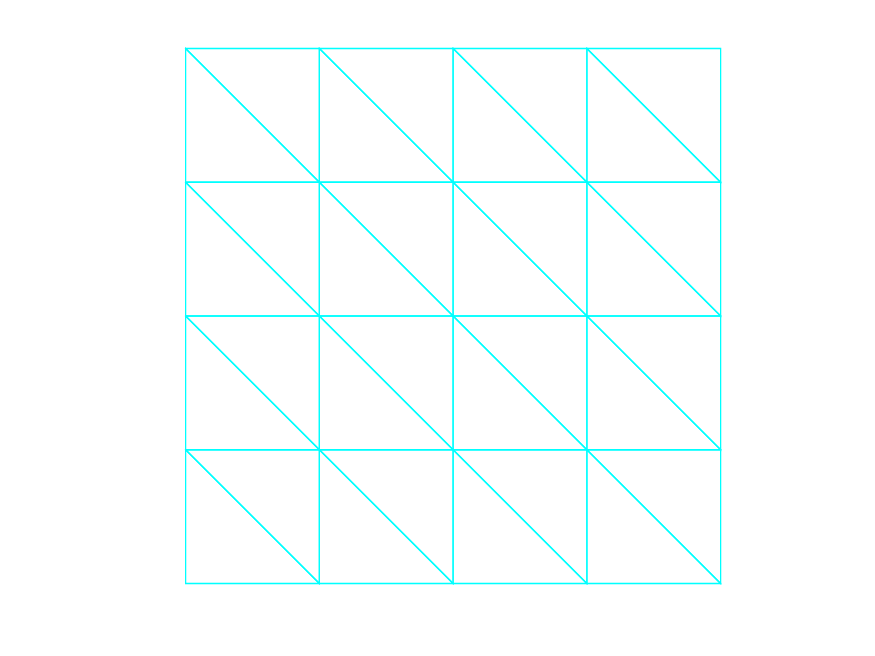}}
\quad 
\subfigure[] 
{\includegraphics[height=3.8cm,width=5cm]{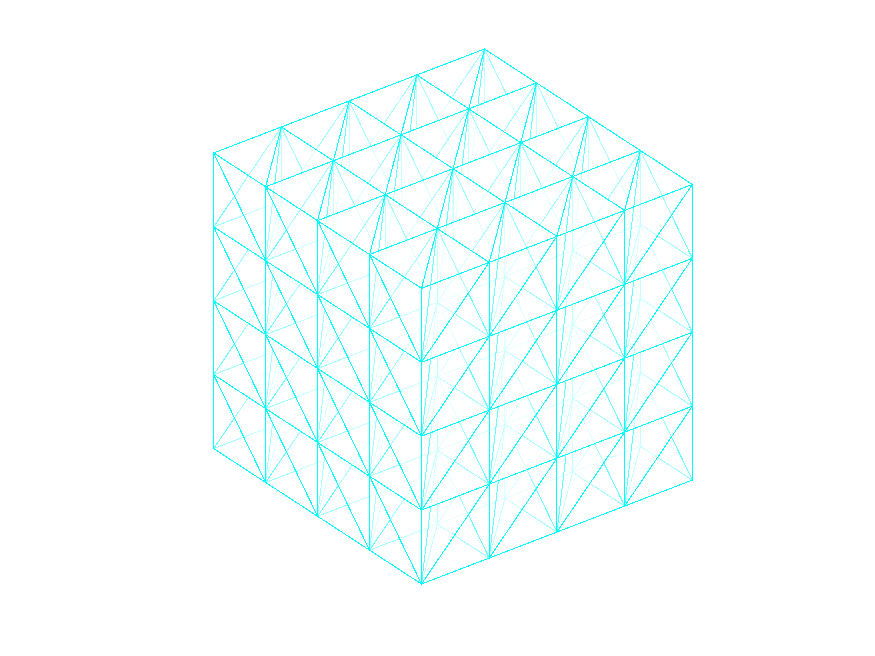}}
\caption{The meshes:  $(a) $ $4\times4$ mesh for $\Omega=[0,1]^2$; $(b)$ $4\times4\times4$ mesh for $\Omega=[0,1]^3$.}

\label{fig:mesh1}

\end{figure}

\begin{exam}[A 2D case] \label{exam1}
In the   model   \eqref{mhd1} we set
$$\Omega=[0,1]^2, H_a=N= R_m=1.$$
The exact solution $(\mathbf{u}, \mathbf{B}, T, p, r)$ is of the form
\begin{eqnarray*}
\left\{\begin {array}{lll}
u_1=-x^2(x-1)^2y(y-1)(2y-1),\\
u_2=y^2(y-1)^2x(x-1)(2x-1),\\
B_1=-x^2(x-1)^2y(y-1)(2y-1),\\
B_2=y^2(y-1)^2x(x-1)(2x-1),\\
p=x(x-1)(x-1/2)y(y-1)(y-1/2),\\
r=x(x-1)(x-1/2)y(y-1)(y-1/2),\\
T=x(x-1)y(y-1).\\
\end{array}\right.
\end{eqnarray*}
We compute  the   scheme (\ref{Tscheme0101*}) with $k=1,2$ on   $M\times M$ uniform regular triangular meshes (cf. Figure \ref{fig:mesh1}) with $M=4,8,16,32,64$.
Numerical results are listed in Tables \ref{linetableK=1(l=1)for1} to \ref{linetableK=2(l=2)for1}.
\end{exam}

\begin{exam}[A 3D case] \label{exam2}
In the   model  \eqref{mhd1} we set
$$\Omega=[0,1]^3, H_a= N= R_m=1.$$
The exact solution $(\mathbf{u}, \mathbf{B}, T, p, r)$ is of the form
\begin{eqnarray*}
\left\{\begin {array}{lll}
u_1=-1/20\pi(\sin(\pi x))^2\sin(\pi y)
\cos(\pi y)\sin(\pi z)
\cos(\pi z),\\
u_2=1/10\pi\sin(\pi x)\cos(\pi x)
(\sin(\pi y))^2\sin(\pi z)
\cos(\pi z),\\
u_3=-1/20\pi\sin(\pi x)\cos(\pi x)
\sin(\pi y)\cos(\pi y)
(\sin(\pi z))^2,\\
B_1=-1/20\pi(\sin(\pi x))^2\sin(\pi y)
\cos(\pi y)\sin(\pi z)
\cos(\pi z),\\
B_2=1/10\pi\sin(\pi x)\cos(\pi x)
(\sin(\pi y))^2\sin(\pi z)
\cos(\pi z),\\
B_3=-1/20\pi\sin(\pi x)\cos(\pi x)
\sin(\pi y)\cos(\pi y)
(\sin(\pi z))^2,\\
P=1/10\cos(\pi x)\cos(\pi y)\cos(\pi z),\\
r=1/10\sin(\pi x)\sin(\pi y)\sin(\pi z),\\
T=u_1+u_2+u_3.
\end{array}\right.
\end{eqnarray*}
We compute the   scheme (\ref{Tscheme0101*}) with $k=1,2$  on   $M\times M\times M$  uniform regular tetrahedral meshes (cf. Figure \ref{fig:mesh1}) with $M=4,8,16,32$.
Numerical results are listed in Tables \ref{linetableK=1(l=1)for2}  to \ref{linetableK=2(l=2)for2}.
\end{exam}

Tables \ref{linetableK=1(l=1)for1} to \ref{linetableK=2(l=2)for2} show the histories of convergence for the velocity $\mathbf{u}_{h}$, the magnetic field $\mathbf{B}_{h}$, the pressure $p_{h}$, 
the temperature $ T_h $
and the magnetic pseudo-pressure $r_h$. Results of $  ||\nabla\cdot\mathbf{u}_{h}||_{0,\infty,\Omega}$
and $ ||\nabla\cdot\mathbf{B}_{h}||_{0,\infty,\Omega}$
are also listed  to verify the divergence-free property. From the numerical results of the two examples, we have the following observations:
\begin{itemize}
\item The convergence rates of $\|\nabla\mathbf{u}-\nabla_h\mathbf{u}_{h}\|_0,$
$\|\nabla\times\mathbf{B}-\nabla_h\times\mathbf{B}_{h}\|_0, $
$\|\nabla T-\nabla_hT_{h}\|_0,$ $\|\bm{L}-\bm{L}_h\|_0$, $\|\bm{N}-\bm{N}_h\|_0$, $\|\bm{A}-\bm{A}_h\|_0$,
$ \|p-p_{h}\|_0$ and  $ \|r-r_h\|_0$
for the  HDG scheme with $k=1,2$ are of
$k^{th}$ orders, which are   consistent with the established theoretical
results in   Theorem \ref{estimates2}.

\item The convergence rates of $\|\mathbf{u}-\mathbf{u}_{h}\|_0$,
$\|\mathbf{B}-\mathbf{B}_{h}\|_0$,
and $\|T-T_{h}\|_0$ are of $(k+1)^{th}$ orders.

\item  
The discrete velocity and the discrete magnetic field  are globally divergence-free.
\end{itemize}

\begin{table}
\normalsize
\caption{Convergence results with  $k = 1$ for Example 1: $M\times M$ meshes }\label{linetableK=1(l=1)for1}
\centering
\footnotesize
{
\begin{tabular}{p{1.0cm}<{\centering}|p{1.8cm}<{\centering}|p{0.8cm}<{\centering}|p{1.8cm}<{\centering}|p{0.8cm}<{\centering}|p{1.8cm}<{\centering}|p{0.8cm}<{\centering}|p{1.8cm}<{\centering}}
\hline
\multirow{2}{*}{M}&
\multicolumn{2}{c|}{$\frac{\|\mathbf{u}-\mathbf{u}_{h}\|_0}{\|\mathbf{u\|_0}}$ }
&\multicolumn{2}{c|}{$\frac{\|\bm{L}-\bm{L}_h\|_0}{\|\bm{L}\|_0}$}
&\multicolumn{2}{c|}{$\frac{\|\nabla\mathbf{u}-\nabla_h\mathbf{u}_{h}\|_0}{\|\nabla\mathbf{u}\|_0}$}
&\multirow{2}{*}{$||\nabla\cdot\mathbf{u}_{h}||_{0,\infty,\Omega}$}\cr\cline{2-7}
&error&order&error&order&error&order\cr
\cline{1-8}

$4 $  & 5.6717e-01 &  -  & 5.1166e-01 &  -  
 &  6.4278e-01 & - 
&3.8026e-17\\
\hline
$8 $  & 1.5224e-01 &  1.89  & 2.7237e-01  & 0.91  &   3.2258e-01 &  0.99 
&  1.0408e-17\\
\hline
$16 $ &  3.9918e-02 &  1.93 &  1.3841e-01  & 0.97 &  1.6098e-01 &  1.00 
&  2.0239e-17\\
\hline
$32 $ &  1.0236e-02 &  1.96  & 6.9404e-02 &  0.99  &  8.0401e-02 & 1.00  
& 1.8400e-18\\
\hline
$64 $  & 2.5909e-03 &  1.98  & 3.4720e-02 &  1.00  &  4.0183e-02  &   1.00  
& 1.2251e-16\\
\hline

\end{tabular}
}

{
\begin{tabular}{p{1.0cm}<{\centering}|p{1.8cm}<{\centering}|p{0.8cm}<{\centering}|p{1.8cm}<{\centering}|p{0.8cm}<{\centering}|p{1.8cm}<{\centering}|p{0.8cm}<{\centering}|p{1.8cm}<{\centering}}
\hline
\multirow{2}{*}{M}&
\multicolumn{2}{c|}{$\frac{\|\mathbf{B}-\mathbf{B}_{h}\|_0}{\|\mathbf{B\|_0}}$ }
&\multicolumn{2}{c|}{$\frac{\|\bm{N}-\bm{N}_h\|_0}{\|\bm{N}\|_0}$}
&\multicolumn{2}{c|}{$\frac{\|\nabla\times\mathbf{B}-\nabla_h\times\mathbf{B}_{h}\|_0}{\|\nabla\times\mathbf{B}\|_0}$}
&\multirow{2}{*}{$||\nabla\cdot\mathbf{B}_{h}||_{0,\infty,\Omega}$}\cr\cline{2-7}
&error&order&error&order&error&order\cr
\cline{1-8}
$4 $  & 1.1606e-00  & - &  5.6441e-01  & - 
&8.2061e-01 &-
&  1.2266e-17\\
\hline
$8 $ &  2.9482e-01  & 1.97 &  2.7675e-01  & 1.02  &3.7929e-01  &1.11
& 4.9065e-18\\
\hline
$16 $ &  7.4064e-02  & 1.99 &  1.3198e-01 &  1.06  &1.6468e-01 &1.20
&  1.2880e-17\\
\hline
$32 $  & 1.8525e-02  & 1.99 &  6.4624e-02  & 1.03 &7.5900e-02 &1.11
&  1.4812e-16\\
\hline
$64 $  & 4.6309e-03  & 2.00 &  3.2115e-02 &  1.00 &3.6869e-02 &1.04
& 6.5426e-16\\
\hline

\end{tabular}
}
{
\begin{tabular}{p{0.5cm}<{\centering}|p{1.5cm}<{\centering}|p{0.5cm}<{\centering}|p{1.5cm}<{\centering}|p{0.5cm}<{\centering}
|p{1.5cm}<{\centering}|p{0.5cm}<{\centering}|p{1.5cm}<{\centering}|p{0.5cm}<{\centering}
|p{1.5cm}<{\centering}|p{0.5cm}<{\centering}}
\hline
\multirow{2}{*}{M}&
\multicolumn{2}{c|}{$\frac{\|T-T_{h}\|_0}{\|T\|_0}$  }
&\multicolumn{2}{c|}{$\frac{\|\bm{A}-\bm{A}_h\|_0}{\|\bm{A}\|_0}$}
&\multicolumn{2}{c|}{$\frac{\|\nabla T-\nabla_h T_{h}\|_0}{\|\nabla T\|_0}$}
&\multicolumn{2}{c|}{$\frac{\|p-p_{h}\|_0}{\|p\|_0}$}
&\multicolumn{2}{c}{$\frac{\|r-r_h\|_0}{\|r\|_0}$}\cr\cline{2-11}
&error&order&error&order&error&order&error&order&error&order\cr
\cline{1-11}
$4 $   &2.0610e-01  & -  & 2.9184e-01 &  -
&3.3214e-01  &-
&  3.8537e-00 &  - &  4.7146e-00 &  -\\
\hline
$8 $  & 5.2226e-02  & 1.98 &  1.4848e-01  & 0.97 &1.6220e-01 &1.03
&   2.1460e-00  & 0.84 &  2.8776e-00  &  0.71\\
\hline
$16 $ & 1.3101e-02  & 1.99  & 7.4559e-02 &  0.99 &8.0451e-02 &1.01
&   9.6509e-01  & 1.15  & 1.4532e-01 &   0.98\\
\hline
$32 $ & 3.2779e-03 &  1.99   &3.7319e-02  & 0.99 &4.0135e-02 &1.00
&   4.1801e-01 &  1.20 &  7.2538e-01  &  1.00\\
\hline
$64 $  & 8.1963e-04 &  2.00 &  1.8664e-02 &  1.00  &2.0056e-02 &1.00
&    1.9086e-01  & 1.13  & 3.6249e-02  &  1.00\\
\hline

\end{tabular}
}

\end{table}

\begin{table}[H]
\normalsize
\caption{Convergence results with  $k = 2$ for Example 1: $M\times M$ meshes }\label{linetableK=2(l=2)for1}
\centering
\footnotesize
{
\begin{tabular}{p{1.0cm}<{\centering}|p{1.8cm}<{\centering}|p{0.8cm}<{\centering}|p{1.8cm}<{\centering}|p{0.8cm}<{\centering}|p{1.8cm}<{\centering}|p{0.8cm}<{\centering}|p{1.8cm}<{\centering}}
\hline
\multirow{2}{*}{M}&
\multicolumn{2}{c|}{$\frac{\|\mathbf{u}-\mathbf{u}_{h}\|_0}{\|\mathbf{u\|_0}}$ }
&\multicolumn{2}{c|}{$\frac{\|\bm{L}-\bm{L}_{h}\|_0}{\|\bm{L}\|_0}$}
&\multicolumn{2}{c|}{$\frac{\|\nabla\mathbf{u}-\nabla_h\mathbf{u}_{h}\|_0}{\|\nabla\mathbf{u}\|_0}$}
&\multirow{2}{*}{$||\nabla\cdot\mathbf{u}_{h}||_{0,\infty,\Omega}$}\cr\cline{2-7}
&error&order&error&order&error&order\cr
\cline{1-8}

$4 $  & 5.6888e-02&   - &  1.3014e-01 &  - 
&2.4354e-01 &-
&  3.7691e-17\\
\hline
$8 $ &  7.3768e-03 &  2.94  & 3.4894e-02 &  1.89  &6.2798e-02&1.95
& 1.3042e-16 \\
\hline
$16 $  & 9.3177e-04 &  2.98  & 8.9496e-03 &  1.96 &1.5637e-02&2.00
& 4.1064e-17 \\
\hline
$32 $ &  1.1713e-04 &  2.99  & 2.2591e-03&   1.98  &3.8806e-03&2.01
 & 1.6195e-16  \\
\hline
$64 $  & 1.4693e-05  & 2.99 &  5.6704e-04  & 1.99 &9.6538e-04&2.00
 &  1.0799e-15  \\
\hline

\end{tabular}
}

{
\begin{tabular}{p{1.0cm}<{\centering}|p{1.8cm}<{\centering}|p{0.8cm}<{\centering}|p{1.8cm}<{\centering}|p{0.8cm}<{\centering}|p{1.8cm}<{\centering}|p{0.8cm}<{\centering}|p{1.8cm}<{\centering}}
\hline
\multirow{2}{*}{M}&
\multicolumn{2}{c|}{$\frac{\|\mathbf{B}-\mathbf{B}_{h}\|_0}{\|\mathbf{B\|_0}}$ }
&\multicolumn{2}{c|}{$\frac{\|\bm{N}-\bm{N}_{h}\|_0}{\|\bm{N}\|_0}$}
&\multicolumn{2}{c|}{$\frac{\|\nabla\times\mathbf{B}-\nabla_h\times\mathbf{B}_{h}\|_0}{\|\nabla\times\mathbf{B}\|_0}$}
&\multirow{2}{*}{$||\nabla\cdot\mathbf{B}_{h}||_{0,\infty,\Omega}$}\cr\cline{2-7}
&error&order&error&order&error&order\cr
\cline{1-8}
$4 $ & 1.9191e-01  & - &  1.2959e-01 &  - 
&4.3715e-01 &-
&  2.7346e-17\\
\hline
$8 $  & 3.2161e-02 &  2.57  & 2.8450e-02  & 2.18 &1.3871e-01 &1.66
&   4.8258e-17\\
\hline
$16 $ &  4.5933e-03  & 2.80  &7.0130e-03 &  2.02  & 3.8351e-02 &1.85
&   5.4459e-17\\
\hline
$32 $ &  6.0891e-04&   2.91 &  1.7591e-03  & 1.99   & 9.9904e-03  &1.94
& 2.9157e-16\\
\hline
$64 $ &  7.8185e-05 &  2.96  & 4.4098e-04  & 1.99  &2.5432e-03 &1.97
&    1.0700e-15\\
\hline

\end{tabular}
}
{
\begin{tabular}{p{0.5cm}<{\centering}|p{1.5cm}<{\centering}|p{0.5cm}<{\centering}|p{1.5cm}<{\centering}|p{0.5cm}<{\centering}
		|p{1.5cm}<{\centering}|p{0.5cm}<{\centering}|p{1.5cm}<{\centering}|p{0.5cm}<{\centering}
		|p{1.5cm}<{\centering}|p{0.5cm}<{\centering}}
	\hline
	\multirow{2}{*}{M}&
	\multicolumn{2}{c|}{$\frac{\|T-T_{h}\|_0}{\|T\|_0}$  }
	&\multicolumn{2}{c|}{$\frac{\|\bm{A}-\bm{A}_h\|_0}{\|\bm{A}\|_0}$}
	&\multicolumn{2}{c|}{$\frac{\|\nabla T-\nabla_h T_{h}\|_0}{\|\nabla T\|_0}$}
	&\multicolumn{2}{c|}{$\frac{\|p-p_{h}\|_0}{\|p\|_0}$}
	&\multicolumn{2}{c}{$\frac{\|r-r_h\|_0}{\|r\|_0}$}\cr\cline{2-11}
	&error&order&error&order&error&order&error&order&error&order\cr
	\cline{1-11}
$4 $  & 1.2679e-02  & -&   4.1549e-02  & -
&8.3060e-02   &-
& 1.0551e-00 &  -  & 1.3377e-00 &  -\\
\hline
$8 $  & 1.5529e-03   &3.02 &  1.0588e-02 &  1.97  &2.0600e-02   &2.01
&  1.9968e-01 &  2.40  & 3.3459e-01 &  1.99\\
\hline
$16 $  & 1.9229e-04  & 3.01 &  2.6637e-03 &  1.99  &5.1403e-03  &2.00
&  4.3435e-02  & 2.20   &8.5809e-02  & 1.96\\
\hline
$32 $  & 2.3940e-05  & 3.00 &  6.6753e-04 &  1.99    &1.2851e-03 &2.00
&  1.0263e-02 &  2.08 &  2.1581e-02  & 1.99\\
\hline
$64 $  & 2.9871e-06  & 3.00 &  1.6705e-04 &  1.99   &3.2137e-04  &1.99
& 2.5002e-03 &  2.03 &  5.3937e-03 &  2.00\\
\hline

\end{tabular}
}

\end{table}


\begin{table}[H]\label{linetable7}
\normalsize
\caption{Convergence results with  $k = 1$ for Example 2: $M\times M\times M$ meshes }\label{linetableK=1(l=1)for2}
\centering
\footnotesize
{
\begin{tabular}{p{2.0cm}<{\centering}|p{1.9cm}<{\centering}|p{1.1cm}<{\centering}|p{1.9cm}<{\centering}
|p{1.1cm}<{\centering}|p{1.9cm}<{\centering}}
\hline
\multirow{2}{*}{M}&
\multicolumn{2}{c|}{$\frac{\|\mathbf{u}-\mathbf{u}_{h}\|_0}{\|\mathbf{u\|_0}}$ }&\multicolumn{2}{c|}{$\frac{\|\nabla\mathbf{u}-\nabla_h\mathbf{u}_{h}\|_0}{\|\nabla\mathbf{u}\|_0}$}
&\multirow{2}{*}{$||\nabla\cdot\mathbf{u}_{h}||_{0,\infty,\Omega}$}\cr\cline{2-5}
&error&order&error&order\cr
\cline{1-6}

$4$ & 8.4484e-01 &  - &  1.3080e+00  & -&  2.2644e-15   \\
\hline
$8$ &  2.3984e-01 &  1.82 &  6.5791e-01  & 1.00 &  4.0873e-17 \\
\hline
$16$  & 5.8894e-02 &  2.03  & 2.9664e-01 & 1.15  & 5.7086e-17 \\
\hline
$32$ &  1.4583e-02 &  2.01 &  1.4088e-01  & 1.07 &  1.9761e-15\\
\hline

\end{tabular}
}

{
\begin{tabular}{p{2.0cm}<{\centering}|p{1.9cm}<{\centering}|p{1.1cm}<{\centering}|p{1.9cm}<{\centering}
|p{1.1cm}<{\centering}|p{1.9cm}<{\centering}}
\hline
\multirow{2}{*}{M}&
\multicolumn{2}{c|}{$\frac{\|\mathbf{B}-\mathbf{B}_{h}\|_0}{\|\mathbf{B\|_0}}$ }&\multicolumn{2}{c|}{$\frac{\|\nabla\times\mathbf{B}-\nabla_h\times\mathbf{B}_{h}\|_0}{\|\nabla\times\mathbf{B}\|_0}$}
&\multirow{2}{*}{$||\nabla\cdot\mathbf{B}_{h}||_{0,\infty,\Omega}$}\cr\cline{2-5}
&error&order&error&order\cr
\cline{1-6}
$4$  & 1.3036e+00  & - &  1.7242e+00 &  - &  3.9403e-14\\
\hline
$8$  & 3.1903e-01  & 2.03 &  8.8334e-01 &  0.96  & 1.2132e-16\\
\hline
$16$  & 7.7396e-02 &  2.04  & 4.2322e-01 &  1.06  & 5.9822e-16\\
\hline
$32$  & 1.9814e-02  & 1.96  & 2.2287e-01  & 0.92 &  1.5168e-14\\
\hline

\end{tabular}
}
{
\begin{tabular}{p{0.5cm}<{\centering}|p{1.5cm}<{\centering}|p{0.5cm}<{\centering}
|p{1.5cm}<{\centering}|p{0.5cm}<{\centering}|p{1.5cm}<{\centering}|p{0.5cm}<{\centering}
|p{1.5cm}<{\centering}|p{0.5cm}<{\centering}}
\hline
\multirow{2}{*}{M}&
\multicolumn{2}{c|}{$\frac{\|T-T_{h}\|_0}{\|T\|_0}$  }&\multicolumn{2}{c|}{$\frac{\|\nabla T-\nabla_h T_{h}\|_0}{\|\nabla T\|_0}$}
&\multicolumn{2}{c|}{$\frac{\|p-p_{h}\|_0}{\|p\|_0}$}
&\multicolumn{2}{c}{$\frac{\|r-r_h\|_0}{\|r\|_0}$}\cr\cline{2-9}
&error&order&error&order&error&order&error&order\cr
\cline{1-9}
$4$  & 1.3145e-01 &  - &  2.6572e-01 &  -	& 3.0341e+00  & - & 2.6418e+00 &  -\\
\hline
$8$  & 3.3598e-02 &  1.96 & 1.3456e-01 &  0.98&  1.5205e+00 &  0.99 &1.4048e+00  & 0.91\\
\hline
$16$ &  8.1510e-03 &  2.04 &  6.5097e-02 &  1.04	&  7.8994e-01  & 0.94 &6.6614e-01 &  1.07\\
\hline
$32$ & 1.9265e-03 &  2.08  & 3.3348e-02 &  0.96 & 3.9773e-01 &  0.98& 3.3533e-01 &  0.99\\
\hline

\end{tabular}
}

\end{table}

\begin{table}[H]\label{linetable7}
\normalsize
\caption{Convergence results with  $k = 2$ for Example 2: $M\times M\times M$ meshes }\label{linetableK=2(l=2)for2}
\centering
\footnotesize
{
\begin{tabular}{p{2.0cm}<{\centering}|p{1.9cm}<{\centering}|p{1.1cm}<{\centering}|p{1.9cm}<{\centering}
|p{1.1cm}<{\centering}|p{1.9cm}<{\centering}}
\hline
\multirow{2}{*}{M}&
\multicolumn{2}{c|}{$\frac{\|\mathbf{u}-\mathbf{u}_{h}\|_0}{\|\mathbf{u\|_0}}$ }&\multicolumn{2}{c|}{$\frac{\|\nabla\mathbf{u}-\nabla_h\mathbf{u}_{h}\|_0}{\|\nabla\mathbf{u}\|_0}$}
&\multirow{2}{*}{$||\nabla\cdot\mathbf{u}_{h}||_{0,\infty,\Omega}$}\cr\cline{2-5}
&error&order&error&order\cr
\cline{1-6}

$4$ & 9.2796e-01 &  - &  1.3621e+00  & -&  1.5044e-13   \\
\hline
$8$ &  1.2375e-01 &  2.90 &  3.5075e-01  & 1.95 & 1.1446e-13 \\
\hline
$16$  & 1.5849e-02 &  2.96  & 8.5451e-02 & 2.03  & 4.9513e-15 \\
\hline
$32$ &  1.9771e-03 &  3.00 &  2.1015e-02  & 2.02 &  2.9984e-15\\
\hline

\end{tabular}
}

{
\begin{tabular}{p{2.0cm}<{\centering}|p{1.9cm}<{\centering}|p{1.1cm}<{\centering}|p{1.9cm}<{\centering}
|p{1.1cm}<{\centering}|p{1.9cm}<{\centering}}
\hline
\multirow{2}{*}{M}&
\multicolumn{2}{c|}{$\frac{\|\mathbf{B}-\mathbf{B}_{h}\|_0}{\|\mathbf{B\|_0}}$ }&\multicolumn{2}{c|}{$\frac{\|\nabla\times\mathbf{B}-\nabla_h\times\mathbf{B}_{h}\|_0}{\|\nabla\times\mathbf{B}\|_0}$}
&\multirow{2}{*}{$||\nabla\cdot\mathbf{B}_{h}||_{0,\infty,\Omega}$}\cr\cline{2-5}
&error&order&error&order\cr
\cline{1-6}
$4$  & 5.5091e-01  & - &  1.8659e+00 &  - &  2.3349e-13\\
\hline
$8$  & 6.3612e-02  & 3.05 &  4.5817e-01 &  2.02  & 8.4842e-14\\
\hline
$16$  & 7.7808e-03 &  3.00  & 1.2473e-01 &  1.88  & 4.7358e-15\\
\hline
$32$  & 1.0141e-03  & 2.93 & 3.1923e-02  & 1.97 &  2.3222e-15\\
\hline

\end{tabular}
}
{
\begin{tabular}{p{0.5cm}<{\centering}|p{1.5cm}<{\centering}|p{0.5cm}<{\centering}
|p{1.5cm}<{\centering}|p{0.5cm}<{\centering}|p{1.5cm}<{\centering}|p{0.5cm}<{\centering}
|p{1.5cm}<{\centering}|p{0.5cm}<{\centering}}
\hline
\multirow{2}{*}{M}&
\multicolumn{2}{c|}{$\frac{\|T-T_{h}\|_0}{\|T\|_0}$  }&\multicolumn{2}{c|}{$\frac{\|\nabla T-\nabla_h T_{h}\|_0}{\|\nabla T\|_0}$}
&\multicolumn{2}{c|}{$\frac{\|p-p_{h}\|_0}{\|p\|_0}$}
&\multicolumn{2}{c}{$\frac{\|r-r_h\|_0}{\|r\|_0}$}\cr\cline{2-9}
&error&order&error&order&error&order&error&order\cr
\cline{1-9}
$4$  & 8.6635e-02 &  - &  2.1682e-01 &  -	& 1.7040e+00  & - & 1.0295e+00 &  -\\
\hline
$8$  & 1.2193e-02 &  2.82 & 5.1053e-02 &  2.08&  4.2772e-01 &  1.99 &2.4570e-01  & 2.06\\
\hline
$16$ &  1.5038e-03 &  3.01&  1.2736e-02 &  2.00	&  1.0150e-01  & 2.07 &6.3467e-02 &  1.95\\
\hline
$32$ & 1.8960e-04 &  2.98  & 3.2012e-03 &  1.99 & 2.5328e-02 &  2.00& 1.5634e-02 &  2.02\\
\hline

\end{tabular}
}

\end{table}

\section{Conclusions}

In this paper, we have developed an HDG     method of arbitrary order
for the steady thermally   coupled incompressible
Magnetohydrodynamics flow.  The well-posedness of the discrete scheme has been established.
The method yields globally divergence-free  approximations of velocity and magnetic field, and is of optimal order convergence for the velocity, the magnetic field, the pressure, the magnetic pseudo-pressure, and temperature approximations.
Numerical experiments have verified the theoretical results.

\end{document}